\documentclass{amsart}[12pt, article]
\usepackage{amsmath}
\usepackage{mathrsfs}
\usepackage{amsfonts}
\usepackage{txfonts}
\usepackage{amssymb}

\newcommand\NN{{\mathbb N}}
\newcommand\RR{{{\mathbb R}}}

\def\SS {\mathbb{S}}

\newcommand\cA{{\mathcal A}}
\newcommand\cB{{\mathcal B}}
\newcommand\cD{{\mathcal D}}
\newcommand\cE{{\mathcal E}}
\newcommand\cH{{\mathcal H}}
\newcommand\cF{{\mathcal F}}
\newcommand\cL{{\mathcal L}}
\newcommand\cS{{\mathcal S}}
\newcommand\cT{{\mathcal T}}

\newcommand\cN{{\mathcal N}}
\newcommand\cX{{\mathcal X}}

\newcommand\pP{{\bf P}}
\newcommand\iI{{\bf I}}
\newcommand\p{{\partial}}
\newcommand\pa{{\partial}}
\newcommand\ra{{\rangle}}
\newcommand\la{{\langle}}

\newtheorem{theo}{Theorem}[section]
\newtheorem{lemm}[theo]{Lemma}

\newtheorem{coro}[theo]{Corollary}
\newtheorem{prop}[theo]{Proposition}
\newtheorem{rema}[theo]{Remark}

\begin{document}

\title[Global existence for soft potential]
{ Boltzmann equation without angular cutoff in the whole space: I, Global existence for soft potential}
\author{R. Alexandre}
\address{R. Alexandre,
Department of Mathematics, Shanghai Jiao Tong University
\newline\indent
Shanghai, 200240, P. R. China
\newline\indent
and \newline\indent
IRENAV Research Institute, French Naval Academy
Brest-Lanv\'eoc 29290, France
}
\email{radjesvarane.alexandre@ecole-navale.fr}
\author{Y. Morimoto }
\address{Y. Morimoto, Graduate School of Human and Environmental Studies,
Kyoto University
\newline\indent
Kyoto, 606-8501, Japan} \email{morimoto@math.h.kyoto-u.ac.jp}
\author{S. Ukai}
\address{S. Ukai, 17-26 Iwasaki-cho, Hodogaya-ku, Yokohama 240-0015, Japan}
\email{ukai@kurims.kyoto-u.ac.jp}
\author{C.-J. Xu}
\address{C.-J. Xu, School of Mathematics, Wuhan University 430072,
Wuhan, P. R. China
\newline\indent
and \newline\indent
Universit\'e de Rouen, UMR 6085-CNRS,
Math\'ematiques
\newline\indent
Avenue de l'Universit\'e,\,\, BP.12, 76801 Saint
Etienne du Rouvray, France } \email{Chao-Jiang.Xu@univ-rouen.fr}
\author{T. Yang}
\address{T. Yang, Department of mathematics, City University of Hong Kong,
Hong Kong, P. R. China
\newline\indent
and \newline\indent
School of Mathematics, Wuhan University 430072,
Wuhan, P. R. China} \email{matyang@cityu.edu.hk}

\subjclass[2000]{35A05, 35B65, 35D10, 35H20, 76P05, 84C40}


\keywords{Boltzmann equation, coercivity estimate, non-cutoff cross sections,
global existence, non-isotropic norm, soft potential.}

\begin{abstract}
It is known that the singularity in the non-cutoff cross-section of the
Boltzmann equation leads to the gain of regularity and gain of weight in the
velocity variable. By defining and analyzing a non-isotropy norm
which precisely captures the dissipation in the linearized collision
operator, we first give a new and precise coercivity estimate for the non-cutoff
Boltzmann equation for general physical cross sections. Then the Cauchy problem for the
Boltzmann equation is considered in the framework of small perturbation of
an equilibrium state. In this part, for the soft potential case
in the sense that there is no positive power gain of weight in the
coercivity estimate on the linearized operator, we derive some new functional estimates
on the nonlinear collision operator. Together with the coercivity estimates, we prove the global existence of classical solutions for the Boltzmann equation
in weighted Sobolev spaces.
\end{abstract}

\maketitle

\tableofcontents

\section{Introduction}\label{part2-section1}

This is the first part of a series of papers related to the inhomogeneous Boltzmann equation without angular cut-off, in the whole space and for general physical cross-sections. This global project is a natural continuation of our previous study \cite{amuxy3} which was specialized to Maxwellian type cross sections.

In this part, we first establish an essential coercivity estimate of the linearized collision
operator, in the framework of general cross sections. As shown in \cite{amuxy3,amuxy3b} for the special Maxwellian case, this estimate will play an important role for the related Cauchy problem.

Based on this estimation, together with Part II \cite{amuxy4-3},
we will prove the global existence of classical non-negative solutions
to the Boltzmann equation without angular cutoff, for the soft and hard potentials respectively, so that we are able to cover a general physical setting. Finally, in the paper \cite{amuxy4-4}, we will study the qualitative properties of solutions, that include full regularity, non-negativity, uniqueness and convergence rates  to the equilibrium. This series of works establish a  satisfactory theory on the well-posedness
and full regularity of classical solutions.

In our presentation, we consider the problem in the physical case with dimension $3$. However, our results hold true for any dimension bigger than $2$.

Consider
\begin{equation}\label{part2-4-2-1.1}
f_t+v\cdot\nabla_x f=Q(f, f)\,,\,\,\,\,\,\,f|_{t=0}=f_0.
\end{equation}
Here, $f= f(t,x,v)$ is the density distribution function
of particles, having position
$x\in \RR^3$ and velocity $v\in \RR^3$ at time $t$.
The right hand side of (\ref{part2-4-2-1.1}) is the
Boltzmann bilinear collision operator, which is given in the classical $\sigma -$representation by
\[
Q(g, f)=\int_{\RR^3}\int_{\mathbb S^{2}}B\left({v-v_*},\sigma
\right)
 \left\{g'_* f'-g_*f\right\}d\sigma dv_*\,,
\]
where $f'_*=f(t,x,v'_*), f'=f(t,x,v'), f_*=f(t,x,v_*), f=f(t,x,v)$, and for
$\sigma\in \mathbb S^{2}$,
$$
v'=\frac{v+v_*}{2}+\frac{|v-v_*|}{2}\sigma,\,\,\, v'_*
=\frac{v+v_*}{2}-\frac{|v-v_*|}{2}\sigma,\,
$$
which gives the relation between the post and pre collisional
velocities that follow from the conservation of momentum and kinetic energy.

For monatomic gas,  the non-negative cross section
 $B(z, \sigma)$ depends only on $|z|$ and the scalar product
$\frac{z}{|z|}\,\cdot\, \sigma$. As in our previous works, we  assume
that it takes  the form
\begin{equation}\label{part2-hyp-2}
B(v-v_*, \cos \theta)=\Phi (|v-v_*|) b(\cos \theta),\,\,\,\,\,
\cos \theta=\frac{v-v_*}{|v-v_*|} \, \cdot\,\sigma\, , \,\,\,
0\leq\theta\leq\frac{\pi}{2},
\end{equation}
in which it contains a kinetic factor given by
\begin{equation*}
\Phi (|v-v_*|) =\Phi_\gamma(|v-v_*|)= |v-v_*|^{\gamma},
\end{equation*}
and a factor related to the collision angle containing a singularity,
\begin{equation*}
\begin{array}{l}
 b(\cos \theta)\approx K\theta^{-2-2s} \
\
 \mbox{when} \ \ \theta\rightarrow 0+,
\end{array}
\end{equation*}
for some constant $K>0$.

An important example of this singular cross section is the inverse power law potential $\rho^{-r}$ with $r>1$, $\rho$
being the distance between two interacting particles, in which $s=\frac{1}{r}\in ]0,1[$ and $\gamma=1-4s\in ]-3,1[$, cf. \cite{Ce88}.

In the theory on the non-cutoff Boltzmann equation,  the sign of $\gamma +2s$ plays a crucial role. Hence, from
now on,
the case when $\gamma +2 s\leq 0$ is referred to the {\it non-cutoff soft potential}, while the case
 $\gamma +2 s> 0$ to  the {\it non-cutoff hard potential}.  Note that this is different from the
 traditional classification on the index for the inverse power law.

In our present series of works, the well-posedness theory established applies to the general cross-sections
with $\gamma>-3$ and $0<s<1$, that includes the inverse power law as a special example.
Note  that  $\gamma >-3$ is needed for
the Boltzmann operator to be well-posed, cf. \cite{villani2}.


Being concerned with a close to equilibrium framework, as in \cite{amuxy3}, the setting of the problem can be formulated as follows.
First of all, without loss of generality, consider the perturbation around a
normalized Maxwellian distribution
\[
\mu (v)=(2\pi )^{-\frac{3}{2}}e^{-\frac{|v|^2}{2\,\,}},
\]
by setting $f=\mu+\sqrt{\mu}  g$. Since  $Q(\mu,
\mu)=0$, we have
\[
Q(\mu+\sqrt{\mu}  g,\, \mu+\sqrt{\mu}\,  g)=Q(\mu,\, \sqrt{\mu}\,g)+
Q(\sqrt{\mu}\,g,\, \mu)+Q(\sqrt{\mu}\,g,\, \sqrt{\mu}\,g).
\]
Denote
\begin{equation*}
\Gamma(g,\, h)=\mu^{-1/2}Q(\sqrt{\mu}\,g,\, \sqrt{\mu}\,h).
\end{equation*}
Then the linearized Boltzmann operator takes the form
\begin{equation*}
\cL g=\cL_1\, g+\cL_2\, g=-\Gamma(\sqrt{\mu}\,,\, g)- \Gamma(g,\,
\sqrt{\mu}\,).
\end{equation*}

Now the original problem (\ref{part2-4-2-1.1}) is reduced to the Cauchy problem
for the perturbation
$g$
\begin{equation}\label{part2-4-2-1.5}
\left\{\begin{array}{l} g_t + v\cdot\nabla_x g + \cL g=\Gamma(g,\,
g), \,\,\, t>0\,,\\
g|_{t=0}=g_0.\end{array}\right.
\end{equation}

This close to equilibrium framework is classical for the Boltzmann equation with angular cutoff, but much less is known for
the Boltzmann equation without angular cutoff, though the spectrum
of the linearized operator without angular cut-off
 was analyzed a long time ago by Pao in \cite{Pao}.

However, since the
late 1990s, the regularizing effect on the solution, produced by the singularity of the cross-section, has become  reachable by rigorous analysis. Let us mention the systematic work on the entropy dissipation method initiated
by Alexandre \cite{alex-two-maxwellian} and developed  firstly by Lions \cite{Lions98}, and then by many others,
cf \cite{al-1, Vi99, villani2} and references therein. Since then, various works have been done
on deriving the coercivity estimates in different settings and in different norms
for different purposes. In particular, this kind of coercivity estimates has displayed some non-isotropic property in the very loose sense that, on one hand one gets a gain of the regularity in Sobolev norm of fractional order; and on the other hand, one also get a gain the moment to some fractional power in the velocity variable,   cf. \cite{alex-review,al-1,
amuxy-nonlinear-b,amuxy-nonlinear-3,amuxy3,desv-wen1,gr-st,gr-st-1,HMUY,mouhot, mouhot-strain,ukai-2,Vi99,villani2} and references therein.
Furthermore, these coercivity estimates
have been proven to be very useful in getting the global existence and gain of full regularity
in all variables for the Boltzmann equation without angular cutoff, as shown in our previous work \cite{amuxy3}. For details
about the recent progress in some of the directions mentionned previously, readers are referred to the survey paper
by Alexandre, \cite{alex-review}.

Since the coercivity estimate plays an important role in the study on
the angular non-cutoff Boltzmann equation, such estimate in
terms of the indices $\gamma$ and $s$,
 has been
pursued by many people. One of the purposes of this paper is to present a precise
estimate that gives the essential properties of this singular behavior, that will be stated in the next theorem. Let us note that this result is
proved in a general setting and it
  improves on previous results, such as those obtained in \cite{amuxy-nonlinear-b,
amuxy-nonlinear-3,amuxy3,mouhot,mouhot-strain}.
And this estimate will be used herein and in our  papers \cite{amuxy4-3, amuxy4-4} on the global existence
in the hard potential case, and qualitative study of solutions.

To derive the desired coercivity estimate, we generalize the non-isotropic norm introduced in \cite{amuxy3} as
\begin{align*}
||| g|||^2&= \iiint \Phi (|v-v_*|)b(\cos\theta) \mu_*\, \big(g'-g\,\big)^2\,\\
 &\,\,\,\,\,\,\,\,+\iiint \Phi (|v-v_*|)b(\cos\theta)g^2_* \big(\sqrt{\mu'}\,\, - \sqrt{\mu}\,\,
\big)^2\nonumber\,,
\end{align*}
where the integration  is over
$\RR^3_v\times\RR^3_{v_\ast}\times\SS^2_\sigma$. Note that it is a
norm with respect to the velocity variable $v\in\RR^3$ only. We can compare this non-isotropic norm with classical weighted Sobolev norms, see precisely Proposition \ref{part1-prop2.1}.

The introduction of this norm was motivated by the study on
the Landau equation which can be viewed as the grazing limit of the Boltzmann equation.
It is known that for the Landau equation, see for example \cite{guo}, that the essential norm in order to capture the dissipation
of the linearized Landau operator can be defined just as the Dirichlet form of the
linearized operator. By doing so, a norm can be well defined without loss of any
dissipative information in the operator and this can be done directly for the Landau
equation mainly because the corresponding Landau operator is a differential operator. However, for the
Boltzmann equation without angular cutoff, the collision operator is a singular
integral operator so that a direct analog is not obvious or feasible.
Therefore,
in the first part of this paper, we analyze the properties of the non-isotropic norm and obtain the precise coercivity
estimate of the linearized collision operator. At this point, let us mention the different approach undertaken by Gressman-Strain \cite{gr-st,gr-st-1}.

We shall use the following standard weighted Sobolev space defined, for $k, \ell \in \RR$, as
$$
H^k_\ell = H^k_\ell (\RR^3_v ) = \lbrace f\in \cS ' (\RR^3_v ) ; \ W_\ell f \in H^k (\RR^3_v ) \rbrace
$$
and
$$
H^k_\ell (\RR^6_{x, v} ) = \lbrace f\in \cS ' (\RR^6_{x, v} ) ; \ W_\ell f \in H^k (\RR^6_{x, v} ) \rbrace
$$
where $W_\ell (v) = \langle v\rangle^\ell = (1+ | v|^2 )^{\ell/2}$ is always the weight for $v$ variables. Herein, $(\cdot, \cdot)_{L^2} = (\cdot,\cdot)_{L^2 (\RR^3_v)}$ denotes the usual scalar product in $L^2=L^2 (\RR^3)$ for $v$ variables. Recall that $L^2_\ell=H^0_\ell$.

We shall use also in the following two different Sobolev spaces, one  with $x$-derivatives only,  another one with $x,v$
derivatives and weight in the  velocity variable $v$.
For $k\in\NN,\,\,\ell\in\RR$, let
\begin{align*}
{\cH}^k_\ell(\RR^6)&=\Big\{f\in\cS'(\RR^6_{x, v})\, ;\,\,
\|f\|^2_{{ \cH}^k_\ell(\RR^6)}=\sum_{|\alpha|+|\beta|\leq N}\|W_{\ell-|\beta|}
\partial^\alpha_\beta f\|^2_{L^2(\RR^6)}<+\infty\,\Big\}\, ,
\\
{\tilde \cH}^k_\ell(\RR^6)&=\Big\{f\in\cS'(\RR^6_{x, v})\, ;\,\,
\|f\|^2_{{\tilde \cH}^k_\ell(\RR^6)}=\sum_{|\alpha|+|\beta|\leq N}\|\tilde W_{\ell-|\beta|}
\partial^\alpha_\beta f\|^2_{L^2(\RR^6)}<+\infty\,\Big\}\, ,
\end{align*}
where $\tilde W
_\ell=(1+|v|^2)^{| s+\gamma/2|\,\ell/2}$.

We recall that the linearized  operator $\cL$ has the following null space,
which is spanned by the set of collision invariants:
\begin{equation*}
\mathcal{N}=\mbox{Span}\left\{\sqrt \mu\,,\,
v_1\sqrt\mu\,,\,v_2\sqrt\mu\,,\,v_3\sqrt\mu\,,\,
|v|^2\sqrt\mu\,\,\right\},
\end{equation*}
that is, $\Big(\cL g,\, g\Big)_{L^2(\RR^3_v)}=0$
if and only if $g\in {\mathcal N}$.

\begin{theo}\label{part1-theo2}
Assume that the cross-section satisfies \eqref{part2-hyp-2} with  $0<s<1$ and $\gamma>-3$. Then there exist two generic constants $C_1, C_2>0$ such that for any suitable function $g$
\[
C_1 |||(\iI-\pP)
g |||^2
\leq \Big(\cL g,\, g\Big)_{L^2}\leq  C_2 |||
g |||^2     \, ,
\]
where  $\pP$ is the $L^2$-orthogonal
projection onto the null space $\mathcal{N}$.
\end{theo}

This coercivity estimate of the linearized collisional operator will prove to be crucial for the global existence of classical solutions to the Boltzmann equation. For this purpose, the analysis on the
nonlinear operator is necessary, and we prove the following upper bound estimate.

\begin{theo}\label{theorem-1.1-b} For all $0< s<1$,
assume that $\gamma >\max\{-3,-\frac 32 -2s\}$. Then, one has,
\begin{align*}
\left| \big(\Gamma (f , g),\, h\big)_{L^2} \right| &\lesssim \Big\{ \|f\|_{L^2_{s+\gamma/2}}|||g|||
+ \|g\|_{L^2_{s+\gamma/2}}|||f||| \\
&+ \min\big\{  ||f||_{L^2} || g||_{L^2_{s+\gamma /2}} ,||f||_{L^2_{s+\gamma /2}} || g||_{L^2} \big\} \Big\}|||h||| \, ,\notag
\end{align*}
 for suitable functions $f, g, h$ .
\end{theo}

We will then concentrate on the global existence of
solutions, both weak and strong, for the {\it non-cutoff soft potential case}
in the framework of small perturbation of an equilibrium state. Even though some
estimates hold for the general case and will be used in the forthcoming papers,
the condition $\gamma +2s \le 0$ will be imposed  in the main
existence results. In the Part II \cite{amuxy4-3}, we will then present
the global existence theory for the hard potential case, that is, the condition $\gamma +2s > 0$  imposed. Furthermore,
the qualitative behavior of the solutions, such as
the uniqueness, non-negativity, regularity and convergence rate to the equilibrium will be investigated  in \cite{amuxy4-4}. Note that both the global existence and the qualitative study on the solution behavior
were firstly investigated in \cite{amuxy3} for the Maxwellian molecule case where a generalized uncertainty
principle obtained in \cite{amuxy-nonlinear-b} was used.

We begin with a local existence of classical solutions that  holds true in general case.
\begin{theo}\label{part2-theorem-anygamma}
Assume that the cross-section satisfies \eqref{part2-hyp-2} with $\gamma+ 2s\le 0$, $0<s <1$ and $\gamma >-3$.
Let $N \ge 6$ and $\ell\ge N$.
For a small $\varepsilon>0$, if $||g_0||_{\cH^N_\ell (\RR^6 )} \le \varepsilon$,
then there exists $T>0$ such that the Cauchy problem \eqref{part2-4-2-1.5} admits a solution
$$
g\in L^\infty ([0, T];\, \cH^N_\ell (\RR^6 )).
$$
\end{theo}

Since we are interested in getting global existence results,
the next statement deals
with this issue  asking only for control of $x$ derivatives.

\begin{theo}\label{part2-theo1.1}
Assume that the cross-section satisfies \eqref{part2-hyp-2} with $\gamma+ 2s\le 0$, $0<s <1$ and $\gamma >
\max\{-3, -\frac32 -2s\}$. Let $N\geq 3$. For a
small $\varepsilon>0$, if
$
\|g_0\|_{H^N(\RR^3_x; L^2(\RR^3_v))}\leq \varepsilon,
$
then the Cauchy problem \eqref{part2-4-2-1.5} admits a global solution
$$
g \in L^\infty([0, +\infty[;\, H^N(\RR^3_x; L^2(\RR^3_v))).
$$
\end{theo}

The above global
existence result is in a non-weighted function space  without $v$ derivatives in the
framework of weak solutions. On the other hand, we will prove the following global existence result
on classical solutions for which the proof is more involved. Note that
for the qualitative study on the solution behavior, such as the  regularity
as will be shown in \cite{amuxy4-4}, solutions
in a function space with $x$ and $v$ derivatives together with weight in $v$ are needed.
Hence, the next theorem also serves for this purpose.

\begin{theo}\label{part2-theo1.2}
Assume that the cross-section satisfies \eqref{part2-hyp-2} with $\gamma+ 2s\le 0$, $0<s <1$ and $\gamma >
\max\{-3, -\frac32 -2s\}$. Let $N\geq 6,\, \ell\geq N$. For a small   $\varepsilon>0$,
 if
$ \|g_0\|_{{\tilde \cH}^{N}_\ell(\RR^6)}\leq \varepsilon,
$
then the Cauchy problem \eqref{part2-4-2-1.5} admits a global solution
$$
g\in L^\infty([0, +\infty[\,;\,\,{\tilde \cH}^{N}_\ell(\RR^6)).
$$
\end{theo}

Let us now review some related works on this topic. First of all,
 the well-posedness theory for the Boltzmann equation has now been well established
under the Grad's angular cutoff assumption. Under this assumption, there exist basically
three frameworks of existence of global solutions. The first one was initiated by Grad \cite{Grad} and firstly
completed by Ukai \cite{ukai-1a,ukai-1b,ukai-2} in the framework of weighted $L^\infty_v$ function space for small
perturbation of an equilibrium, where the
spectrum analysis was used through a bootstrap argument. An important progress on the
existence theory is the introduction of the renormalized solutions for large perturbation
in the framework
of  $L^1_v$ function space by DiPerna-Lions \cite{D-L,Lions94},
where the velocity averaging lemma plays a key role. Recently, solutions in $L^2_v$ framework
were established by macro-micro decompositions and energy method for small perturbation of an equilibrium,
cf \cite{guo,guo-1,liu-1,liu-2}.

However, without Grad's angular cutoff assumption, the established mathematical theories are far less.
In this direction, the spectral analysis of the linearized collisional operator was studied by
Pao \cite{Pao}. In 1990's, some simplified models, such as Kac's model and the Boltzmann equation in lower
dimensions with symmetry, were successfully studied, \cite{D95,desv-ajout1,desv-ajout2}. In 2000's, the mathematical theory
for the spatially homogeneous Boltzmann equation was satisfactorily solved, \cite{al-1,al-saf,desv-wen1,HMUY,MU,MUXY-DCDS}. For the original Boltzmann
equation in physical space, in the framework of renormalized solutions, the only existing result
can be found in \cite{al-3} where the basic existence result is still lacking. There are some local existence results, \cite{alex-two-maxwellian,ukai}, see also the reviews \cite{alex-review,villani2}.

Since 2006, we have been working on the original Boltzmann equation without angular cutoff, cf. \cite{amuxy-nonlinear-b,amuxy-nonlinear-3,amuxy3}. Based on a new
generalized uncertainty principle proved in \cite{amuxy-nonlinear-b}, we developed a new approach for
the regularity study. In the framework of small perturbation of an
equilibrium in the whole space, the first complete global well-posedness theory and regularity were established for the Maxwellian molecule case \cite{amuxy3}. As a continuation of these works, we successfully solve, in this series of papers, the fundamental problems, that is, existence, uniqueness, regularity, non-negativity and convergence
rates of solutions, so that a complete and satisfactory mathematical theory is now established under some minimal regularity requirement on the initial data. Through this analysis,
mathematical tools and techniques from harmonic analysis are used and
some new ones are introduced, such as the generalized uncertainty principle and the above non-isotropic norm. Here we would like to mention that recently by using a different method,  an existence result on  solutions in the torus case was obtained in  \cite{gr-st,gr-st-1,gr-st-2}.

Finally, we present the main strategy of analysis in this paper. Based on the essential
coercivity estimate on the linearized collisional operator and the non-isotropic norm
proven in a first step, what is needed in this paper is then the detailed analysis on the collisional
operator in both unweighted and weighted spaces, where its upper bounds, commutators with
differentiation in $v$ and commutators with weights in $v$ are given. Through this analysis,
we can see the role played by the parameters $\gamma$ and $s$ in the cross section. With
these estimates, the energy method can be applied through the macro-micro decomposition
analysis introduced in \cite{guo,guo-1}. Basically, the microscopic component of the
solution
is controlled by the essential coercivity estimate on the linearized collisional operator
in the non-isotropic norm, while the dissipation on the macroscopic component
is recovered by the system on the fluid functions through the macro-micro decomposition.
Then the nonlinear terms are essentially of higher order in the non-isotropic norm
so that the energy estimate can be closed in the framework of small perturbation.

The rest of the paper is arranged as follows. In the next section, we extend the definition of the
non-isotropic norm introduced in \cite{amuxy3} and then state the main estimates in this
paper. The proof of the upper and lower bound estimates of the non-isotropic norm will
be given in Section 3. In Section 4, we will prove the equivalence of
the Dirichlet form of the linearized collison operator and the square of the non-isotropic
norm. The equivalence of the non-isotropic norms with different kinetic factors and different
weights will be shown in Section 5. An upper bound estimate on the nonlinear collision operator which is useful for the well-posedness
theory for the Boltzmann equation will be given next. However, because of unnecessary restrictions on the values of the parameter $\gamma$, we shall amplify such estimations and obtain some new functional estimates in
$v$ variable only in another Section. The functional estimates in both $v$ and $x $ variables
are given in Section \ref{part2-section3}. With these estimates, the local  and global
existence of both weak and classical solutions are given in the last two sections respectively.

{\bf Notations:} Herein, letters $f$, $g, \cdots$ stand for various suitable  functions, while $C$, $c, \cdots$ stand for various numerical constants, independent from functions $f$, $g, \cdots$ and which may vary from line to line. Notation $A\lesssim B$ means that there exists a constant $C$ such that $A \le C B$, and similarly for $A\gtrsim B$. While $A \sim B$ means that there exist two generic constants $C_1, C_2 >0$ such that
$$C_1A \leq B \leq C_2 A.$$

\section{non-isotropic norm and estimates of linearized collision operators}\label{part1-section2}
\setcounter{equation}{0}

For later use, we will need to compare the original cross-section with the
situation when its kinetic part is mollified. That is,
 for the function $\Phi(z)$ appearing in the cross-section, we denote by $\tilde\Phi(z)=(1+| z|^2)^{\gamma\over 2}$ its smoothed version.
To show the dependence of the estimates on the mollified or non-mollified
kinetic factor in the cross-session, we will use the notations $Q^{\tilde\Phi}$ and
$Q^{\Phi}$ to denote the Boltzmann collisional operator when the kinetic part is  $\tilde\Phi$
and $\Phi$ respectively. In particular, $Q=Q^\Phi$.
This upper-script will be also used for other operators as well.
\smallskip

First of all, let us recall that
\begin{align*}
\Big(\cL g,\, g\Big)_{L^2} =-\Big(\Gamma(\sqrt\mu\,,
g)+\Gamma(g,\,\sqrt\mu\,),\, g\Big)_{L^2}\geq 0\, ,
\end{align*}
and the definition of the non-isotropic norm
\begin{align}\label{part1-2.2.1}
||| g|||^2&= \iiint \Phi (|v-v_*|)b(\cos\theta) \mu_*\, \big(g'-g\,\big)^2\,\\
 &\,\,\,\,\,\,\,\,+\iiint \Phi (|v-v_*|)b(\cos\theta)g^2_* \big(\sqrt{\mu'}\,\, - \sqrt{\mu}\,\,
\big)^2\nonumber\\
&=J_1+J_2\nonumber\,,
\end{align}
where the integration  is over
$\RR^3_v\times\RR^3_{v_\ast}\times\SS^2_\sigma$.

The following proposition gives a precise version of Theorem \ref{part1-theo2}.
\begin{prop}\label{part1-prop2.1} Assume that the cross-section satisfies (\ref{part2-hyp-2}) with
$0<s<1$ and $\gamma>-3$. Then
there exist two generic constants $C_1, C_2>0$ such that
\begin{equation*}
C_1 ||| (\iI - \pP )g|||^2\leq \Big(\cL g,\, g\Big)_{L^2}
\leq  2\Big(\cL_1 g,\, g\Big)_{L^2}\leq C_2 ||| g|||^2
\end{equation*}
for any suitable function $g$.
\end{prop}

\smallbreak
Concerning the lower and upper bounds of the non-isotropic norm we have
\begin{prop}\label{part1-prop2}
Assume that the cross-section satisfies (\ref{part2-hyp-2}) with
$0<s<1$ and $\gamma>-3$.
Then there exist two generic constants $C_1, C_2>0$ such that
\[
C_1 \left\{\left\|
g\right\|^2_{H^s_{\gamma/2}}+\left\|
g\right\|^2_{L^2_{s+\gamma/2}}\right\}
\leq ||| g |||^2 \leq  C_2 \left\|
g\right\|^2_{H^s_{s+\gamma/2}}
\]
for any suitable function $g$.
\end{prop}
>From this estimate and Theorem \ref{part1-theo2}, we can get the following estimate in classical weighted Sobolev spaces
\begin{align}\label{equivalent}
C_1 \left\{\left\|
(\iI - \pP ) g\right\|^2_{H^s_{\gamma/2}}+\left\|
(\iI - \pP ) g\right\|^2_{L^2_{s+\gamma/2}}\right\}
\leq \Big(\cL g,\, g\Big)_{L^2} \leq  C_2 \left\|
g\right\|^2_{H^s_{s+\gamma/2}}     \, .
\end{align}

In the following, we will use the lower script $\Phi$ on the non-isotropic norm, and so use the notation $|||g|||_{\Phi}$
if we need to
specify its dependence on the kinetic factor $\Phi$. Notations $J_1^{\Phi}, J_2^{\Phi}$ will be also used for the same purpose.

Part of the proof on the lower bound of the non-isotropic norm given
in
Proposition \ref{part1-prop2} is essentially due to the following equivalence
relations.

\begin{prop}\label{part1-prop3}
Assume that the cross-section satisfies (\ref{part2-hyp-2}) with
$0<s<1$ and $\gamma>-3$. Then we have
\begin{equation*}
|||g|||_{\Phi} \sim |||g|||_{\tilde \Phi}\,.
\end{equation*}
\end{prop}

Concerning the dependence on the index $\gamma$ in $\Phi_\gamma = |v-v_*|^\gamma$,  we have
\begin{prop}\label{part1-prop4}
Assume that the cross-section satisfies (\ref{part2-hyp-2}) with
$0<s<1$ and $\gamma>-3$. Then for any $\beta >-3$, we have
\begin{equation*}
|||g|||_{\Phi_{\gamma}} \sim |||\la v \ra^{(\gamma-\beta)/2}g|||_{\Phi_\beta}\,.
\end{equation*}
\end{prop}

\subsection{Bounds on the non-isotropic norm}\label{part1-Section3}

\smallskip

This section is devoted to the proof of Proposition
\ref{part1-prop2}. We will
often use the following elementary estimate stated in velocity dimension $n$, since it will be needed for both cases $n=2$ and $n=3$.

\begin{lemm}\label{part1-tro-funda} Let the velocity dimension be $n$, $n\in \NN$, $\rho >0, \delta \in \RR$ and
let $\mu_{\rho,\delta}(u) = \la u \ra^\delta e^{-\rho|u|^2}$ for $u \in \RR^n$. If $\alpha > -n$
and $\beta \in \RR$, then we have
\begin{equation}\label{part1-ele}
I_{\alpha, \beta}(u) =\int_{\RR^n} |w|^\alpha \la w \ra^\beta \mu_{\rho, \delta}(w+u) dw
\sim \la u \ra^{\alpha + \beta}\,.
\end{equation}
\end{lemm}
\begin{proof}
Since we have
\[
\la u \ra^\beta \la u + w \ra^{-|\beta|}
\leq \la w \ra^\beta \leq \la u \ra^\beta \la u + w \ra^{|\beta|},
\]
it suffices to show \eqref{part1-ele} with $\beta=0$,
by taking $\mu_{\rho,\delta\pm |\beta|}$ instead of $\mu_{\rho, \delta}$.
Taking into account the fact that $\alpha >-n$,
this estimate is obvious when $|u| \leq 1$. If $|u| \geq 1$, then we have
\[
I_{\alpha,0}(u) \geq 4^{-|\alpha|}\la u \ra^\alpha
 \int_{\{|u+w| \leq 1/2\}}\mu_{\rho,\delta}(u+w) dw \gtrsim \la u \ra^\alpha,
\]
because $|u+w| \leq 1/2$ implies that $4^{-1}\la u \ra \leq |w| \leq 4 \la u \ra$.
Noticing that $2|w| \geq \la w \ra$ if $|w| \geq 1$, we have
\begin{align*}
I_{\alpha,0}(u) & \leq \Big(\max_{|w|\leq 1} \mu_{\rho,\delta}(u+w)\Big)
\int_{ \{|w|\leq 1\}} |w|^\alpha dw +
2^{|\alpha|}\int_{ \{|w|\geq 1\}} \la w \ra^\alpha \mu_{\rho,\delta}(u+w)dw \\
& \lesssim \Big( \la u \ra^{|\delta|} e^{-\rho|u|^2/2} + \la u \ra^\alpha
\int_{\RR^n} \la u+w \ra^\alpha \mu_{\rho,\delta}(u+w)dw
\Big)\lesssim \la u \ra^\alpha \,.
\end{align*}

And this completes the proof of the lemma.

\end{proof}

Recall from \eqref{part1-2.2.1} that the non-isotropic norm contains two parts, denoted by $J_1$
and $J_2$ respectively. The estimation on each part will be given in
the following subsections. We start with the estimation on $J_2$ because
the analysis is easier.

 Let us start with the following upper bound on
 $J_2$.

\begin{lemm}\label{part1-J-2-upper}
Under the same assumptions as in Theorem \ref{part1-theo2}, we have
\[
J_2 := \iiint b(\cos\theta)\Phi(|v-v_*|) g^2_* \big(\sqrt{\mu'}\,\, - \sqrt{\mu}\,\,
\big)^2 dv dv_* d \sigma
\lesssim \|g\|_{L^2_{s+\gamma/2}}^2\,.
\]
\end{lemm}
\begin{proof}
Note that
\begin{align*}
J_2 &\leq 2 \iiint b\Phi(|v-v_*|) g_*^2 \Big(\mu'^{1/4}
- \mu^{1/4}\Big)^2\Big(\mu'^{1/2}
+ \mu^{1/2}\Big)  dv dv_* d\sigma\\
\lesssim & \iiint
b \,|v' -v_*|^\gamma g_*^2 \Big(\mu'^{1/4}
- \mu^{1/4}\Big)^2 \mu'^{1/2}  dv dv_* d\sigma\\
& +
\iiint
b \,|v -v_*|^\gamma g_*^2 \Big(\mu'^{1/4}
- \mu^{1/4}\Big)^2 \mu^{1/2}  dv dv_* d\sigma \\
=& F_1 + F_2\,.
\end{align*}
By the regular change of variables $v \rightarrow v'$, we have
\begin{align*}
F_1 &\lesssim \iint  |v'-v_*|^\gamma \Big(
\int b(\cos \theta) \min \Big ( |v'-v_*|^2  \theta^2, 1 \Big) d \sigma\Big)
 g_*^2 \mu'^{1/2}  dv' dv_* \\
&\lesssim \int \Big(\int |v'-v_*|^{\gamma+2s} \mu' dv' \Big) g_*^2 dv_*
\lesssim \|g\|^2_{L^2_{s+\gamma/2}}\,,
 \end{align*}
 where we have used Lemma \ref{part1-tro-funda} in the case $n=3$ to get the last inequality.
A direct estimation show thats the same bound holds true for $F_2$.
And this completes the proof of the lemma.
\end{proof}

\begin{rema}
Note that the above lemma holds even if $\Phi$ is replaced by $\tilde \Phi$
by using  Lemma \ref{part1-tro-funda}.
\end{rema}

We now turn to the lower bound for $J_2$.

\begin{lemm}\label{part1-radja-moment-estimate}
Under the assumptions \eqref{part2-hyp-2} with {
$\gamma> -3$}, there exists a constant $ C>0$ such that
\[
J_2 := \iiint b(\cos\theta)\Phi(|v-v_*|) g^2_* \big(\sqrt{\mu'}\,\, - \sqrt{\mu}\,\,
\big)^2 dv dv_* d \sigma
\geq C \|g\|_{L^2_{s+\gamma/2}}^2.
\]
\end{lemm}
\begin{proof}
We will apply the argument used in \cite{Vi99}. By shifting to the $\omega$-representation,
\begin{align*}
v' =v- \big((v -v_*)\cdot \omega \big )\omega \enskip \,
v_*' =v +\big((v -v_*)\cdot \omega \big )\omega \enskip , \omega \in \SS^2,
\end{align*}
in view of  the change of variables $(v,v_*) \rightarrow (v_*,v)$, we get,
\begin{align*}
J_2 = 4 \iiint b(\cos\theta) \sin(\theta/2)\Phi(|v-v_*|) g^2 \big( \sqrt{\mu'_*}\,\,
- \sqrt{\mu_*}\,\,
\big)^2 dv dv_* d \omega\,,
\end{align*}
because $d\sigma = 4 \sin (\theta/2) d \omega$.
Then, we use the Carleman representation.
The idea of this representation
  is to replace the set of variables $(v,v_*,\omega)$ by the set $(v,v',v'_*)$. Here,
$v, v' \in \RR^3$ and $v'_* \in E_{vv'}$, where
$E_{vv'}$ is  the hyperplane passing through $v$ and orthogonal
to $v-v'$. By using the formula
\[
dv_*d\omega = \frac{dv'_* dv'}{|v-v'|^2},
\]
cf.  page 347 of \cite{Vi99}, and by taking the change of variables
 \[
(v,v',v'_*) \rightarrow (v, v+h, v+y),
\]
with $h \in \RR^3$ and $y \in E_h$,  where $ E_h$ is
the hyperplane orthogonal to $h$ passing through the origin in $\RR^3$, we have
\begin{align*}
J_2 \sim &
\int_{\RR_{v}^3} \int_{\RR_h^3}
\int_{y \in E_h \cap {\{|y|\geq |h|\}}} \frac{|y|^{1+2s+\gamma}}{|h|^{1+2s}}
 g(v)^2 \\
& \qquad \times \big(\sqrt{\mu(v+ y )}\,\, - \sqrt{\mu(v+y+h)}\,\,
\big)^2  dv\frac{dh d y}{|h|^2}\,,
\end{align*}
because
\begin{align*}
&|h| = |v'-v| = |v'_* -v| \tan \frac{\theta}{2} = |y| \tan \frac{\theta}{2}, \quad
\theta \in [0, \pi/2] \,, \\
&b(\cos\theta) \sin(\theta/2)\Phi(|v-v_*|)
\sim
\frac{|v_* - v'|^{1+2s+\gamma}}{|v- v'|^{1+2s}} {\bf 1}_{\{|v_*'-v| \geq |v'-v|\}}
\,.
\end{align*}

We decompose $v = v_{1} + v_{2}$, where $v_{2}$ is the orthogonal projection  of
 $v$ on  $E_h$. Since $\mu$ is invariant by rotation, we may assume
 $v = (0,0, |v|)$ without loss of generality. By introducing  the polar coordinates
 \[h = (\rho \sin \vartheta \cos \phi,
 \rho \sin \vartheta \sin \phi, \rho \cos \vartheta)\,, \vartheta \in [0,\pi],
 \, \phi \in [0, 2\pi], \rho >0,
\]
we obtain $|v_{1}|= |v||\cos \vartheta|$,
$|v_{1} +h|= | \, |v|\cos \vartheta + \rho|$
and $|v_{2}|= |v|\sin \vartheta$.
Note that if $0 < \vartheta \leq \pi/2$, then
\begin{align*}
\big(\sqrt{\mu(v+ y)}\,\, - \sqrt{\mu(v+y+h)}\,\,
\big)^2
&= \mu(v_{2}+ y) \big(\sqrt{\mu(v_{1})}\,\, - \sqrt{\mu(v_{1}+h)}\,\,
\big)^2\\
&\geq \mu(v_{2}+ y)\mu(v_{1})\big(1 -e^{-\rho^2/4}\big)^2/(2\pi)^{3/2}.
\end{align*}
Therefore, we have for any $\delta >0$
\begin{align*}
J_2 &\geq C \int_{\RR_{v}^3}g(v)^2
\Big\{\int_{\RR_h^3}\frac{\big(\sqrt{\mu(v_{1})}\,\, - \sqrt{\mu(v_{1} +h)}\,\,
\big)^2}{|h|^{3+2s}}\\
&\qquad \qquad \times
\Big(\int_{y \in E_h \cap { \{|y|\geq |h|\}}}  |y|^{1+2s+\gamma} \mu(v_{2}+ y)dy \Big) dh \Big \}
dv\\
&\geq C \int_{\RR_{v}^3}g(v)^2
\Big\{ \int^{\pi/2}_{\pi/2 -1/\langle v \rangle}
\mu(v_{1})\left(\int_0^\delta \frac{\big(1 - e^{-\rho^2/4}\big)^2}
{\rho^{1+2s}} \right. \\
&\qquad \qquad \times \Big(\int_{y \in E_h }  |y|^{1+2s+\gamma} \mu(v_{2}+ y)dy \\
& \qquad \qquad \left. - \int_{y \in E_h \cap { \{|y|\leq \rho \}}}  |y|^{1+2s+\gamma} \mu(v_{2}+ y)dy\Big)
d \rho \right) \sin \vartheta d \vartheta  \Big\}dv\,.
\end{align*}
Since we have
\[
\int_{y \in E_h \cap { \{|y|\leq \rho \}}}  |y|^{1+2s+\gamma} \mu(v_{2}+ y)dy
\leq \delta^{2s} \int_{y \in E_h}  |y|^{1+\gamma} \mu(v_{2}+ y)dy,
\enskip \mbox{ if $\rho \leq \delta$ },
\]
and it follows from Lemma \ref{part1-tro-funda} in the case $n=2$, that
\[  \int_{y \in E_h}  |y|^{\beta} \mu(v_{2}+ y)dy \sim \la v_{2} \ra^\beta
\enskip \mbox{if $\beta >-2$}\,\,,
\]
there exist two constants $C_1, C_2 >0$ such that if $\rho \leq \delta$, we have
\begin{align*}
\int_{y \in E_h }  |y|^{1+2s+\gamma} \mu(v_{2}+ y)dy &- \int_{y \in E_h \cap {\{|y|\leq \rho \}}}  |y|^{1+2s+\gamma} \mu(v_{2}+ y)dy\\
&\geq C_1 \la v_{2} \ra^{1+2s+\gamma}
-C_2 \delta^{2s} \la v_{2} \ra^{1+\gamma}\,.
\end{align*}
Taking a sufficiently small $\delta>0$ gives
\begin{align*}
J_2
&\geq C \int_{\RR_{v}^3}g(v)^2
\Big\{ \int^{\pi/2}_{\pi/2 -1/\langle v\rangle}
\mu(v_{1})\\
& \qquad \times \left(\int_0^\delta \frac{\big(1 - e^{-\rho^2/4}\big)^2}
{\rho^{1+2s}} d \rho \right) \la v_{2} \ra^{1+2s+\gamma}
\sin \vartheta d \vartheta  \Big\}dv\\
&\geq C_{\delta} \int_{\RR_{v}^3}  \langle v \rangle^{2s+\gamma} g(v)^2
\Big\{ \int^{\pi/2}_{\pi/2 -1/\langle v \rangle}
e^{-|v|^2 \cos^2 \vartheta} \langle v \rangle d \vartheta  \Big \}dv\\
& \geq C_{\delta} \|g\|^2_{s+\gamma/2}\,.
\end{align*}
The proof of the lemma is now completed.
\end{proof}

\begin{rema}\label{part1-remark-of-lem3-1}
In the above proof, the factor $|y|^\gamma$ can be replaced by $\la y \ra^\gamma$,
so that Lemma \ref{part1-radja-moment-estimate} is valid even if $\Phi$
is replaced by $\tilde \Phi = \la v-v_* \ra^\gamma$.
By the above lemma together with Lemma \ref{part1-J-2-upper} and the Remark
after it,  we can conclude
\begin{equation}\label{part1-J-2-equivalence}
J_2^{\Phi} \sim \|g\|^2_{L^2_{s+\gamma/2}} \sim  J_2^{\tilde \Phi}\,.
\end{equation}
\end{rema}

\smallskip
We now turn to the estimation of the first term of the non-isotropic norm, that is, $J_1$. We will firstly show that the singular
 behavior of the cross-section when $v=v_*$ can be smoothed out.
This point is given by the following proposition.

\begin{prop}\label{part1-radja-equivalence-modified}
Under the same assumption as in  Theorem \ref{part1-theo2}, we have
\[
J_1^{\Phi} + \|g\|^2_{L^2_{s+\gamma/2}}
 \sim J_1^{\widetilde \Phi} + \|g\|^2_{L^2_{s+\gamma/2}}\,.
\]
\end{prop}
\begin{rema}
This proposition  is nothing but Proposition \ref{part1-prop3}
by  Remark \ref{part1-remark-of-lem3-1}.
\end{rema}

\begin{proof}
By using similar arguments as in the proof of Lemma \ref{part1-radja-moment-estimate},
it follows from the
Carleman representation that
\begin{align*}
J_1^{\Phi} &\sim
\int_{\RR_{v}^3} \int_{\RR_h^3}
\int_{y \in E_h \cap \{|y| \geq |h|\}} \frac{|y |^{1+2s+\gamma}}{|h|^{1+2s}}
 \mu(v)\big(g(v+ y)\,\, - g(v+y+h)\,\,
\big)^2  dv \frac{dh d y}{|h|^2}\\
&= \int_{\RR_{v}^3} \int_{\RR_h^3}
\int_{y \in E_h \cap \{|y| \geq |h|\}} \frac{|y|^{1+2s+\gamma}}{|h|^{1+2s}}
 \mu(v+y)\big(g(v)\,\, - g(v+h)\,\,
\big)^2  dv \frac{dh d y}{|h|^2}\,,
\end{align*}
where
 the last equality is a direct consequence of the change of variables
$(v +y, y) \rightarrow (v, -y)$.

Similarly, we have
\[
J_1^{\widetilde \Phi}  \sim
\int_{\RR_{v}^3} \int_{\RR_h^3}
\int_{y \in E_h\cap \{|y| \geq |h|\} } \frac{|y|^{1+2s} \langle y \rangle^\gamma}{|h|^{1+2s}}
 \mu(v+y)\big(g(v + h)\,\, - g(v)\,\,
\big)^2  dv \frac{dh d y}{|h|^2} \,.
\]
We claim that
\begin{align}\label{part1-yoshi-1}
&\int_{\RR_{v}^3} \int_{\RR_h^3}
\int_{y \in E_h \cap \{|y| \leq |h|\}} \frac{|y|^{1+2s+\gamma}}{|h|^{1+2s}}
 \mu(v+y)\big(g(v + h)\,\, - g(v)\,\,
\big)^2  dv \frac{dh d y}{|h|^2}  \\
& \qquad \lesssim \|g\|^2_{L^2_{s+\gamma/2}}\,, \notag \\
&\int_{\RR_{v}^3} \int_{\RR_h^3}
\int_{y \in E_h\cap \{|y| \leq |h|\} } \frac{|y|^{1+2s} \langle y \rangle^\gamma}{|h|^{1+2s}}
 \mu(v+y)\big(g(v + h)\,\, - g(v)\,\,
\big)^2  dv \frac{dh d y}{|h|^2}   \label{part1-yoshi-2}\\
& \qquad \lesssim \|g\|^2_{L^2_{s+\gamma/2}}\,.  \notag
\end{align}
Note carefully that the integration in these estimates is performed for "large" values of $h$.

Once we admit those estimates,
to conclude the proof of the lemma, it suffices to show that
\begin{align*}\label{part1-radja-sim}
G(v,h)= \int_{y \in E_h} |y|^{1+2s+ \gamma} \mu(v+y) dy &\sim
\int_{y \in E_h} |y|^{1+2s} \langle y \rangle^\gamma \mu(v+y)dy = \tilde G(v,h).
\end{align*}
We decompose $v= v_{1} + v_{2}$, where $v_{2}$ is the orthogonal projection  of
 $v$ on  $E_h$. Then we have $\mu(v+y) = \mu(v_1) \mu(v_2+y)$, whence
it follows from
Lemma \ref{part1-tro-funda} together with $1+2s+\gamma > -2$ that
\[
G(v,h) \sim \mu(v_1) \la v_2 \ra^{1+2s+ \gamma} \sim \tilde G(v,h)\,.
\]

It remains to show
\eqref{part1-yoshi-1} and \eqref{part1-yoshi-2}. We write
\begin{align*}
&\int_{\RR_{v}^3} \int_{\RR_h^3}
\int_{y \in E_h \cap \{|y| \leq |h|\}} \frac{|y|^{1+2s+\gamma}}{|h|^{1+2s}}
 \mu(v+y)\big(g(v + h)\,\, - g(v)\,\,
\big)^2  dv \frac{dh d y}{|h|^2}\\
&=
\int_{\RR_{v}^3} \int_{\{|h| \leq 1\}}
\int_{y \in E_h \cap \{|y| \leq |h|\}}
\enskip + \int_{\RR_{v}^3} \int_{\{|h| \geq 1\}}
\int_{y \in E_h \cap \{|y| \leq |h|\}} \enskip = A_1 + A_2\,.
\end{align*}
Take a small $\delta >0$ such that $\gamma - \delta >-3$. Then,
in view of $1+\gamma -\delta >-2$, we have
\begin{align*}
A_1 &\leq
\int_{\RR_{v}^3} \int_{\{|h| \leq 1\}}
\int_{y \in E_h}
\frac{|y|^{1+\gamma-\delta}}{|h|^{1-\delta}}
 \mu(v+y)\big(g(v + h)\,\, - g(v)\,\,
\big)^2  dv \frac{dh d y}{|h|^2}\\
&= \int_{\RR_{v}^3}\mu(v_1) \int_{\{|h| \leq 1\}}
\Big(\int_{y \in E_h}
 |y|^{1+\gamma-\delta}
 \mu(v_2+y) dy \Big)\big(g(v + h)\,\, - g(v)\,\,
\big)^2  \frac{dh}{|h|^{3-\delta}}dv \\
&\lesssim \int_{\RR_{v}^3}\mu(v_1)\la v_2\ra^{1+\gamma-\delta} \int_{\{|h| \leq 1\}}
\big(g(v + h)\,\, - g(v)\,\,
\big)^2  \frac{dh}{|h|^{3-\delta}}dv\\
&\lesssim \int_{\RR_{v}^3}\int_{\{|h| \leq 1\}}\Big(\mu(v_1-h) +  \mu(v_1) \Big)\la v_2\ra^{1+\gamma-\delta}
\big|g(v)\big |^2
 \frac{dh}{|h|^{3-\delta}}dv\,,
\end{align*}
where we have used the change of variables $v+h \rightarrow v$ for the factor
$g(v+h)$.
As in the proof of Lemma \ref{part1-radja-moment-estimate},
by assuming
 $v = (0,0, |v|)$, we introduce the polar coordinates
 \[h = (\rho \sin \vartheta \cos \phi,
 \rho \sin \vartheta \sin \phi, \rho \cos \vartheta)\,, \vartheta \in [0,\pi],
 \, \phi \in [0, 2\pi], \rho >0\,.
\]
Since $|v_{1}|= |v||\cos \vartheta|$,
$|v_{1} -h|= | \, |v|\cos \vartheta - \rho|$
and $|v_{2}|= |v|\sin \vartheta$, by using the change of varible $|v|\cos \vartheta
=r$, we obtain
\begin{align*}
A_1 &\lesssim \int_{\RR_{v}^3}\big|g(v)\big |^2 \int_0^1
\frac{1}{\rho^{1-\delta}}\\
&\times \left(\int_{-|v|}^{|v|} \frac{(1+|v|^2 -r^2)^{(1+\gamma -\delta)/2}}{|v|}
\Big(e^{-|r-\rho|^2/2} + e^{-r^2/2}\Big)
dr\right) d\rho dv\,.
\end{align*}
Similarly, if $1+2s-\delta >1$, then we have
\begin{align*}
A_2&\leq
\int_{\RR_{v}^3} \int_{\{|h| \geq 1\}}
\int_{y \in E_h}
\frac{|y|^{1+\gamma+2s-\delta}}{|h|^{1+2s-\delta}}
 \mu(v+y)\big(g(v + h)\,\, - g(v)\,\,
\big)^2  dv \frac{dh d y}{|h|^2}\\
&\lesssim \int_{\RR_{v}^3}\big|g(v)\big |^2 \int_1^\infty
\frac{1}{\rho^{1+2s-\delta}}\\
&\times \left(\int_{-|v|}^{|v|} \frac{(1+|v|^2 -r^2)^{(1+\gamma +2s-\delta)/2}}{|v|}
\Big(e^{-|r-\rho|^2/2} + e^{-r^2/2}\Big)
dr\right) d\rho dv\,.
\end{align*}
If $1+\gamma+2s-\delta \geq 0$, then
\begin{align*}
K(v,\rho )&=\int_{-|v|}^{|v|} \frac{(1+|v|^2 -r^2)^{(1+\gamma +2s-\delta)/2}}{|v|}
\Big(e^{-|r-\rho|^2/2} + e^{-r^2/2}\Big) dr \\
&\leq
\la v \ra^{(\gamma +2s-\delta)/2}
\int_{-|v|}^{|v|}
\Big(e^{-|r-\rho|^2/2} + e^{-r^2/2}\Big) dr
\lesssim \la v \ra^{\gamma+2s},
\end{align*}
which shows
\begin{align}\label{part1-A2es}
A_2 \lesssim
\int_{\RR_{v}^3}\big|g(v)\big |^2 \int_1^\infty
\frac{K(v,\rho)}{\rho^{1+2s-\delta}}d\rho dv
\lesssim  \int \la v \ra^{\gamma+2s}\big|g(v)\big |^2  dv\,.
\end{align}
On the other hand, if $1+\gamma+2s-\delta <0$ and $|v| \geq 3$, then
\begin{align*}
K(v,\rho )&\lesssim \int_{0}^{|v|} \frac{(|v|^2 -r^2)^{(1+\gamma +2s-\delta)/2}}{|v|}
\Big(e^{-|r-\rho|^2/2} + 3 e^{-r^2/2}\Big) dr \\
&\lesssim
|v |^{(-1+\gamma +2s-\delta)/2}
\int_{0}^{|v|} \Big(|v| - r \Big)^{(1+\gamma +2s-\delta)/2}
\Big(e^{-|r-\rho|^2/2} + 3 e^{-r^2/2}\Big) dr\\
&\lesssim
 \la v \ra^{\gamma+2s} +
|v |^{(-1+\gamma +2s-\delta)/2}
\int_{|v|/2}^{|v|} \Big(|v| - r \Big)^{(1+\gamma +2s-\delta)/2}
3 e^{-|r-\rho|^2/2} dr,
\end{align*}
because
\begin{align*}
\int_0^{|v|} \Big(|v| - r \Big)^{(1+\gamma +2s-\delta)/2}
e^{-|r|^2/2} dr &\lesssim |v|^{(1+\gamma +2s-\delta)/2}
\int_0^{|v|/2}e^{-|r|^2/2} dr \\
&+ e^{-|v|^2/8}\int_{|v|/2}^{|v|}
\Big(|v| - r \Big)^{(1+\gamma +2s-\delta)/2} dr,
\end{align*}
where we have used that
 $(1+\gamma +2s-\delta)/2 >-1$ for small $\delta>0$
 that follows from the assumption $\gamma >-3$.
We now consider
\begin{align*}
&\int_1^\infty
\frac{d \rho}{\rho^{1+2s-\delta}}
\int_{|v|/2}^{|v|} \Big(|v| - r \Big)^{(1+\gamma +2s-\delta)/2}
e^{-|r-\rho|^2/2} dr\\
&\leq
\int_{|v|/2}^{|v|}\Big(|v| - r \Big)^{(1+\gamma +2s-\delta)/2}
\Big (\int_{\{|r-\rho|\leq \sqrt{2\log|v|}\}} (|v|/3)^{-(1+2s-\delta)} d\rho \Big)
dr\\
&+ \int_{|v|/2}^{|v|}\Big(|v| - r \Big)^{(1+\gamma +2s-\delta)/2}
\Big (\int_{\{|r-\rho|\geq \sqrt{2\log|v|}\}}\frac{ |v|^{-1}d \rho}{\rho^{1+2s-\delta}}
\Big) dr\\
&\lesssim \Big(|v|^{(1+\gamma +2s-\delta)/2}
\sqrt{2\log|v|} +|v|^{(1+\gamma +2s-\delta)/2}\Big) \lesssim \la v
\ra^{(1+\gamma +2s)/2}\,.
\end{align*}
Therefore, in the case when $1+\gamma+2s-\delta <0$, we
also have \eqref{part1-A2es}.
Similarly, we have

\begin{align*}
A_1\lesssim
\int_{\RR_{v}^3}\big|g(v)\big |^2 \int_0^1
\frac{K(v,\rho)}{\rho^{1-\delta}}d\rho dv
\lesssim  \int \la v \ra^{\gamma+2s}\big|g(v)\big |^2  dv\,,
\end{align*}
which shows \eqref{part1-yoshi-1}. The proof of \eqref{part1-yoshi-2}
is similar, and thus the proof of the proposition is completed.

\end{proof}

\begin{lemm}\label{part1-coer}
There exist constants $C_1, C_2> 0$ such that
\begin{align}\label{part1-radja-Sobolev-coercive}
J_1 &\geq C_1 \|\la v \ra^{\gamma/2} g\|^2_{H^s}
- C_2 \|g\|^2_{L^2_{s+\gamma/2}}\,.
\end{align}
The same conclusion holds even with $\mu$ replaced by
$\mu^\rho$ for any fixed $\rho>0$.
\end{lemm}
\begin{proof}
It follows from Proposition \ref{part1-radja-equivalence-modified} that
\begin{align}\label{part1-radja-triple-lower}
C \Big( J_1^{\Phi} + \|g\|^2_{L^2_{s+\gamma/2}} \Big)&\geq  2 \, J_1^{\widetilde \Phi}
 \notag \\
&
\geq \iiint b(\cos\theta)
\frac{\mu_*}{\la v_*\ra^{|\gamma|}}  \, \Big(\la v' \ra^{\gamma/2}g'-
\la v \ra^{\gamma/2} g\,\Big)^2
d\sigma dvdv_* \\
&\quad  - 2 \iiint b(\cos\theta)
\frac{\mu_*}{\la v_*\ra^{|\gamma|}} \Big(
\la v \ra^{\gamma/2} -\la v' \ra^{\gamma/2}\Big)^2 |g|^2
d\sigma dvdv_*  \notag \\
&= A_1+A_2, \notag
\end{align}
because $\widetilde \Phi(|v-v_*|) \sim \displaystyle \la v'-v_* \ra^\gamma \geq
\frac{\la v' \ra^{\gamma}}{\la v_*\ra^{|\gamma|}}$ and $2(a+b)^2 \geq a^2 - 2b^2$.
Setting $\tilde g = \la v \ra^{\gamma/2} g$ and noting $ C_\gamma \mu(v)\la v \ra^{-|\gamma|}
\geq \mu(2v)$ for a $C_\gamma >0 $, as in Proposition 1 of \cite{al-1},    we have
\begin{align*}
C_\gamma \, A_1 &\geq \int_{\RR^{6}}\int_{\mathbb S^{2}}b\Big (\frac{v-v_*}{|v-v_*|} \cdot \,\sigma \Big)
\mu(2v_*)\Big(\tilde g(v)- \tilde g(v')\Big)^2
d\sigma d v_* d v\\
&=(4\pi)^{-3}\int_{\RR^{3}}\int_{\mathbb
S^{2}}b\Big(\frac{\xi}{|\xi|} \,\cdot\,\sigma\Big) \Big\{
\hat{\mu}(0)|\hat{\tilde g}(\xi)|^2+\hat{\mu}(0)|\hat{ \tilde g}(\xi^+)|^2\\
&\qquad \qquad  \qquad -2Re\,\hat{\mu}(\xi^-/2) \hat{ \tilde g}(\xi^+)\bar{\hat{ \tilde g}}(\xi)\Big\} d\sigma d \xi\\
&\geq \frac{1}{2(4\pi)^{3}} \int_{\RR^{3}} |\hat{ \tilde g }(\xi)|^2
{ \left\{\int_{\mathbb S^{2}}b\Big(\frac{\xi}{|\xi|}
\,\cdot\,\sigma\Big) (\hat{\mu}(0)- |\hat{\mu}(\xi^-/2)|) d\sigma\right\}}
d \xi\,.
\end{align*}
Since we have
 $\hat{\mu}(0)- |\hat{\mu}(\xi^-/2)|= c(1- e^{-|\xi^-|^2/8})\geq c' { |\xi^-|^2}$ if $|\xi^-| \leq 1$,
in view of
 $|\xi^-|^2 = |\xi|^2 \sin^2 \theta/2 \geq |\xi|^2 (\theta/\pi)^2$,  we obtain for $|\xi| \geq 1$
\begin{align*}
\int_{\mathbb S^{2}}b\Big(\frac{\xi}{|\xi|}
\,\cdot\,\sigma\Big) { (\hat{\mu}(0)- |\hat{\mu}(\xi^-)|) }d\sigma
&\geq \int_{|\xi| (\theta/\pi)\leq 1}\sin \theta b(\cos \theta)
{ |\xi|^2 (\theta/\pi)^2 }d \theta\\
&\geq c'' K { |\xi|^2 }\int_0^{1 /|\xi|} \theta^{-1-2s}{ \theta^2} d \theta\\
&= c'' K{ |\xi|^2 }|\xi|^{2s-2}/(2-2s).
\end{align*}
Therefore, we have
\begin{align}\label{part1-coerv}
A_1
&\geq C_1 \int_{|\xi|\geq 1}{ |\xi|^{2s} }|\hat {\tilde g}(\xi)|^2d \xi
\geq C_1 2^{-2s} \int_{|\xi|\geq 1}(1+|\xi|^2)^{s} |\hat{ \tilde g}(\xi)|^2 d \xi\\
&\geq C_1 2^{-2s}\|\la v \ra^{\gamma/2} g\|^2_{H^s(\RR^3_v)}-C_1\| \la v \ra^{\gamma/2} g\|^2_{L^2(\RR^3_v)}\,.\notag
\end{align}

As for $A_2$, we note that
if $v_\tau = v' + \tau(v-v')$ for $\tau \in [0,1]$, then
\[
\la v \ra \leq \la v -v_*\ra + \la v_* \ra \leq \sqrt 2\la v_\tau -v_*\ra + \la v_*\ra
\leq (1+\sqrt 2) \la v_\tau \ra \la v_* \ra,
\]
and $\la v_\tau \ra \leq (1+\sqrt 2) \la v \ra \la v_* \ra$, which show
$\la v_\tau \ra^\beta \leq C_\beta \la v \ra^\beta \la v_* \ra^{|\beta|}$ for any $\beta
\in \RR$. It follows that
\begin{align*}
\Big|\la v \ra^{\gamma/2} -\la v' \ra^{\gamma/2}\Big|
&\leq C_\gamma \int_0^1 \la
v' + \tau(v-v') \ra^{(\gamma/2 -1)} d \tau |v-v_*|\theta\,\\
&\leq C'_\gamma \Big (\la v \ra^{(\gamma/2 -1)}
\la v_* \ra^{|\gamma/2 -1|} \Big) \, \la v-v_* \ra \theta,
\end{align*}
and thus we have
\begin{align*}
A_2 &\leq C \iint \frac{\mu_*}{\la v_*\ra^{|\gamma|}}|g|^2
\Big\{ \Big (\la v \ra^{(\gamma -2)}
\la v_* \ra^{|\gamma-2|} \Big)^2 \Big(
\int_0^{\la v -v_* \ra^{-1}} \theta^{-1-2s}  \Big(
\la v-v_* \ra \theta \Big)^2  d \theta \Big)\\
& \qquad +
\int_{\la v -v_* \ra^{-1}} ^{\pi/2}
\Big (\la v \ra^{\gamma}
+ \la v \ra^\gamma \la v_* \ra^{|\gamma|} \Big)\,\theta^{-1-2s} d\theta \Big\}
 dvdv_* \\
&\leq C \iint   \Big (\la v \ra^{2s + \gamma}
\la v_* \ra^{2s+ \max(|\gamma-2|-|\gamma|, 0)} \Big) \mu_*|g|^2dvdv_*
\leq C \|g\|^2_{L^2_{s+\gamma/2}}\,,
\end{align*}
which together with  \eqref{part1-coerv} yields
the desired estimate \eqref{part1-radja-Sobolev-coercive}.
The last estimate of the lemma is obvious by replacing $\mu$ by $\mu^\rho$ in each step of the above arguments, so that the proof of the lemma is completed.
\end{proof}

Lemma \ref{part1-radja-moment-estimate} together with Lemma \ref{part1-coer}
implies that we have the following lower bound on the non-isotropic norm,
\begin{equation*}
|||g|||^2 \gtrsim \Big( \|g\|^2_{H^s_{\gamma/2}} + \|g\|^2_{L^2_{s+\gamma/2}}
\Big)\,.
\end{equation*}
Therefore, to complete the proof of Proposition \ref{part1-prop2}, it remains to
show

\begin{lemm}\label{part1-J-1-upper}
Let $\gamma >-3$. Then we have
\[
J_1 \lesssim \|g\|^2_{H^s_{s+\gamma/2}} + \|g\|^2_{L^2_{s+\gamma/2}}\,.
\]
The same conclusion holds even if $\mu$ in $J_1$ is replaced by $\mu^\rho$ for any fixed
$\rho >0$.
\end{lemm}
\begin{proof}
As for Lemma \ref{part1-coer}, it follows from
Proposition \ref{part1-radja-equivalence-modified} that,
for suitable constants $C_1, C_2 >0$, we have
\begin{align}\label{part1-radja-triple-lower-2}
C_1J_1^{\Phi} -  C_2 \|g\|^2_{L^2_{s+\gamma/2}} &\leq J_1^{\widetilde \Phi} \notag \\
&
\leq 2 \iiint b(\cos\theta)
{\mu_*}{\la v_*\ra^{|\gamma|}}  \, \Big(\la v' \ra^{\gamma/2}g'-
\la v \ra^{\gamma/2} g\,\Big)^2
d\sigma dvdv_* \\
&\quad  + 2 \iiint b(\cos\theta)
{\mu_*}{\la v_*\ra^{|\gamma|}} \Big(
\la v \ra^{\gamma/2} -\la v' \ra^{\gamma/2}\Big)^2 |g|^2
d\sigma dvdv_* \notag\\
&= B_1+B_2, \notag
\end{align}
because $\widetilde \Phi(|v-v_*|) \sim \displaystyle \la v'-v_* \ra^\gamma \leq
{\la v' \ra^{\gamma}}{\la v_*\ra^{|\gamma|}}$ and $(a+b)^2 \leq 2(a^2 + b^2)$.

By the same argument for $A_2$ in the proof of Lemma \ref{part1-coer}, we get
$
B_2 \lesssim
\|g\|_{L^2_{s+\gamma/2}}\,.
$

To estimate $B_1$, we apply Theorem 2.1 of \cite{amuxy-nonlinear-3} about
the upper bound on the collision operator in the Maxwellian molecule case.
It follows from (2.1.9) of \cite{amuxy-nonlinear-3} with $(m, \alpha)=(-s, -s) $ that
\[
\Big| \Big(Q^{\Phi_0 }(F,G),G\Big)
 \Big | \lesssim \|F\|_{L^1_{s+2s} }\|G\|^2_{H^s_s}\,.
\]
Since $2 a(b-a)= - (b-a)^2 + (a^2 -b^2)$, we get
\begin{align*}
\Big( Q^{\Phi_0 }(F,G),G\Big) &=\iiint b F_* G( G' -G)\\
&= - \frac{1}{2}\iiint b F_*(G'-G)^2 + \frac{1}{2}\iint F_*\big ({G'}^2 -G^2\big )\,,
\end{align*}
and therefore
\begin{align*}
\Big |\iiint b F_*(G'-G)^2\Big| &\leq 2\Big| \Big(Q^{\Phi_0 }(F,G),G\Big)
 \Big | + \Big|\iint F_*\big ({G'}^2 -G^2\big )\Big|\\
&\lesssim \|F\|_{L^1_{3s} }\|G\|^2_{H^s_s} + \|F\|_{L^1} \|G\|_{L^2}^2 \,,
\end{align*}
where we have used the cancellation lemma from \cite{al-1} for the second term.
Choosing $F= \mu \la v \ra^{|\gamma|}$  and $G= \la v \ra^{\gamma/2} g$,
it follows that $B_1 \lesssim \|g \|^2_{H^s_{s+ \gamma/2}}$, completing the proof of the lemma.
\end{proof}

\subsection{Equivalence to the linearized operator}\label{part1-Section4}

We will now show that the Dirichlet form of the
linearized collision operator is equivalent to the square of the
non-isotropic norm, and therefore, the proof of Proposition \ref{part1-prop2.1}
will be given. Let us note that for the bilinear operator $\Gamma(\,\cdot,\,\cdot\,)$, for suitable functions $f, g$, one has
\begin{align*}
\Big(\Gamma(f,\, g),\,  h\Big)_{L^2} &= \iiint
b(\cos\theta) \Phi (|v-v_*|)
\sqrt{\mu_\ast}\, \big( f'_\ast g' - f_\ast g\big) h\\
&= \iiint b(\cos\theta) \Phi (|v-v_*|) \sqrt{\mu'_\ast}\, \big(f_\ast g - f'_\ast
g'\big) h'\,,
\end{align*}
and by adding these two lines, it follows that
\begin{equation}\label{part1-2.2.5}
\Big(\Gamma (f,\, g),\, h\Big)_{L^2} = \frac{1}{2} \iiint
b(\cos\theta) \Phi (|v-v_*|) \Big(f'_\ast g' - f_\ast g \Big) \Big(
\sqrt{\mu_\ast}\, h - \sqrt{\mu'_\ast}\, h' \Big)\,.
\end{equation}

The following lemma shows that $\cL_1$ dominates  $\cL$.
\begin{lemm}\label{part1-lemm2.1.1}
Under the conditions \eqref{part2-hyp-2} on the cross-section with $0<s<1$ and $\gamma\in\RR$, we have
\begin{equation}\label{part1-2.1.5}
\Big(\cL_1 g,\, g\Big)_{L^2}\geq \frac{1}{2} \Big(\cL g,\,
g\Big)_{L^2}\, .
\end{equation}
\end{lemm}
\begin{proof} {}By standard changes of variables, the following computations hold true
\begin{align*}
&\Big(\cL_1 g,\, g\Big)_{L^2} =-\Big(\Gamma(\sqrt\mu\,,
g),\, g\Big)_{L^2}\\
&= \frac{1}{2} \iiint B(|v-v_*|, \cos \theta) \Big((\mu'_*)^{1/2} g'-
(\mu_\ast)^{1/2} g \Big)^2dv_*d\sigma dv\\
&=\frac{1}{2} \iiint  B(|v-v_*|, \cos \theta)\Big((\mu')^{1/2} g'_*-
(\mu)^{1/2} g_* \Big)^2dv_*d\sigma dv\\
&= \frac{1}{4} \iiint B(|v-v_*|, \cos \theta) \\
&\quad \times \left\{ \Big((\mu'_*)^{1/2} g'-
(\mu_\ast)^{1/2} g \Big)^2+ \Big((\mu')^{1/2} g'_*- (\mu)^{1/2} g_*
\Big)^2\right\} \, ,
\end{align*}
and
\begin{align*}
&\Big(\cL g,\, g\Big)_{L^2} =-\Big(\Gamma(\sqrt\mu\,,
g)+\Gamma(g,\,\sqrt\mu\,),\, g\Big)_{L^2(\RR^3_v)}\\
&= \iiint B \,\Big((\mu_*)^{1/2} g- (\mu'_\ast)^{1/2}
g'+g_*(\mu)^{1/2}-
g'_*(\mu')^{1/2} \Big)(\mu_*)^{1/2}\, g
\\
&= \iiint B \,\Big((\mu'_*)^{1/2} g'- (\mu_\ast)^{1/2}
g+g'_*(\mu')^{1/2}- g_*(\mu)^{1/2} \Big)(\mu'_*)^{1/2}\, g'
\\
&= \iiint B\, \Big((\mu)^{1/2} g_*- (\mu')^{1/2}
g'_*+g(\mu_*)^{1/2}- g'(\mu'_*)^{1/2} \Big)(\mu)^{1/2}\, g_*
\\
&= \iiint B\, \Big((\mu')^{1/2} g'_*- (\mu)^{1/2}
g_*+g'(\mu'_*)^{1/2}- g(\mu_*)^{1/2} \Big)(\mu')^{1/2}\, g'_*\\
&=\frac{1}{4} \iiint B \left\{ \Big((\mu_\ast)^{1/2} g
-(\mu'_*)^{1/2} g'\Big)+ \Big((\mu)^{1/2} g_* -(\mu')^{1/2}
g'_*\Big)\right\}^2 .
\end{align*}
Therefore, (\ref{part1-2.1.5}) follows {}from $ (\alpha+\beta)^2 \leq
 2 (\alpha^2+\beta^2)$  and the proof is completed.
\end{proof}

Now for the term $\cL_2$, we have

\begin{lemm}\label{part1-radja-upper-estimate-L-2} One has
\begin{align*}
\left|\Big(\cL_2 g,\, h\Big)_{L^2}\right|
&\lesssim \| \mu^{1/10^3} g\|_{L^2}  \| \mu^{1/10^3} h\|_{L^2}   \, .
\end{align*}
\end{lemm}
\begin{proof}
It follows from \eqref{part1-2.2.5} that
\begin{align*}
\Big(\cL_2(g),\, h\Big)_{L^2} &=-  \frac{1}{2} \iiint
B \Big(g'_\ast \sqrt{\mu'} - g_\ast \sqrt \mu \Big) \Big(
\sqrt{\mu_\ast}\, h - \sqrt{\mu'_\ast}\, h' \Big)\,\\
&= \Big(g \,, \cL_2(h)\Big)_{L^2} \,,
\end{align*}
that is, $\cL_2$ is symmetric. Hence it suffices to show the lemma in the case
when $g=h$.
Putting $G = \sqrt \mu g$, we have
\begin{align*}
- \cL_2 g =& \mu^{-1/2} Q (G,\, {\mu} ) \\
= &\mu^{-1/2} \iint
b(\cos\theta)\Phi(|v-v_*|) G'_* \Big(\mu'\, - \mu\, \Big)dv_* d\sigma \\
&\qquad + \sqrt \mu  \iint b(\cos\theta)\Phi(|v-v_*|)
\Big( G'_* - G_* \Big) dv_* d\sigma \\
= &I_1(v) + I_2(v).
\end{align*}
Thanks to the cancellation lemma, we have
$
I_2(v) =
\sqrt {\mu(v)} \big( S * G)(v)$ with $S(v) \sim |v|^\gamma$, whence we have
\begin{align}\label{part1-radja-I-222}
\Big|(I_2,g)_{L^2}\Big|
&\lesssim \int_v\int_{v_*}|v-v_*|^\gamma\sqrt \mu\sqrt \mu_*|g||g_*|dvdv_*   \\
&\lesssim \int_v \int_{v_*} |v-v_*|^\gamma\Big \{  (\mu_*^{1/4}\mu^{1/4} g)^2 + (\mu^{1/4}\mu_*^{1/4}g_*)^2 \Big\} dv dv_*
\notag \\
&\lesssim \|\la v \ra^\gamma \mu ^{1/4} g\|^2_{L^2} \lesssim \|\mu^{1/8} g\|^2_{L^2} \,,\notag
\end{align}
by means of Lemma \ref{part1-tro-funda}.

Writing
\[
\mu'\, - \mu\, = \, \sqrt{\mu'}\big( \sqrt{\mu'} -\sqrt \mu\big)
+ \sqrt{\mu} \big( \sqrt{\mu'} -\sqrt \mu\big)
\]
and using  $\sqrt{\mu' \mu'_*}= \sqrt{\mu \mu_*}$\,, we have
\[
I_1(v)
=\iint
 b(\cos\theta)\Phi(|v-v_*|)g'_\ast \Big(\sqrt{\mu_\ast}\, + \sqrt{\mu'_\ast}\,
\Big)\, \,\Big(\sqrt{\mu'}- \sqrt{\mu}\Big)
dv_*d\sigma.
\]
Hence
\begin{align*}
\Big( I_1 ,\, g\Big)_{L^2} =&
 \iiint b(\cos\theta)\Phi(|v-v_*|)g'_\ast \Big(\sqrt{\mu_\ast}\, - \sqrt{\mu'_\ast}\,
\Big)\, \,\Big(\sqrt{\mu'}- \sqrt{\mu}\Big)
g dv dv_* d\sigma \,\\
&+2 \iiint b(\cos\theta)\Phi(|v-v_*|) G_*\Big(\sqrt{\mu}\, - \sqrt{\mu'}\,
\Big)\,  g' dv dv_* d\sigma \\
=& A_1 + A_2 \,,
\end{align*}
where we have used the change of variables $(v,v_*) \rightarrow (v',v'_*)$
for the second term.
We can write
\begin{align*}
A_1 &=
\iiint b(\cos\theta)\Phi(|v-v_*|)\Big({\mu_\ast}^{1/4}\, - {\mu'_\ast}^{1/4}\,
\Big)\,  \,\Big({\mu'}^{1/4}- {\mu}^{1/4}\Big)g'_\ast
g\, \\
&\qquad \qquad \qquad \times \Big({\mu_\ast}^{1/4}\, + {\mu'_\ast}^{1/4}\,
\Big)\,  \,\Big({\mu'}^{1/4}+{\mu}^{1/4}\Big) d \sigma dv dv_* \,.
\end{align*}
Since we have
\begin{align*}
|v'_*|^2 &\leq (|v'_* - v'| + |v'|)^2 \leq (\sqrt 2 |v_* -v'| + |v'|)^2\\
&\leq (\sqrt 2 |v_*| + (\sqrt 2 +1 )|v'|)^2 \leq 4|v_*|^2 + 2(\sqrt 2 +1 )^2|v'|^2,
\end{align*}
and in the same way,
$|v|^2 \leq 4|v'|^2 + 2(\sqrt 2 +1 )^2|v_*|^2$,
we get, by adding the two corresponding inequalities, that
$\mu_* \mu' \leq (\mu'_* \mu)^{1/(10+ 4 \sqrt 2)}$.
Moreover, we have $\mu'_*\mu' = \mu_* \mu \leq (\mu'_* \mu)^{1/5}$
because $|v'_*|^2 \leq (|v'_*-v| +|v|)^2 \leq (|v_*-v| +|v|)^2\leq 2|v_*|^2 + 8 |v|^2$.
Noticing that
\[
\Big|\Big({\mu_\ast}^{1/4}\, - {\mu'_\ast}^{1/4}\,
\Big)\,  \,\Big({\mu'}^{1/4}- {\mu}^{1/4}\Big)\Big|
\lesssim |v-v'_*|^2 \theta^2,
\]
we get 
\begin{align}\label{part1-radja-A-111}
|A_1| & \lesssim \iint |v-v'_*|^{\gamma +2}
\left\{ \int_0^{\pi/2} \theta^{1-2s} d \theta \right\} (\mu'_* \mu)^{1/80}
g'_\ast
g dv dv'_* \\
&\lesssim \iint |v-v'_*|^{\gamma}
(\mu'_* \mu)^{1/160}
g'_\ast
g dv dv'_* \lesssim \|\mu^{1/10^{3}} g\|^2_{L^2}\, , \notag
\end{align}
by an argument similar to the analysis of $I_1$.

For $A_2$, we use the regular change of variable $v \rightarrow v'$, and denote
its inverse transformation by $v'\rightarrow \psi_{\sigma}(v')=v$.
Then
\begin{align*}
A_2 &= 2 \iint \sqrt \mu_* g_* \left \{ \int_{\SS^2}
b\Big( \frac{\psi_{\sigma}(v')-v_*}{|\psi_{\sigma}(v')-v_*|}\cdot \sigma\Big)
\Phi(|\psi_{\sigma}(v')- v_*|)  \right. \\
& \qquad \qquad \left.\enskip \enskip \times
\Big(\sqrt{\mu(\psi_{\sigma}(v'))}\, - \sqrt{\mu(v'})\,
\Big)
\Big| \frac{\partial (\psi_{\sigma}(v'))}{
\partial (v')}\Big| d\sigma \right \} g(v') dv_* dv' \,. \notag 
\end{align*}
It follows from
the Taylor expansion that 
\begin{align*}
\sqrt{\mu(\psi_{\sigma}(v'))}\, - \sqrt{\mu(v')})
&= \Big(\nabla \sqrt{\mu}\Big )(v')\cdot \Big(\psi_{\sigma}(v')
- v' \Big)  \\
+ \int_0^1(1-\tau)& \Big(\nabla^2 \sqrt{\mu}\Big )(v'+ \tau
(\psi_{\sigma}(v') - v'))\enskip \Big(\psi_{\sigma}(v')
- v' \Big)^2 d \tau.
\end{align*}
Note that the integral with respect to $\sigma$ corresponding to the first order term
vanishes, by means of the symmetry on $\SS^2$. Putting $v'_{\tau} =
v' + \tau
(\psi_{\sigma}(v') - v')$, we have
$|v'|^2 \leq (|v'-v_*| + |v_*|)^2 \leq ((|v'_{\tau}-v_*| + |v_*|)^2
\leq 2|v'_{\tau}|^2 + 8|v_*|^2$, so that
\[
\left |    \sqrt {\mu(v_*)}\Big(\nabla^2 \sqrt{\mu}\Big )(v' + \tau
(\psi_{\sigma}(v') - v')) \right |
\lesssim (\mu(v_*) \, \mu(v'))^{1/12}\,.
\]
Since $|\psi_{\sigma}(v') - v'| \lesssim |v'-v_*|\theta$, we have
\begin{align*}
|A_{2}| &\lesssim \iint \left\{ \int^{\pi/2}_{0}
\theta^{1-2s} d\theta\right\} |v' - v_* |^{\gamma+ 2}
(\mu_* \, \mu')^{1/12}
 |g_*| \, |g'|
dv_* dv'\\
&\lesssim \iint |v' - v_* |^{\gamma} (\mu_* \, \mu')^{1/24}
 |g_*| \, |g'|dv_* dv'
\lesssim \| \mu^{1/10^3} g \|_{L^2}^2.
\end{align*}
Together with \eqref{part1-radja-I-222} and \eqref{part1-radja-A-111}, this yields the desired estimate
and completes the proof of the lemma.
\end{proof}

Let us note the following inequality between $\Big({ \cL}_1 g,\, g\Big)_{L^2}$, corresponding to the first term of the linear operator,
and the non-isotropic norm.

\begin{prop}\label{part1-radja-mouhot-strain} { Let $\gamma >-3$}.
There exists a constant $C>0$ such that
\begin{align}\label{part1-radja-L-1-lower}
|||g |||^2\geq \Big(\cL_1 g,\, g\Big)_{L^2} \geq  \frac{1}{10} |||g |||^2 - C
\|g\|_{L^2_{\gamma/2}}^2\,.
\end{align}
\end{prop}

\begin{proof}
The equalities
\begin{align*}
&2 \Big(\cL_1 g,\, g\Big)_{L^2}= - 2\Big(\Gamma(\sqrt{\mu},\,
g),\, g\Big)_{L^2} \notag \\
&= \iiint B\,\left(
(\mu'_*)^{1/2} g' - (\mu_*)^{1/2}\, g \right)^2 dv dv_* d \sigma\\
&= \iiint B\, \left( (\mu'_*)^{1/2} (g' -g)+
g((\mu'_*)^{1/2}\,-(\mu_*)^{1/2} \right)^2 dv dv_* d \sigma\notag ,
\end{align*}
together with the inequality
$$
2(a^2+b^2)\geq(a+b)^2 \geq \frac 1 2 a^2 -b^2$$
yields
\begin{align}\label{part1-first-step}
|||g |||^2\geq \Big(\cL_1 g,\, g\Big)_{L^2} \geq \frac 1 4 J_1 - \frac 1 2 J_2
 \geq \frac 1 4 |||g|||^2 - \frac{3}{4} J_2 \,.
\end{align}
It follows from the equality $(a+b)^2 = a^2 + b^2 +2ab$ that
\begin{align}\label{part1-radja-fundamental-formula}
2\Big(\cL_1 g,\, g\Big)_{L^2}
\geq J_2 - C \| g\|^2_{L^2_{\gamma/2}}\,,
\end{align}
which yields the desired estimate \eqref{part1-radja-L-1-lower} together with \eqref{part1-first-step}.

Indeed, note that
\begin{align*}
&2 \Big(\cL_1 g,\, g\Big)_{L^2}\\
&= \iiint B\, \left( (\mu'_*)^{1/2} (g' -g)+
g((\mu'_*)^{1/2}\,-(\mu_*)^{1/2} )\right)^2 dv dv_* d \sigma\\
&= J_1 + J_2 + 2 \iiint B\, (g'-g)g  (\mu '_\ast)^{1/2}
\Big((\mu'_*)^{1/2} -(\mu_*)^{1/2} \Big) dv dv_* d \sigma\,. \notag
\end{align*}
Using the identity $2 (\beta -\alpha) \alpha = \beta^2 -\alpha^2 -(\beta-\alpha)^2$,
we have
\begin{align*}
2(g'-g)g  &(\mu '_\ast)^{1/2}
\Big((\mu'_*)^{1/2} -(\mu_*)^{1/2} \Big)\\
= &\frac{1}{2}\Big (g'^2 -g^2 -(g'-g)^2 \Big)
\Big(\mu'_* - \mu_* +\big ((\mu'_*)^{1/2} -(\mu_*)^{1/2} \big)^2 \Big)\\
&=-\frac{1}{2}\Big(g' - g\Big)^2 \Big((\mu'_*)^{1/2} -(\mu_*)^{1/2} \Big)^2
+ \frac{1}{2} \big({g}^2-{g'}^2\,\big) \big(\mu_* - \mu'_* \big)\\
&+\frac{1}{2}\Big(g' - g\Big)^2\big(\mu_* - \mu'_* \big)
+ \frac{1}{2}\big({g'}^2-{g}^2\,\big)\Big((\mu'_*)^{1/2} -(\mu_*)^{1/2} \Big)^2\\
=&I_1 +I_2 +I_3 +I_4\,.
\end{align*}
Using the change of variables $(v', v'_\ast ) \rightarrow (v,v_\ast) $,
we see that
\[
\Big|\iiint B \, I_2 dvdv_*d\sigma\Big | = \Big|\iiint
B\,\mu_* \big(g^2 - g'^2\big) dv dv_* d\sigma \Big| \leq C \|g\|_{L^2_{\gamma/2}}^2\,,
\]
by means of the cancellation lemma.
Furthermore,
\begin{align*}
\iiint B \, I_1 dvdv_*d\sigma =&
-\frac{1}{2}\iiint B (\mu_* + \mu'_*)\Big(g' - g\Big)^2 dvdv_* d\sigma\\
&+\iiint B (\mu_*)^{1/2}(\mu'_*)^{1/2}\Big(g' - g\Big)^2 dvdv_* d\sigma
\geq -J_1 \,,
\end{align*}
where we have used the change of variables $(v', v'_\ast ) \rightarrow (v,v_\ast)$.
Thus, we obtain \eqref{part1-radja-fundamental-formula}
because the integrals corresponding to the last two terms $I_3$ and $I_4$
vanish, ending the proof of the proposition.
\end{proof}

{\bf End of the proof of Proposition \ref{part1-prop2.1}:} It follows from
\eqref{part1-radja-L-1-lower} and \eqref{part1-2.1.5} that
\[
|||g |||^2\geq \Big(\cL_1 g,\, g\Big)_{L^2} \geq \frac 12\Big(\cL g,\, g\Big)_{L^2} \,.
\]

On the other hand, note that $\Big(\cL g,\, g\Big)_{L^2} = \Big(\cL (\iI - \pP )g,\, (\iI -\pP )g\Big)_{L^2}$, from the very definition of the projection operator $\pP$.

Thus, from Proposition \ref{part1-radja-mouhot-strain} and Lemma \ref{part1-radja-upper-estimate-L-2}, we get
\begin{align*}
\Big(\cL g,\, g\Big)_{L^2}=\Big(\cL_1 (\iI - \pP )g,\, (\iI - \pP )g\Big)_{L^2}+\Big(\cL_2 (\iI - \pP )g,\, (\iI - \pP )g\Big)_{L^2}
\end{align*}
\begin{align*}
\geq  \frac {1}{ 10 }|||(\iI - \pP )g|||^2-C\|(\iI - \pP )g\|^2_{L^2_{\gamma/2}} \,.
\end{align*}
Since it is known from \cite{mouhot} that we have
\begin{align*}
\Big(\cL g,\, g\Big)_{L^2}
\geq  C\|(\iI - \pP )g\|^2_{L^2_{\gamma/2}} \,,
\end{align*}
we get on the whole
\[
|||(\iI - \pP ) g|||^2\leq C\Big(\cL g,\, g\Big)_{L^2} \,
\enskip \,.
\]


\subsection{Non-isotropic norms with different kinetic factors}\label{part1-Section5}

\smallskip
This subsection is devoted to the proof of Proposition \ref{part1-prop4}. That is, we will show some equivalence relations between the non-isotropic
norms with different kinetic factors and different weights.

For the proof, we introduce some further notations. Let $\rho >0$, $\mu_\rho(v) = \mu(v)^\rho$, and set
\[
J^{\Phi_\gamma}_{1, \rho}(g)=
\iiint \Phi_\gamma  (|v-v_*|)b(\cos\theta) \mu_{\rho, *}\, \big(g'-g\,\big)^2 dvdv_*d\sigma\,.
\]
We simply write $J^{\Phi_\gamma}_{1}(g)$ if $\rho=1$, and also  introduce the notation $J^{\Phi_\gamma}_{2,\rho}(g)$ similarly
with $\mu$ replaced by $\mu_\rho$.

Then
it follows from  \eqref{part1-J-2-equivalence} and the change of variables
$v \rightarrow v/ \sqrt \rho $ that
\[
J_{2,\rho}^{\Phi_\gamma} (g) \sim \|g\|^2_{L^2_{s+\gamma/2}} =
\|\la v \ra^{(\gamma-\beta)/2} g\|^2_{L^2_{s+ \beta/2}}
\sim J_{2,\rho}^{\Phi_{\beta}}(\la v \ra^{(\gamma-\beta)/2} g)\,.
\]
By the last assertions of Lemmas \ref{part1-coer}
and \ref{part1-J-1-upper}, there exist constants $C_1, C_2 >0$ such that
\begin{equation}\label{part1-rho-triple}
C_1 \|g\|^2_{H^s_{\gamma/2}}
\leq J_{1,\rho}^{\Phi_\gamma}(g) +  \|g\|^2_{L^2_{s+\gamma/2}}
\leq C_2 \|g\|^2_{H^s_{s+\gamma/2}}\,.
\end{equation}
Furthermore,
it follows from \eqref{part1-radja-triple-lower}, \eqref{part1-radja-triple-lower-2}
and the proofs of Lemmas \ref{part1-coer}
and \ref{part1-J-1-upper} that
\begin{align*}
J_{1,2}^{\Phi_0}(\la v \ra^{\gamma/2}g) \lesssim J_1^{\Phi_\gamma}(g)
\lesssim J_{1,1/2}^{\Phi_0}(\la v \ra^{\gamma/2}g), \qquad \mbox{modulo} \enskip
\|g|^2_{L^2_{s+\gamma/2}}\,,
\end{align*}
because we have $C_1 \mu_2 \leq \mu \la v \ra^{\pm |\gamma|} \leq C_2 \mu_{1/2}$.

Therefore, to complete the proof of Proposition \ref{part1-prop3}, it suffices to show
that for any $\rho, \rho' >0$
\begin{equation}\label{part1-equivalence-mu}
J_{1,\rho}^{\Phi_0}(g)
\sim
J_{1,\rho'}^{\Phi_0}(g), \qquad  \mbox{modulo} \enskip \|g\|^2_{L^2_{s}} \,.
\end{equation}
In fact, note that
\[
J_1^{\Phi_\gamma}(g) \sim J_{1,\rho}^{\Phi_0}(\la v \ra^{\gamma/2}g) \sim
J_{1}^{\Phi_\beta}(\la v \ra^{(\gamma-\beta)/2} g), \enskip \mbox{modulo}
\enskip \|g\|^2_{L^2_{s+\gamma/2}}.
\]
This equivalence looks  quite obvious, however, for completeness, we shall give a proof. In fact, \eqref{part1-equivalence-mu} is a direct consequence of the following lemma,
by taking $f= \mu_{\rho'}$.

\begin{lemm}\label{part1-imp-upper-maxwel}
Assume that \eqref{part2-hyp-2} holds with $0<s<1$. Then there
exists a constant $C>0$ such that
\[
\iiint b \, f_*^2 (g'-g)^2  d\sigma dvdv_*
\leq C
\|f\|^2_{L^2_s}
\Big( J_{1,\rho}^{\Phi_0}(g) + \|g\|^2_{L^2_s} \Big)\,.
\]
\end{lemm}
Once the equivalence \eqref{part1-equivalence-mu} has been established, we have
\begin{coro}\label{part1-imp-upper-maxwel-2}
Assume that \eqref{part2-hyp-2} holds with $0<s<1$. Then there
exists a constant $C>0$ such that
\[
\iiint b \, f_*^2 (g'-g)^2  d\sigma dvdv_*
\leq C
\|f\|^2_{L^2_s} |||g|||_{\Phi_0}^2\,.
\]
\end{coro}

\begin{proof}
It is enough to consider the case $\rho=1$.
As in the proof of Lemma \ref{part1-coer}, it follows from
Proposition 2 of \cite{al-1} that
\begin{align}\label{part1-radja-intermid}
J_1^{\Phi_0}(g)&=\iiint b(\cos\theta)  {\mu} _\ast
 ( g'-  g)^2dv_*d\sigma dv \notag \\
&=\frac{1}{(2\pi)^3}\iint b \left(\frac{\xi}{|\xi|}\cdot \sigma
\right) \Big(\widehat { \mu} (0) |\widehat { g} (\xi )|^2 +
|\widehat { g} (\xi^+)|^2  \notag \\
&\qquad \qquad \qquad  - 2
Re\,\, \widehat {\mu} (\xi^-)  \widehat
{  g} (\xi^+ ) \overline{\widehat { g}} (\xi )
\Big)d\xi d\sigma  \notag \\
&=\frac{1}{(2\pi)^3}\iint b \Big(\frac{\xi}{|\xi|}\cdot \sigma \Big)
\Big(\widehat {\mu} (0) | \widehat {g} (\xi ) -
\widehat { g} (\xi^+) |^2 \\
& \qquad \qquad \qquad + 2 Re\,
\Big(\widehat{ \mu }(0) - \widehat { \mu} (\xi^-) \Big)
 \widehat {g} (\xi^+ )
\overline{\widehat {g}} (\xi ) \Big)d\xi d\sigma , \notag
\end{align}
and
\begin{align*}
A = &\iiint b(\cos\theta) f^2_\ast (g'-  g)^2dv_*d\sigma dv\\
= &\frac{1}{(2\pi)^3}\iint b \left(\frac{\xi}{|\xi|}\cdot \sigma
\right) \Big(\widehat{f^2} (0) | \widehat {g }(\xi ) -
\widehat { g} (\xi^+) |^2 \\
&\qquad \qquad \qquad +
2 Re\, \Big(\widehat{f^2} (0) - \widehat{f^2} (\xi^-) \Big)
\widehat { g}
(\xi^+ ) \overline{\widehat { g}} (\xi ) \Big)d\xi d\sigma\, .
\end{align*}
Since $\widehat{f^2} (0)=\|f\|^2_{L^2}$ and
$\widehat { \mu} (0)=c_0 >0$,  we obtain
\begin{align*}
c_0 A &= c_0 \iiint b(\cos\theta) f^2_\ast ( g'- g)^2dv_*d\sigma dv\\
&=\|f\|^2_{L^2}
J_1^{\Phi_0}(g) \\
 &\qquad-\frac{2}{(2\pi)^3} \|f\|^2_{L^2} \iint b
\left(\frac{\xi}{|\xi|}\cdot \sigma \right)\,Re\, \Big(\widehat { \mu} (0)
- \widehat { \mu} (\xi^-) \Big) \widehat { g} (\xi^+ )
\overline{\widehat { g} } (\xi )
\Big)d\xi d\sigma \\
&\qquad\qquad+ \frac{2 c_0 }{(2\pi)^3}\iint b
\left( \frac{\xi}{|\xi|}\cdot \sigma \right) Re\, \Big(\widehat{f^2}
(0) - \widehat{f^2} (\xi^-) \Big) \widehat { g} (\xi^+ )
\overline{\widehat {g}}
(\xi ) \Big) d\xi d\sigma \\
&= \|f\|^2_{L^2}
J_1^{\Phi_0}(g) + A_{1} + A_2\,.
\end{align*}
Write
\begin{align*}
A_{2} = & \frac{2 c_0 }{(2\pi)^3}
\Big \{
\int |{\widehat { g}}
(\xi ) |^2 \Big(\int b
\left( \frac{\xi}{|\xi|}\cdot \sigma \right) Re\, \Big(\widehat{f^2}
(0) - \widehat{f^2} (\xi^-) \Big)  d\sigma \Big) d \xi \\
&+\iint b
\left( \frac{\xi}{|\xi|}\cdot \sigma \right) Re\, \Big(\widehat{f^2}
(0) - \widehat{f^2} (\xi^-) \Big) \Big(\widehat { g} (\xi^+ )
- \widehat {g} (\xi) \Big)
\overline{\widehat {g}}
(\xi )  d\xi d\sigma \Big \}\\
=& A_{2,1} + A_{2,2}\,.
\end{align*}

It follows from Cauchy-Schwarz's inequality that
\begin{align*}
|A_{2,2}| &\leq
C \Big( \iint b
\left( \frac{\xi}{|\xi|}\cdot \sigma \right)
|\widehat{f^2}
(0) - \widehat{f^2} (\xi^-)|^2 |
\widehat { g}|^2
(\xi )
d\xi d\sigma
\Big)^{1/2}\\
& \qquad \times \Big(\iint b
\left( \frac{\xi}{|\xi|}\cdot \sigma \right)
|\widehat {g} (\xi^+ )
- \widehat {g} (\xi) |^2 d\xi d\sigma
\Big)^{1/2}\\
& = B_1^{1/2} \times B_2^{1/2}\,.
\end{align*}
Since
\[
|\widehat{f^2}
(0) - \widehat{f^2} (\xi^-)|
\leq \int f^2(v)|1- e^{-iv\cdot \xi^-}|dv,
\]
we have
\begin{align*}
B_1 &\leq C \iiint |\widehat { g} (\xi) |^2
f^2(v) f^2(w)\\
& \quad \times
\Big( \int b
\left( \frac{\xi}{|\xi|}\cdot \sigma \right)( |1- e^{-iv\cdot \xi^-}|^2 +
|1- e^{-iw\cdot \xi^-}|^2) d\sigma \Big) dv dw d\xi\\
&\leq
C\| g\|^2_{H^s}\|f\|_{L^2}^2 \|f\|_{L^2_s}^2,
\end{align*}
because
\begin{align*}
&\int b
\left( \frac{\xi}{|\xi|}\cdot \sigma \right)|1- e^{-iv\cdot \xi^-}|^2 d\sigma\\
&\leq C\Big ( \int_0^{(\la v\ra \la \xi \ra)^{-1}}
\theta^{-1-2s} (|v ||\xi|)^2 \theta^2
d\theta +
\int_{(\la v \ra \la \xi \ra)^{-1}}^{\pi/2}
\theta^{-1-2s}
d\theta \Big)\\
&\leq C \la v \ra^{2s} \la \xi \ra^{2s}\,.
\end{align*}
Then we have 
$
|A_{2,1}| \leq C \|g\|_{H^s}^2 \|f\|_{L^2_s}^2 $
because
\begin{align*}
&\int b
\left( \frac{\xi}{|\xi|}\cdot \sigma \right) Re\, \Big(\widehat{f^2}
(0) - \widehat{f^2} (\xi^-) \Big)  d\sigma\\
&= \int f^2(v)  \Big(\int b
\left( \frac{\xi}{|\xi|}\cdot \sigma \right)\big (1- \cos (v\cdot \xi^-)\big)d\sigma
\Big)dv\\
& \leq C \la \xi \ra^{2s}\int f^2(v) \la v \ra^{2s}  dv \,.
\end{align*}
Since $\widehat { \mu}(\xi)$
is real-valued,
it follows that
\[
Re\, \Big(\widehat { \mu} (0)
- \widehat { \mu} (\xi^-) \Big) \widehat { g} (\xi^+ )
\overline{\widehat { g} } (\xi )
= \Big (\int
\big (1- \cos (v\cdot \xi^-)\big)
\mu(v) dv\Big)\,
Re\,\widehat { g} (\xi^+ )
\overline{\widehat { g} } (\xi )\,.
\]
Therefore, by using Cauchy-Schwarz's inequality and the change of variables
$\xi \rightarrow \xi^+$ ( see the proof of Lemma 2.8 in \cite{amuxy3}), we obtain
$
|A_{1}| \leq C \|f \|_{L^2}^2 \| g\|_{H^s}^2$.
Furthermore, it follows from \eqref{part1-radja-intermid} that
\begin{align*}
B_2 = &\iint b \Big(\frac{\xi}{|\xi|}\cdot \sigma \Big)
 | \widehat { g} (\xi ) -
\widehat {g} (\xi^+) |^2 d\xi d \sigma \\
\leq &C \Big( J_1^{\Phi_0}(g)   + \| g\|_{H^s}^2 \Big)
\,,
\end{align*}
which yields $|A_{2,2}| \leq C\|f\|_{L^2}\|f\|_{L^2_s}
\| g\|_{H^s}\Big( J_1^{\Phi_0}(g)   + \| g\|_{H^s}^2 \Big)^{1/2} $.
Hence
\[
|A_{2}| \leq C\|f\|_{L^2_s}^2
\| g \|_{H^s}\Big( J_1^{\Phi_0}(g)   + \| g\|_{H^s}^2 \Big)^{1/2}  \,.
\]
Finally, we have $$A \leq C \|f\|_{L^2_s}^2
\| g \|_{H^s}\Big( J_1^{\Phi_0}(g)   + \| g\|_{H^s}^2 \Big)^{1/2}
\leq C \|f\|_{L^2_s}^2\Big( J_1^{\Phi_0}(g)   + \| g\|_{L^2_s}^2 \Big),
$$ by means of \eqref{part1-rho-triple} with $\gamma=0$, completing the proof of the lemma.
\end{proof}


\section{Estimates of nonlinear collision operator in velocity space}\label{part2-section2}
\setcounter{equation}{0}

In this section, we derive various estimates on the nonlinear collision operator. Even though we consider the soft potential case in this paper,
some of the following estimates also hold for general case so that they will be used in Part II.

\subsection{Upper bounds in general case}\label{part2-section2.2}

In this sub-section, we will establish various functional estimates which hold true under
the more general assumption $0<s<1$ and $\gamma >-3$. In particular, all the results in this part are independent of the sign of $\gamma +2s$.

\begin{prop}\label{part2-prop5.1} For all $0< s<1$ and $\gamma >-3$, we have
\begin{align*}
\big | \big( \Gamma(f,g) ,h \big)_{L^2} \big|
\lesssim |||h|||_{\Phi_\gamma}
\Big\{ \|f\|_{L^\infty} |||g|||_{\Phi_\gamma} + \Big(
\|\nabla f \|_{L^\infty} +  \|f\|_{L^\infty}\Big)
\|g\|_{L^2_{ s+\gamma/2}}^2 \Big \} .
\end{align*}
\end{prop}

\begin{proof}
Direct calculation gives
\begin{align*}
&\Big( \Gamma(f,g) ,h \Big )_{L^2}  = \iiint b \Phi_\gamma \mu_*^{1/2} \Big(f'_*
g' - f_* g \Big) h dvdv_* d \sigma \\
=& \frac{1}{2}\iiint \Big( b  \Phi_\gamma\Big)^{1/2}{\mu'_*}^{1/4}  \Big(f'_*
g' - f_* g \Big) \Big( b { \Phi_\gamma} \Big)^{1/2}
\Big( {\mu_*}^{1/4} h - {\mu'_*}^{1/4} h'\Big) \\
& + \frac{1}{2} \iiint \Big( b  \Phi_\gamma\Big)^{1/2}{\mu_*}^{1/4}
\Big(f'_*
g' - f_* g \Big) \Big( b { \Phi_\gamma} \Big)^{1/2}
 \Big( {\mu_*}^{1/4} - {\mu'_*}^{1/4} \Big)h\,.
\end{align*}
Noticing that
\[
{\mu_*}^{1/4} h - {\mu'_*}^{1/4} h' =
{\mu'_*}^{1/4} \Big( h-h' \Big) +\Big( {\mu_*}^{1/4} - {\mu'_*}^{1/4} \Big)h\,,
\]
by using the Cauchy-Schwarz inequality, we have
\begin{align*}
\Big|
\Big( \Gamma(f,g) ,h \Big ) \Big|
\lesssim &  \, \Big( \iiint  b  \Phi_\gamma \mu_*^{1/2} \Big(f'_*
g' - f_* g \Big) ^2  d \sigma dv dv_*   \Big)^{1/2} \, ||| h |||_{\Phi_\gamma} \\
=&  A^{1/2} \, ||| h|||_{\Phi_\gamma} \,,
\end{align*}
where we have used the fact that the non-isotropic norm is invariant by replacing $\mu$ by $\mu^\rho$ for any fixed
$\rho>0$ (see  the previous section).
We then estimate
\begin{align*}
A & \leq 3 \Big(
\iiint  b  \Phi_\gamma \mu_*^{1/4} \Big( \big( \mu^{1/8} f\big)'_*
 - \big( \mu^{1/8} f\big) _*  \Big) ^2  g^2 d \sigma dv dv_* \\
& \qquad + \iiint  b  \Phi_\gamma \mu_*^{1/8} \Big( \big( \mu^{1/8} f\big)'_* \Big)^2\Big(
g' - g \Big) ^2 d \sigma dv dv_*  \notag \\
&\qquad + \iiint  b \Phi_\gamma \mu_*^{1/4}
\Big(\mu^{1/8}_* - {\mu'_*}^{1/8}\Big)^2
 \Big(f'_*
g'  \Big) ^2  d \sigma dv dv_* \Big) \notag \\
&= A_1 +A_2 +A_3 \,.\notag
\end{align*}

It is easy to see that
\[
A_2 +A_3
\lesssim \| f\|_{L^\infty}^2    |||g|||^2_{\Phi_\gamma} \,.
\]
Note that
\[
\Big|\big(\mu^{1/8} f\big)'_*
 - \big( \mu^{1/8} f\big) _*  \Big|
 \leq \min \Big\{
\Big( \|\nabla f \|_{L^\infty} +  \|f\|_{L^\infty}\Big) \theta |v-v_*|,
\|f\|_{L^\infty} \Big\} .
\]
Then
we have
\begin{align*}
A_1 &\lesssim
\Big( \|\nabla f \|_{L^\infty} +  \|f\|_{L^\infty}\Big)^2
\iint \Phi_\gamma \Big (
\int b(\cos \theta) \min(\theta^2 |v-v_*|^2\,, 1 ) d\sigma \Big)
\mu_*^{1/4} g^2 dv dv_* \\
&\lesssim \Big( \|\nabla f \|_{L^\infty} +  \|f\|_{L^\infty}\Big)^2\iint |v-v_*|^{\gamma+2s} \mu_*^{1/4} g^2 dv dv_*
\lesssim \Big( \|\nabla f \|_{L^\infty} +  \|f\|_{L^\infty}\Big)^2\|g\|^2_{s+\gamma/2}\,,
\end{align*}
where we have used $\gamma +2s >-3$ and the fact that
\begin{align*}
\int b(\cos \theta) \min(\theta^2 |v-v_*|^2\,, 1 ) d\sigma
&\leq |v-v_*|^2
\int_0^{\min(\pi/2, |v-v_*|^{-1})} \theta^{1-2s} d\theta \\
&
+ \int_{\min(\pi/2, |v-v_*|^{-1})}^{\pi/2} \theta^{-1-2s} d\theta
\lesssim |v-v_*|^{2s}\,.
\end{align*}
And this completes the proof of the proposition.
\end{proof}

\begin{lemm}\label{part1-imp-upper-hard}  Let $\gamma \ge 0$.
Assume that \eqref{part2-hyp-2} holds with $0<s<1$. Then
\[
\iiint \Phi_\gamma(|v-v_*|) b \, f_*^2 (g'-g)^2  d\sigma dvdv_*
\lesssim
\|f\|^2_{L^2_{s+\gamma/2}}
|||g|||_{\Phi_\gamma}^2 \,.
\]
\end{lemm}
\begin{proof}
Since $\Phi_\gamma(|v-v_*|) \lesssim \la v' \ra^\gamma + \la v_* \ra^\gamma$, we have
\begin{align*}
&\iiint b(\cos \theta) \Phi(|v-v_*|)f_*^2 (g' - g)^2 d\sigma dv dv_* \\
&\qquad \lesssim
 \iiint b(\cos \theta) f_*^2\Big(\la v' \ra^{\gamma/2}g'-
\la v \ra^{\gamma/2} g\,\Big)^2
d\sigma dvdv_* \\
& \qquad \qquad
+ \iiint b(\cos \theta) \Big(\la v_* \ra^{\gamma/2} f_*\Big)^2 (g' - g)^2 d\sigma dv dv_*\\
& \qquad \qquad
+ \iiint b(\cos \theta)f_*^2
\Big(
\la v \ra^{\gamma/2} -\la v' \ra^{\gamma/2}\Big)^2 |g|^2
d\sigma dvdv_*
\\
& \qquad = A_1 + A_2 + A_3\,.
\end{align*}
Noticing that
\begin{align*}
\Big|\la v \ra^{\gamma/2} -\la v' \ra^{\gamma/2}\Big|
&\leq C_\gamma \int_0^1 \la
v' + \tau(v-v') \ra^{(\gamma/2 -1)^+} d \tau |v-v_*|\theta\,\\
&\leq C'_\gamma \Big (\la v \ra^{(\gamma/2 -1)^+}
+ \la v_* \ra^{(\gamma/2 -1)^+} \Big) \, \la v-v_* \ra \theta,
\end{align*}
we have
\begin{align*}
A_3 &\lesssim \iint f_*^2 |g|^2\Big\{ \Big (\la v \ra^{(\gamma/2 -1)^+}
+ \la v_* \ra^{(\gamma/2 -1)^+} \Big)^2 \\
& \qquad \qquad \times \Big(
\int_0^{\la v -v_* \ra^{-1}} \theta^{-1-2s}  \Big(
\la v-v_* \ra \theta \Big)^2  d \theta \Big)\\
& \qquad \qquad +
\int_{\la v -v_* \ra^{-1}} ^{\pi/2}
\Big (\la v \ra^{\gamma/2}
+ \la v_* \ra^{\gamma/2} \Big)^2 \theta^{-1-2s} d\theta \Big\}
 dvdv_* \\
&\lesssim \iint   \Big (\la v \ra^{2s + \gamma}
+ \la v_* \ra^{2s + \gamma} \Big) f_*^2 |g|^2dvdv_* \\
&\lesssim \Big(\|f\|^2_{L^2_{s+\gamma/2}}\|g\|^2_{L^2} +
\|f\|^2_{L^2}\|g\|^2_{L^2_{s+\gamma/2}}\Big)\,.
\end{align*}
Applying Corollary \ref{part1-imp-upper-maxwel-2} to $A_1$ and $A_2$, it follows that
\begin{align*}
A_1 + A_2 &\lesssim \|f\|^2_{L^2_s}||| \la v \ra^{\gamma/2}g|||^2_{\Phi_0}
+ \| \la v \ra^{\gamma/2}f\|^2_{L^2_s}|||g|||^2_{\Phi_0}\\
&  \lesssim  \|f\|^2_{L^2_{s+\gamma/2}} ||| g|||^2_{\Phi_\gamma} \, ,
\end{align*}
where we have used Proposition \ref{part1-prop4} in the last inequality.
\end{proof}

\begin{prop}\label{part2-prop5.2} For all $0< s <1$ and $\gamma >-3$, one has
\begin{align*}
\big | \big( \Gamma(f,g) ,h \big )_{L^2}  \big|
\lesssim & |||h |||_{\Phi_\gamma}
\Big\{ \|f\|_{L^2_{s+\gamma/2}}|||g|||_{\Phi_\gamma}
+ \|g\|_{L^2_{s+\gamma/2}}|||f|||_{\Phi_\gamma}\notag\\
&\quad + \min \Big ( \|f\|_{L^2}\|g\|_{L^2_{s+\gamma/2}} \,,\,\|f\|_{L^2_{s+\gamma/2}}
\|g\|_{L^2}
\Big)\\
&+
{ \|\mu^{1/40}f\|_{L^2}
\|\mu^{1/60} \,  g\|_{H^{\max (-\gamma/2,\, 1)}} +   |||\mu^{1/40}f|||_{\Phi_\gamma}
 \|\mu^{1/40} \,g\|^2_{L^\infty} }
\Big \}\,.\notag
\end{align*}
\end{prop}

\begin{proof}
We will use the decomposition
\begin{equation*}
\Phi_\gamma (z) =|z|^\gamma{\bf 1}_{\{|z|\geq 1\}}+|z|^\gamma{\bf 1}_{\{|z|\leq 1\}}=\Phi_A(z)+\Phi_B(z).
\end{equation*}
We denote by $\Gamma_A(\cdot, \cdot), \Gamma_B(\cdot, \cdot)$ the collision operators with the kinetic factor
in the cross section given by $\Phi_A$ and $\Phi_B$ respectively.
Similarly to the proof of Proposition \ref{part2-prop5.1}, we have
\begin{align*}
\Big|
\Big( \Gamma_A(f, g),\, h \Big ) \Big|
\lesssim &  \, \Big( \iiint  b  \Phi_A \mu_*^{1/2} \Big(f'_*
g' - f_* g \Big) ^2  d \sigma dv dv_*   \Big)^{1/2} \, ||| h |||_{\Phi_\gamma} \\
=&  A^{1/2} \, ||| h|||_{\Phi_\gamma} \,.\,
\end{align*}
Since $\Phi_A\leq 2^{|\gamma|}\tilde\Phi_\gamma$, we have
\begin{align*}
A & \lesssim
\iiint  b  \tilde\Phi_\gamma \mu_*^{1/4} \Big( \big( \mu^{1/8} f\big)'_*
 - \big( \mu^{1/8} f\big) _*  \Big) ^2  g^2 d \sigma dv dv_* \\
& \qquad + \iiint  b  \tilde\Phi_\gamma \mu_*^{1/8} \Big( \big( \mu^{1/8} f\big)'_* \Big)^2\Big(
g' - g \Big) ^2 d \sigma dv dv_*  \notag \\
&\qquad + \iiint  b \tilde\Phi_\gamma \mu_*^{1/4}
\Big(\mu^{1/8}_* - {\mu'_*}^{1/8}\Big)^2
 \Big(f'_*
g'  \Big) ^2  d \sigma dv dv_*  \notag \\
&= A_1 +A_2 + A_3 \,.\notag
\end{align*}

In order to estimate $A_1$, we make use of Lemma \ref{part1-imp-upper-hard}.
Since $\tilde \Phi_\gamma(|v-v_*|) \mu_*^{1/4} \lesssim \la v \ra^\gamma$, we have
by putting $f= \mu^{1/8} f$ and $g= \la v \ra^{\gamma/2} g$,
\begin{align*}
A_1 &\lesssim
\iiint  b \big( \la v \ra^{\gamma/2} g \big)^2 \Big( \big( \mu^{1/8} f\big)'_*
 - \big( \mu^{1/8} f\big)_*  \Big) ^2  d \sigma dv dv_* \\
 & \lesssim \|\la v \ra^{\gamma/2} g \|^2_{L^2_s} |||\mu^{1/8} f|||_{\Phi_0}^2
 \lesssim \|g\|^2_{L^2_{s+\gamma/2}}|||f|||_{\Phi_\gamma}^2\,.
\end{align*}
We decompose the estimation on $A_2$ as
\begin{align*}
A_2 & \lesssim
\iiint  b \Big( \big( \mu^{1/8} f\big)'_* \Big)^2\Big(
\big(\la v \ra^{\gamma/2} g\big) ' -
  \big(\la v \ra^{\gamma/2} g\big) \Big) ^2 d \sigma dv dv_*  \\
  &\qquad
  + \iiint  b
\Big( \la v \ra^{\gamma/2} - \la v' \ra^{\gamma/2} \Big)^2
\Big( \big( \mu^{1/8} f\big)'_* \Big)^2 g'^2 d \sigma dv dv_*  \\
&= A_{2,1} + A_{2,2}\,.
\end{align*}
Applying again  Lemma \ref{part1-imp-upper-hard} to $A_{2,1}$, we get
\[
A_{2,1} \lesssim \|\mu^{1/8}f \|_{L^2_s}^2 |||\la v \ra^{\gamma/2}g|||_{\Phi_0}^2
\lesssim \|f\|^2_{L^2_{s+\gamma/2}} |||g|||_{\Phi_\gamma}^2\,.
\]
For $A_{2,2}$, we note that
if $v_\tau = v' + \tau(v-v')$ for $\tau \in [0,1]$, then
\[
\la v \ra \leq \la v -v_*\ra + \la v_* \ra \leq \sqrt 2\la v_\tau -v_*\ra + \la v_*\ra
\leq (1+\sqrt 2) \la v_\tau \ra \la v_* \ra,
\]
and $\la v_\tau \ra \leq (1+\sqrt 2) \la v \ra \la v_* \ra$. Thus,
$\la v_\tau \ra^\beta \leq C_\beta \la v \ra^\beta \la v_* \ra^{|\beta|}$ for any $\beta
\in \RR$. Since it follows that
\begin{align*}
\Big|\la v \ra^{\gamma/2} -\la v' \ra^{\gamma/2}\Big|
&\leq C_\gamma \int_0^1 \la
v' + \tau(v-v') \ra^{(\gamma/2 -1)} d \tau |v-v_*|\theta\,\\
&\leq C'_\gamma \Big (\la v \ra^{(\gamma/2 -1)}
\la v_* \ra^{|\gamma/2 -1|} \Big) \, \la v-v_* \ra \theta ,
\end{align*}
we have, by using the change of variables $ (v',v'_*)\rightarrow
(v,v_*)$,
\begin{align*}
A_{2,2} &\lesssim \iint \frac{\big( \mu^{1/8} f\big)_*^2}{\la v_*\ra^{|\gamma|}}|g|^2
\Big\{ \Big (\la v \ra^{(\gamma -2)}
\la v_* \ra^{|\gamma-2|} \Big)^2 \Big(
\int_0^{\la v -v_* \ra^{-1}} \theta^{-1-2s}  \Big(
\la v-v_* \ra \theta \Big)^2  d \theta \Big)\\
& \qquad +
\int_{\la v -v_* \ra^{-1}} ^{\pi/2}
\Big (\la v \ra^{\gamma}
+ \la v \ra^\gamma \la v_* \ra^{|\gamma|} \Big)\,\theta^{-1-2s} d\theta \Big\}
 dvdv_* \\
&\lesssim \iint   \Big (\la v \ra^{2s + \gamma}
\la v_* \ra^{2s+ \max(|\gamma-2|-|\gamma|, 0)} \Big) \big( \mu^{1/8} f\big)_*^2|g|^2dvdv_*
\lesssim \|\mu^{1/10} f\|^2_{L^2}\|g\|^2_{L^2_{s+\gamma/2}}\,.
\end{align*}
Noticing  that
$
\Big(\mu^{1/8}_* - {\mu'_*}^{1/8}\Big)^2 \lesssim \min (|v-v_*|^2 \theta^2, 1),
$
we have
\begin{align*}
A_3 &\lesssim
\iint  \tilde \Phi_\gamma
  \Big( \int_{\SS^2} b(\cos \theta) \min (|v-v_*|^2 \theta^2, 1) d \sigma \Big)
f^2_*
g^2   dv dv_* \\
&\lesssim \iint  \la v-v_* \ra^{\gamma +2s}
f^2_*
g^2   dv dv_* \\
&\lesssim \iint
\la v_*\ra^{\gamma +2s}f_*^2  \la v \ra^{\gamma +2s}g^2  dv dv_*
\lesssim \|f\|^2_{L^2_{s+\gamma/2}}\|g\|^2_{L^2_{s+\gamma/2}}\,,
\end{align*}
when $\gamma+2s \geq 0$ because
$\la v-v_*\ra^{\gamma+2s} \leq \la v_*\ra^{\gamma+2s} \la v\ra^{\gamma+2s}$.

To consider the case $\gamma +2s <0$, we divide the space
$\RR_v^3 \times \RR_{v_*}^3$ into three parts
\begin{align*}
&U_1 =
\{|v-v_*| \leq |v_*|/8\}\,, \enskip U_2= \{|v-v_*| > |v_*|/8\}\cap \{|v_*|\leq 1\}\,,\\
&U_3= \{|v-v_*| > |v_*|/8\}\cap \{|v_*|> 1\}\,.
\end{align*}
Then we have
\begin{align*}
\frac{1}{3} A_3 &=  \iiint  b \tilde \Phi_\gamma {\mu'}_*^{1/4}
\Big( {\mu'}_*^{1/8} -\mu^{1/8}_*  \Big)^2
 \Big(f_*
g  \Big) ^2 d\sigma dv dv_*\\
&= \iint_{U_1} \int d\sigma dv dv_* + \iint_{U_2} \int d\sigma dv dv_*
+\iint_{U_3} \int d\sigma dv dv_* \\
&= A_{3,1} + A_{3,2} + A_{3,3} \,.
\end{align*}
Since
$|v'-v_*| \leq |v-v_*| \leq |v_*|/8$, we have $7|v_*|/8 \leq |v'|, |v| \leq 9|v_*|/8$
and $|v'_*|^2 =|v|^2+|v_*|^2 -|v'|^2 \geq |v_*|^2/2$ in $U_1$.
Hence, in this region,
 we have ${\mu'}_*^{1/4} \leq C \mu_*^{1/8} \leq C ( {\mu_*} \mu )^{1/20}$,
and this  to
\begin{align*}
A_{3,1} \lesssim \iint (\mu {\mu_*})^{1/20}\la v-v_* \ra^{\gamma +2s} f_*^2 g^2 dv dv_* \lesssim \|f\|^2_{L^2_{s+\gamma/2}}
\|g\|^2_{L^2_{s+\gamma/2}}\,.
\end{align*}
Furthermore, we have
\[
A_{3,2} \lesssim \iint_{U_2} \la v-v_* \ra^{\gamma +2s} f_*^2 g^2 dv dv_*\lesssim
\|f\|^2_{L^2_{s+\gamma/2}}
\|g\|^2_{L^2_{s+\gamma/2}} ,
\]
because $\la v -v_* \ra^{-1} \leq  \la v \ra^{-1}\la v_* \ra^{-1} \la v_* \ra^2
\leq 2 \la v \ra^{-1}\la v_* \ra^{-1}$ in $U_2$.
Since $\la v-v_* \ra^{-1} \leq 8|v_*|^{-1}\leq 16 \la v_* \ra^{-1}$ in $U_3$, we get
\[
A_{3,3} \lesssim
\|f\|^2_{L^2_{s+\gamma/2}}
\|g\|^2_{L^2}\,.
\]
Therefore, we have, when $\gamma+2s \le 0$
\[
A_3 \lesssim \|f\|^2_{L^2_{s+\gamma/2}}
\|g\|^2_{L^2}\,.
\]
By considering another partition in $\RR^6_{v,v_*}$ with $v$ and $v_*$ exchanged,
 the estimate $A_3 \lesssim \|f\|^2_{L^2}
\|g\|^2_{L^2_{s+\gamma/2}}$ holds, because $|v_*'-v| \leq |v_*-v| \leq |v|/8$ implies $7|v|/8 \leq |v'_*|, |v_*| \leq 9|v|/8$.

As a conclusion, we have in summary that
\begin{align*}
&\Big|\Big(\Gamma_A(f,g),h\Big)\Big|
\lesssim \Big\{ \|f\|_{L^2_{s+\gamma/2}}|||g|||_{\Phi_\gamma}
+ \|g\|_{L^2_{s+\gamma/2}}|||f|||_{\Phi_\gamma}\\
&\quad + \min \big ( \|f\|_{L^2}\|g\|_{L^2_{s+\gamma/2}} \,,\,\|f\|_{L^2_{s+\gamma/2}}
\|g\|_{L^2}
\big)\Big\}|||h|||_{\Phi_\gamma} \, .
\end{align*}

\smallskip

We now turn to  $\Gamma_B$. For this, firstly, it holds that
\begin{align*}
\Big|
\Big( \Gamma_B(f, g),\, h \Big ) \Big|
\lesssim &  \, \Big( \iiint  b  \Phi_B \mu_*^{1/2} \Big(f'_*
g' - f_* g \Big) ^2  d \sigma dv dv_*   \Big)^{1/2} \, ||| h |||_{\Phi_\gamma} \\
=&  B^{1/2} \, ||| h|||_{\Phi_\gamma} \,.\,
\end{align*}
Since $|v-v_*| \leq 1$ implies
$|v|^2 \leq 2 + 2|v_*|^2$ and then $\mu_* \leq e \mu^{1/2}$, we have
\begin{align*}
B \lesssim  & \iiint_{\{|v-v_*|\leq 1\}}  b \Phi_\gamma \mu_*^{1/10} \mu^{1/10} \Big( \big( \mu^{1/8} f\big)'_*
 - \big( \mu^{1/8} f\big) _*  \Big) ^2  g^2 d \sigma dv dv_* \\
& \qquad + \iiint_{\{|v-v_*|\leq 1\}}   b  \Phi_\gamma \mu_*^{1/10}
\mu^{1/10} \Big( \big( \mu^{1/8} f\big)'_* \Big)^2\Big(
g' - g \Big) ^2 d \sigma dv dv_*  \\
&\qquad + \iiint_{\{|v-v_*|\leq 1\}}   b  \Phi_\gamma \mu_*^{1/10} \mu^{1/10}
\Big(\mu^{1/8}_* - {\mu'_*}^{1/8}\Big)^2
 \Big(f'_*
g'  \Big) ^2  d \sigma dv dv_*  \\
& = B_1 + B_2 +B_3\,.
\end{align*}
Obviously,
\[
B_1 \lesssim \|\mu^{1/20} g\|^2_{L^\infty}|||\mu^{1/8}f|||_{\Phi_\gamma}
\lesssim \|\mu^{1/20} g\|^2_{L^\infty}|||f|||_{\Phi_\gamma}
\,.
\]
Since
$|\mu^{1/8}_* - {\mu'_*}^{1/8}| \lesssim |v-v_*| \theta$,
we see that by the change of variables $(v',v'_*),  \rightarrow (v,v_*)$
\begin{align*}
B_3 \lesssim &
\iiint_{\{|v-v_*|\leq 1\}}   b  \Phi_\gamma \mu_*^{1/10} \mu^{1/10}
\Big({\mu'}^{1/8}_* - {\mu_*}^{1/8}\Big)^2
 \Big(f_*
g  \Big) ^2  d \sigma dv dv_* \\
\lesssim  &  \iint_{\{|v-v_*|\leq 1\}} \big(\mu_*^{1/20}f_* \big )^2 \big(
\mu^{1/20}g\big)^2 |v-v_*|^{\gamma+2} \Big ( \int
b(\cos \theta) \theta^2 d\sigma\Big) dvdv_* \\
\lesssim &  \int
\big(\mu_*^{1/20}f_* \big )^2 \Big(
\sup_{v_*} \int \frac{\big(
\mu^{1/20}g\big)^2 }{
 |v-v_*|^{(-\gamma-2)^+}} dv \Big) dv_*\\
\lesssim & \
\|\mu^{1/20}f\|^2_{L^2_{s+\gamma/2}} \||D_v|^{(-\gamma/2 -1)^+} \mu^{1/20} g \|^2_{L^2}\, ,
\end{align*}
where we have used the Hardy inequality if $\gamma+2<0$, cf. \cite{taylor}.
If one writes
\begin{align*}
B_2 \lesssim &   \iiint_{\{|v-v_*|\leq 1\}}   b  \Phi_\gamma \mu_*^{1/10}
\mu^{1/20} \Big( \big( \mu^{1/8} f\big)'_* \Big)^2\Big(
(\mu^{1/40}g)' - (\mu^{1/40}g) \Big) ^2 d \sigma dv dv_* \\
& +
\iiint_{\{|v-v_*|\leq 1\}}   b  \Phi_\gamma \mu_*^{1/10}
\mu^{1/20} \Big( \big( \mu^{1/8} f\big)'_* \Big)^2 g'^2 \Big(
\mu^{1/40} - {\mu'}^{1/40} \Big) ^2  d \sigma dv dv_*   \\
&= B_{2,1} + B_{2,2}\,,
\end{align*}
then the second term $B_{2,2}$ has a similar upper
bound as   $B_3$.
It remains to estimate
\[
B_{2,1} =
\iiint_{\{|v-v_*|\leq 1\}}   b  \Phi_\gamma \mu_*^{1/10}
\mu^{1/20} \Big( \big( \mu^{1/8} f\big)_* \Big)^2\Big(
(\mu^{1/40}g) - (\mu^{1/40}g)' \Big) ^2 d \sigma dv dv_*,
\]
by the change of variables $(v',v'_*) \rightarrow (v,v_*)$.
By firstly putting $F = \mu^{1/8} f$ and $G = \mu^{1/40}g$, and denoting by $v_\tau =v' + \tau(v-v')$ for $\tau \in [0,1]$,
then by using
\[
|G(v) - G(v')|^2 = \Big|\int_0^1 \nabla G(v_\tau) \cdot  (v-v') d \tau\Big|^2
\leq |v-v_*|^2 (\sin^2 \theta/2) \Big(\int_0^1 |\nabla G(v_\tau)|^2 d \tau\Big) ,
\]
we have
\begin{align*}
B_{2,1} \lesssim &\,  \int_0^1\Big\{
\iiint b |v-v_*|^{\gamma+2} \sin^2 \frac{\theta}{2}
F_*^2    |\nabla G(v_\tau)|^2
dv dv_* d\sigma \Big\}d \tau\,.
\end{align*}
To estimate this term, we need the change of variables
\begin{equation*}
v \to
v_\tau =\frac{1+\tau}{2}v+\frac{1-\tau}{2}(|v-v_*|\sigma+v_*).
\end{equation*}
The Jacobian of this transform
 is  bounded from below uniformly in $v_*,\,\sigma$ and $\tau$, because
\begin{align*}
\Big|\frac{\partial (v_\tau)}{\partial (v)}\Big|&=\Big|\text{det}\Big(
\frac{1+\tau}{2}I+\frac{1-\tau}{2}\sigma \otimes {\bf k} \Big)\Big|
\quad\qquad
 ({\bf k }=\frac{v-v_*}{|v-v_*|})
\\
& =\frac{(1+\tau)^3}{2^3} \Big|1+\frac{1-\tau}{1+\tau} {\bf k}\cdot\sigma
\Big| = \frac{(1+\tau)^3}{2^3}
\Big|\frac{2\tau}{1+\tau}+2\frac{1-\tau}{1+\tau}
\cos^2\frac{\theta}{2}\Big|
\\
&\ge \frac{(1+\tau)^3}{2^3}\Big|\frac{2\tau}{1+\tau}+
\frac{1-\tau}{1+\tau} \Big|=\frac{(1+\tau)^3}{2^3}\ge \frac{1}{2^3}.
\end{align*}
If we set
$\displaystyle
\tilde b = b\Big({\bf k} \cdot \sigma\Big ) \Big (1 - {\bf k} \cdot \sigma \Big)
$, then we have $\int_{\SS^2} \tilde b d \sigma < \infty$. Therefore,
\begin{align*}
B_{2,1}
&\lesssim
\int_0^1  \int
F_*^2  \Big\{ \int_{\SS^2} \tilde b \,\,
 \int \frac{|\nabla G(v_\tau)|^2}
{|v_\tau -v_*|^{-\gamma-2}} dv_\tau d \sigma \Big \} dv_* d\tau \\
&\lesssim
\int_0^1  \int <v_*>^{(\gamma/2+1)^+}
F_*^2  \Big\{ \int_{\SS^2} \tilde b \,\,
 \int \frac{<v_\tau>^{(\gamma/2+1)^+} |\nabla G(v_\tau)|^2}
{|v_\tau -v_*|^{(-\gamma-2)^+}} dv_\tau d \sigma \Big \} dv_* d\tau \\
&\lesssim \|F\|_{L^2_{(\gamma/2+1)^+}}^2 \|\, |D|^{(-\gamma/2 -1)^+}\nabla G\|^2_{L^2_{(\gamma/2+1)^+}}
\lesssim \|F\|^2_{L^2_{(\gamma/2+1)^+}}
\| G\|^2_{H^{\max(-\gamma/2,\, 1)}_{(\gamma/2+1)^+}}\,,
\end{align*}
where we have used
$|v-v_*| \sim |v_\tau -v_*|$.
Finally we obtain
\[
B \lesssim
|||\mu^{1/40} f |||_{\Phi_\gamma} \|\mu^{1/40}g\|_{L^\infty}
+ \|\mu^{1/40} f\|_{L^2}\|\mu^{1/60} g\|_{H^{\max(-\gamma/2,\,1)}} .
\]
This concludes the proof of Proposition \ref{part2-prop5.2}.
\end{proof}

Note that the above estimation is good enough for proving the
local existence for
 the general case. However, the above upper bound related to $B$
is given in Sobolev space with positive index and
this can not be controlled by the non-isotropic norm. Hence, it is not
sufficient for the proof of global existence.

\subsection{A simple proof of Theorem \ref{theorem-1.1-b} for $\gamma>-3/2$}\label{part1-Section6}

We first give a simple proof of upper bound estimates on the Boltzmann nonlinear operator when $\gamma>-\frac 32$. We state it as
\begin{prop}\label{part1-prop5}
Assume that $0<s<1$ and $\gamma>-3/2$.
Then
\begin{align*}
&\Big|\Big(\Gamma(f,g),h\Big)\Big|
\lesssim \Big\{\|f\|_{L^2_{s+\gamma/2}}|||g|||_{\Phi_\gamma}
+ \|g\|_{L^2_{s+\gamma/2}}|||f|||_{\Phi_\gamma}\\
&\quad { + \min \Big ( \|f\|_{L^2}\|g\|_{L^2_{s+\gamma/2}} \,,\,\|f\|_{L^2_{s+\gamma/2}}\|g\|_{L^2}
\Big)}\Big\}|||h|||_{\Phi_\gamma} \, .
\end{align*}
Furthermore, together with $\gamma \ge -3s$, one has
$$
\Big|\Big(\Gamma(f,g),h\Big)\Big|
\lesssim \Big\{\|f\|_{L^2_{s+\gamma/2}}|||g|||_{\Phi_\gamma}
+ \|g\|_{L^2_{s+\gamma/2}}|||f|||_{\Phi_\gamma}
\Big\}|||h|||_{\Phi_\gamma} \, .
$$

\end{prop}

Let us note that the first statement deals with general values of $\gamma>-3/2$, that is not necessarily linked with the value of $s$. For the second statement, note that the condition $\gamma \ge -3s$ is always true in the physical cases mentioned above. Indeed recall here that $\gamma =1-4s$, and that $0< s <1$. Therefore, we can conclude that together with the constraint $\gamma >-3/2$, the physical range $0<s < 5/8$ is allowed.

\begin{proof}
\noindent
{\bf The case when $\gamma \ge 0$.} Note that
\begin{align}\label{part1-elementary}
\Big(\Gamma (f,\, g) ,\, h\Big)_{L^2} &= \Big(\mu^{-1/2}Q
(\mu^{1/2} f,\, \mu^{1/2} g ),\,  h\Big)_{L^2}\\
&= \iiint \Phi_\gamma b(\cos\theta) \mu^{1/2}_\ast \Big( f'_\ast g' - f_\ast g
\Big)\,\,h \notag \\
& =\frac 12 \iiint \Phi_\gamma b(\cos\theta) \Big(f'_\ast g' - f_\ast g\Big)
\Big( \mu^{1/2}_\ast h - \mu^{1/2 '}_\ast h' \Big) \notag \\
&\leq \frac 12 \left(\iiint \Phi_\gamma b(\cos\theta) \Big(f'_\ast g'- f_\ast
g\Big)^2 \right)^{1/2} \notag \\
&\quad \times \left( \iiint \Phi_\gamma b (\cos\theta) \Big((\mu_*)^{1/2}
h - (\mu'_*)^{1/2} h' \Big)^2 \right)^{1/2} \notag\\
&\leq \frac 12  A^{1/2}\times B^{1/2}. \notag
\end{align}
For $B$, we have
\begin{align*}
B& = \iiint \Phi_\gamma b (\cos\theta) \Big( (\mu'_*)^{1/2} (h'-h) + h
\big((\mu'_*)^{1/2} - (\mu_*)^{1/2}\big)\Big)^2\\
&\leq 2 \iiint \Phi_\gamma b(\cos\theta) \left\{\mu'_\ast (h'-h)^2 +  h^2
\Big((\mu'_*)^{1/2} -(\mu_*)^{1/2} \Big)^2 \right\}
\\
&\leq 2 \iiint \Phi_\gamma b(\cos\theta) \mu_\ast (h'-h)^2 +  2 \iiint \Phi_\gamma
b(\cos\theta) h^2_\ast \Big((\mu')^{1/2} -\mu^{1/2} \Big)^2 \\
& = 2 |||
h|||^2_{\Phi_\gamma},
\end{align*}
where we have used the change of variables $(v,v_\ast )\rightarrow (v',v'_\ast )$ for the first
term and $(v,v_\ast )\rightarrow (v_*,v)$ for the second term. Similarly,
\begin{align*}
A& = \iiint \Phi_\gamma b (\cos\theta) \Big(  f'_\ast (g'-g) + g (f'_\ast
-f_\ast )\Big)^2\\
&\leq 2 \iiint\Phi_\gamma b(\cos\theta) \left\{ {f'_*}^{2} (g'-g)^2 + g^2
(f'_\ast -f_\ast )^2 \right\}\\
&\leq 2 \iiint \Phi_\gamma b(\cos\theta) {f_*}^{2} (g'-g)^2+ 2 \iiint \Phi_\gamma
b(\cos\theta) {g_*}^{2} (f'-f)^2.
\end{align*}
Then \eqref{part1-imp-upper-hard} implies that
$$
A\lesssim || f||^2_{L^2_{s+\gamma/2}} \,
||| g|||^2_{\Phi_\gamma}+ || g||^2_{L^2_{s+\gamma/2}}\,
||| f|||^2_{\Phi_\gamma}\, ,
$$
which completes the proof in the case when $\gamma \ge 0$.

\noindent
{\bf The case when $-3/2 < \gamma < 0$.} As in Subsection \ref{part1-Section5},
it is easy to check that for any fixed $\rho>0$,
\begin{align}\label{part1-mu-invariant-norm}
|||g|||^2_{\Phi_\gamma} &\sim J_{1, \rho}^{\Phi_\gamma}(g)
+ J_{2, \rho}^{\Phi_\gamma}(g) \sim J_{1, \rho}^{\Phi_\gamma}(g)
+ \|g\|^2_{L^2_{s+\gamma/2}}\\
&\sim \iiint \frac{\Phi_{2\gamma}}{\tilde \Phi_{\gamma}}
b \mu_{\rho,*} (g' -g )^2 + \iiint \frac{\Phi_{2\gamma}}{\tilde \Phi_{\gamma}}
b g_*^2\Big( \sqrt { \mu'_\rho } - \sqrt{\mu_\rho} \Big)^2\,, \notag
\end{align}
where the assumption $2\gamma > -3$ is required
for the existence of the above integral, and more precisely for
\[\int |v_*|^{2\gamma} \la v_* \ra^{2s-\gamma} \mu_\rho(v+v_*)dv_* \sim
\la v \ra^{\gamma+2s}\,.
\]
Instead of \eqref{part1-elementary}, we write
\begin{align*}
\Big( \Gamma(f,g) ,h \Big ) =& \iiint b \Phi_\gamma \mu_*^{1/2} \Big(f'_*
g' - f_* g \Big) h dvdv_* d \sigma \\
=& \frac{1}{2}\iiint \Big( b \tilde \Phi_\gamma\Big)^{1/2} \Big(f'_*
g' - f_* g \Big) \Big( b \frac{\Phi_{2\gamma}}{\tilde \Phi_\gamma} \Big)^{1/2}
{\mu'_*}^{1/4} \Big( {\mu_*}^{1/4} h - {\mu'_*}^{1/4} h'\Big) \\
& + \frac{1}{2} \iiint \Big( b \tilde \Phi_\gamma\Big)^{1/2}
\Big(f'_*
g' - f_* g \Big) \Big( b \frac{\Phi_{2\gamma}}{\tilde \Phi_\gamma} \Big)^{1/2}
{\mu_*}^{1/4} \Big( {\mu_*}^{1/4} - {\mu'_*}^{1/4} \Big)h\,.
\end{align*}
Noticing that
\[
{\mu_*}^{1/4} h - {\mu'_*}^{1/4} h' =
{\mu'_*}^{1/4} \Big( h-h' \Big) +\Big( {\mu_*}^{1/4} - {\mu'_*}^{1/4} \Big)h\,,
\]
by Cauchy-Schwarz's inequality and \eqref{part1-mu-invariant-norm}, we have
\begin{align*}
\Big|
\Big( \Gamma(f,g) ,h \Big ) \Big|
\lesssim &  \, \Big( \iiint  b \tilde \Phi_\gamma \mu_*^{1/2} \Big(f'_*
g' - f_* g \Big) ^2  d \sigma dv dv_*   \Big)^{1/2} \, ||| h |||_{\Phi_\gamma} \\
=&  A^{1/2} \, ||| h|||_{\Phi_\gamma} \,.
\end{align*}
We estimate
\begin{align*}
A & \leq 3 \Big(
\iiint  b \tilde \Phi_\gamma \mu_*^{1/4} \Big( \big( \mu^{1/8} f\big)'_*
 - \big( \mu^{1/8} f\big) _*  \Big) ^2  g^2 d \sigma dv dv_* \\
& \qquad + \iiint  b \tilde \Phi_\gamma \mu_*^{1/8} \Big( \big( \mu^{1/8} f\big)'_* \Big)^2\Big(
g' - g \Big) ^2 d \sigma dv dv_*  \\
&\qquad + \iiint  b \tilde \Phi_\gamma \mu_*^{1/4}
\Big(\mu^{1/8}_* - {\mu'_*}^{1/8}\Big)^2
 \Big(f'_*
g'  \Big) ^2  d \sigma dv dv_* \Big) \\
&= A_1 +A_2 +A_3 \,.
\end{align*}
Since $\tilde \Phi_\gamma(|v-v_*|) \mu_*^{1/4} \lesssim \la v \ra^\gamma$, 
we have by means of Corollary \ref{part1-imp-upper-maxwel-2}
\begin{align*}
A_1 &\lesssim
\iiint  b \big( \la v \ra^{\gamma/2} g \big)^2 \Big( \big( \mu^{1/8} f\big)'_*
 - \big( \mu^{1/8} f\big)_*  \Big) ^2  d \sigma dv dv_* \\
 & \lesssim \|\la v \ra^{\gamma/2} g \|^2_{L^2_s} |||\mu^{1/8} f|||_{\Phi_0}^2
 \lesssim \|g\|^2_{L^2_{s+\gamma/2}}|||f|||_{\Phi_\gamma}^2\,,
\end{align*}
where we have used Propositions \ref{part1-prop2} and \ref{part1-prop4} in the last inequality.
As for $A_2$, we decompose it as follows :
\begin{align*}
A_2 & \lesssim
\iiint  b \Big( \big( \mu^{1/8} f\big)'_* \Big)^2\Big(
\big(\la v \ra^{\gamma/2} g\big) ' -
  \big(\la v \ra^{\gamma/2} g\big) \Big) ^2 d \sigma dv dv_*  \\
  &\qquad
  + \iiint  b
\Big( \la v \ra^{\gamma/2} - \la v' \ra^{\gamma/2} \Big)^2
\Big( \big( \mu^{1/8} f\big)'_* \Big)^2 g'^2 d \sigma dv dv_*  \\
&= A_{2,1} + A_{2,2}\,.
\end{align*}
Apply Corollary \ref{part1-imp-upper-maxwel-2} again to $A_{2,1}$. Then
\[
A_{2,1} \lesssim \|\mu^{1/8}f \|_{L^2_s}^2 |||\la v \ra^{\gamma/2}g|||_{\Phi_0}^2
\lesssim \|f\|^2_{L^2_{s+\gamma/2}} |||g|||_{\Phi_\gamma}^2\,.
\]
The estimation for $A_{2,2}$ is the same as the one for $A_2$ in the proof of
Lemma \ref{part1-coer}. By using the change of variables $ (v',v'_*)\rightarrow
(v,v_*)$, we obtain
\begin{align*}
A_{2,2}
\lesssim \iint   \Big (\la v \ra^{2s + \gamma}
\la v_* \ra^{2s+ 2} \Big) \big( \mu^{1/8} f\big)_*^2  |g|^2dvdv_*
\lesssim \|\mu^{1/10} f\|^2_{L^2}\|g\|^2_{L^2_{s+\gamma/2}}\,.
\end{align*}
Noticing  that
$
\Big(\mu^{1/8}_* - {\mu'_*}^{1/8}\Big)^2 \lesssim \min (|v-v_*|^2 \theta^2, 1),
$
we have
\begin{align*}
A_3 &\lesssim
\iint  \tilde \Phi_\gamma
  \Big( \int_{\SS^2} b(\cos \theta) \min (|v-v_*|^2 \theta^2, 1) d \sigma \Big)
f^2_*
g^2   dv dv_* \\
&\lesssim \iint  \la v-v_* \ra^{\gamma +2s}
f^2_*
g^2   dv dv_* \\
&\lesssim \iint
\la v_*\ra^{\gamma +2s}f_*^2  \la v \ra^{\gamma +2s}g^2  dv dv_*
\lesssim \|f\|^2_{L^2_{s+\gamma/2}}\|g\|^2_{L^2_{s+\gamma/2}}\,,
\end{align*}
if $\gamma+2s \geq 0$ because of
$\la v-v_*\ra^{\gamma+2s} \leq \la v_*\ra^{\gamma+2s} \la v\ra^{\gamma+2s}$.

To consider the case $\gamma +2s <0$, we divide
$\RR_v^3 \times \RR_{v_*}^3$ into three parts
\begin{align*}
&U_1 =
\{|v-v_*| \leq |v_*|/8\}\,, \enskip U_2= \{|v-v_*| > |v_*|/8\}\cap \{|v_*|\leq 1\}\,,\\
&U_3= \{|v-v_*| > |v_*|/8\}\cap \{|v_*|> 1\}\,.
\end{align*}
Then we have
\begin{align*}
\frac{1}{3} A_3 &=  \iiint  b \tilde \Phi_\gamma {\mu'}_*^{1/4}
\Big( {\mu'}_*^{1/8} -\mu^{1/8}_*  \Big)^2
 \Big(f_*
g  \Big) ^2 d\sigma dv dv_*\\
&= \iint_{U_1} \int d\sigma dv dv_* + \iint_{U_2} \int d\sigma dv dv_*
+\iint_{U_3} \int d\sigma dv dv_* \\
&= A_{3,1} + A_{3,2} + A_{3,3} \,.
\end{align*}
Since
$|v'-v_*| \leq |v-v_*| \leq |v_*|/8$ implies $7|v_*|/8 \leq |v'|, |v| \leq 9|v_*|/8$
and $|v'_*|^2 =|v|^2+|v_*|^2 -|v'|^2 \geq |v_*|^2/2$.
Hence, we have ${\mu'}_*^{1/4} \leq C \mu_*^{1/8} \leq C ( {\mu_*} \mu )^{1/20}$ on $U_1$,
which leads to
\begin{align*}
A_{3,1} \lesssim \iint (\mu {\mu_*})^{1/20}\la v-v_* \ra^{\gamma +2s} f_*^2 g^2 dv dv_* \leq C\|f\|^2_{L^2_{s+\gamma/2}}
\|g\|^2_{L^2_{s+\gamma/2}}\,.
\end{align*}
Furthermore, we have
\[
A_{3,2} \lesssim \iint_{U_2} \la v-v_* \ra^{\gamma +2s} f_*^2 g^2 dv dv_*\lesssim
\|f\|^2_{L^2_{s+\gamma/2}}
\|g\|^2_{L^2_{s+\gamma/2}},
\]
because $\la v -v_* \ra^{-1} \leq  \la v \ra^{-1}\la v_* \ra^{-1} \la v_* \ra^2
\leq 2 \la v \ra^{-1}\la v_* \ra^{-1}$ on $U_2$.
Since $\la v-v_* \ra^{-1} \leq 8|v_*|^{-1}\leq 16 \la v_* \ra^{-1}$ on $U_3$, we get
\[
A_{3,3} \lesssim
\|f\|^2_{L^2_{s+\gamma/2}}
\|g\|^2_{L^2}\,.
\]
Therefore, we have in the case when $\gamma+2s <0$
\[
A_3 \lesssim \,
 \|f\|^2_{L^2_{s+\gamma/2}}
\|g\|^2_{L^2}\,.
\]
If one considers another partition in $R^6_{v,v_*}$ with $v$ and $v_*$ exchanged,
then the estimate
$$A_3 \lesssim
 \|f\|^2_{L^2}
\|g\|^2_{L^2_{s+\gamma/2}}
$$ holds, because $|v_*'-v| \leq |v_*-v| \leq |v|/8$ implies $7|v|/8 \leq |v'_*|, |v_*| \leq 9|v|/8$.

As a conclusion, when $\gamma >-3/2$ and $\gamma +2s \leq 0$ we have
\begin{align*}
&\Big|\Big(\Gamma(f,g),h\Big)\Big|
\lesssim\Big\{ \|f\|_{L^2_{s+\gamma/2}}|||g|||_{\Phi_\gamma}
+ \|g\|_{L^2_{s+\gamma/2}}|||f|||_{\Phi_\gamma}\\
&\quad + \min \Big ( \|f\|_{L^2}\|g\|_{L^2_{s+\gamma/2}} \,,\,\|f\|_{L^2_{s+\gamma/2}}
\|g\|_{L^2}
\Big)\Big\}|||h|||_{\Phi_\gamma} \, .
\end{align*}
which concludes the proof of the first statement of Proposition \ref{part1-prop5}.

\noindent{\bf The case $\gamma +2s <0$, $\gamma \geq -3s$}.  We go back to the definition of $A_3$, that is (we have performed the usual change of variables)
\begin{align*}
&A_3 \sim  \iiint  b \tilde \Phi_\gamma {\mu'}_*^{1/4}
\Big( {\mu'}_*^{1/8} -\mu^{1/8}_*  \Big)^2
 f^2_*
g^2  d\sigma dv dv_* \\
&\lesssim \iiint  b \tilde \Phi_\gamma
\Big( {\mu'}_*^{1/8} -\mu^{1/8}_*  \Big)^2
 f^2_*
g^2  d\sigma dv dv_* \, .
\end{align*}
We estimate the spherical integral as usual, that is over the sets
$$
| v'_* -v_*| \le {1\over 2} <v_*> \mbox{ and } | v'_* -v_*| \ge {1\over 2} <v_*>
$$
It follows by Taylor formula that, on the first set (which is the singular part), one has, for another non important and non negative constant $c$
$$
(\mu^{1/8 '}_* - \mu^{1/8}_* )^2 \lesssim \theta^2 | v-v_*|^2 \mu^c_* .
$$
On the other set, we just estimate the square by $1$. Note that on the second set we have, $| v-v_*| \gtrsim <v_*>$.

Then we find, by now standard computations, that
\begin{align*}
&A_3 \lesssim \iint  < v-v_*>^\gamma \mu^c_* | v-v_*|^{2s} <v_*>^{2-2s} f^2_* g^2\\
&+ \iint  < v-v_*>^\gamma  | v-v_*|^{2s} <v_*>^{-2s} f^2_* g^2 \\
&= \tilde A_{3,1}+\tilde A_{3,2}
\end{align*}
Now, for $\tilde A_{3,1}$, we write $<v-v_*>^{\gamma +2s} \lesssim < v>^{\gamma +2s} <v_*>^{-\gamma -2s}$ and we see that we may absorb all the powers of $<v_*>$ with the maxwellian, to get, for another non negative constant $d$
$$
\tilde A_{3,1} \lesssim || \mu^d f||^2_{L^2} || g||^2_{L^2_{\gamma /2 +s}} .
$$
For $\tilde A_{3,2}$, we write
$$
<v-v_*>^{\gamma +2s} <v_*>^{-2s} \lesssim <v>^{\gamma +2s} <v_*>^{-\gamma -2s} < v_*>^{-2s} = <v>^{\gamma +2s} <v_*>^{-\gamma -4s} .
$$
Note that the power $-\gamma -4s$ which enters the power over $<v_*>$ can be written
$$
-\gamma -4s = -( \gamma +2s) - 2s
$$
the first term being positive. Of course $-\gamma -4s \le 0 $ iff $ \gamma \ge -4s $, and this is true since we have assumed that $\gamma \ge -3s$. Furthermore $\gamma +4s \ge - \gamma -2s$ again because $\gamma \ge -3s$.
Therefore we obtained
$$
\tilde A_{3,2} \lesssim || f||^2_{L^2_{\gamma /2 +s}} || g||^2_{L^2_{\gamma /2 +s}} ,
$$
concluding the proof of the second statement.
\end{proof}

Let us note that the proof of Proposition \ref{part1-prop5} gives the following corollary.

\begin{coro}\label{part1-coro-radja1}
With the regularized potential, together with assumptions \eqref{part2-hyp-2}, $0<s<1$ and $\gamma \ge -3s$, one has
$$
\Big|\Big(\Gamma_{\tilde \Phi}(f,g),h\Big)\Big|
\lesssim \Big\{\|f\|_{L^2_{s+\gamma/2}}|||g|||_{\Phi_\gamma}
+ \|g\|_{L^2_{s+\gamma/2}}|||f|||_{\Phi_\gamma}
\Big\}|||h|||_{\Phi_\gamma} \, .
$$
\end{coro}
Note carefully that in this last result, the constraint $\gamma >-3/2$ is removed, and we have retained the constraint $\gamma \ge -3s$, which is always true for physical cases, as we saw above.

\subsection{Proof of Theorem \ref{theorem-1.1-b}}
The general case is long and will be divided into several steps. One of the key ingredients in the proof is to split
the term $(\Gamma (f,g); h)$ into two parts. In this way, one part
can be dealt with the method introduced in the previous subsection, while the other one will be analyzed by direct Fourier transform.

\begin{lemm}\label{part2-weight-radja-lemma-2.1} For any integer $k \geq 2$ we can write
\begin{align*}
& \mu_*^{1/2} = \Big (\mu^{a^1} - \mu^{a^1}_* \Big)^k  \sum_{i=1}^{k+2} \alpha^2_i  \mu_*^{a^2_i}\mu^{ b^2_i}
+ \sum^{k}_{i=1}\alpha^3_i \mu_*^{a^3_i} \mu^{b^3_i}\\
& =^{^{\hskip-0.25cm {\mbox{\tiny def } }}} \mu (v,v_*) + \sum^{k}_{i=1}\alpha^3_i \mu_*^{a^3_i} \mu^{b^3_i} .
\end{align*}
In the above,  $\alpha_i^j$ are real numbers for
all $i$ and $j$, and the other exponents are
strictly positive, at the exception of $b_i^2 =0$, with  $b^3_i > a^3_i$.
\end{lemm}
\begin{proof}
 Differentiating  $k-1$ times the identity
$\displaystyle \sum_{j=0}^{2k} x^j = \frac{1-x^{2k+1}}{1-x}$ we have
\begin{align*}
\sum_{j=k-1}^{2k} \frac{j!}{(j-k+1)!} x^{j-k+1} &= \left( \frac{1}{1-x} \right)^{(k-1)} - \sum_{j=0}^{k-1}
\left(
\begin{array}{c}
k-1\\ j \end{array}
\right) \left(\frac{1}{1-x}\right)^{(j)} \Big( x^{2k+1} \Big)^{(k-1-j)}\\
&= \frac{(k-1)!}{(1-x)^k} \left\{1 - \sum_{j=0}^{k-1} \frac{(2k+1)!}{j!}
\left(\sum_{n=0}^{k-j-1} \frac{(-1)^n}{n! (2k-j-n+1)!}\right) x^{2k-j+1}\right\}.
\end{align*}
 By setting $x= B/A$ and multiplying the above identity by  $A^{2k+2}\Big( 1 -B/A \Big)^k / (k-1)!$, we obtain
 \begin{align*}
 A^{2k+2}&= (A-B)^k \sum_{j=k-1}^{2k}\frac{j!}{(j-k+1)!}  A^{2k-j+1}B^{j-k+1}\\
 &+ \sum_{j=0}^{k-1} \frac{(2k+1)!}{j!}
\left(\sum_{n=0}^{k-j-1} \frac{(-1)^n}{n! (2k-j-n+1)!}\right)  A^{j+1}B^{2k-j+1} \,,
 \end{align*}
 which gives the desired formula.
 \end{proof}

With the help of Lemma \ref{part2-weight-radja-lemma-2.1}, we can
analyze   $ (\Gamma (f,g) ; h )$ as follows. Write
$$
(\Gamma (f , g),\, h) = (\Gamma_\mu (f , g),\, h) + (\Gamma_{rest} (f , g),\, h),
$$
with
\begin{equation}\label{part2-weight-radja-2}
(\Gamma_\mu (f , g) , h) = \iiint b\Big( {{v-v_*}\over{| v-v_*|}}\cdot\sigma\Big) \Phi_\gamma (v-v_*) \mu (v,v_*) \big(f'_* g'- f_* g\big) h dvdv_* d\sigma,
\end{equation}
and $(\Gamma_{rest} (f , g) ; h) $ is a finite linear combination of terms
in the form of
\begin{align}\label{part2-weight-radja-3}
(\Gamma_{mod,i} (f , g) , h) &= \iiint b\Big( {{v-v_*}\over{| v-v_*|}}\cdot\sigma\Big) \Phi_\gamma (v-v_*)
\big(f'_* g'- f_* g\big)\mu_*^{c_i}\mu^{c_i} \mu^{d_i} h dvdv_* d\sigma\\
&= (Q ( \mu^{c_i} f, \mu^{c_i} g) ) \,,\,  \mu^{d_i} h ), \notag
\end{align}
with $d_i>0$, $c_i>0$, for $1\leq i \leq 3$.

The following two propositions give estimates on each of these scalar products, and all together imply Theorem \ref{theorem-1.1-b}.

\begin{prop}
For all $0< s< 1$ and  $\gamma >-3$, one has
\begin{align*}
| (\Gamma_{\mu} (f , g) ; h) | &\lesssim \Big\{ \|f\|_{L^2_{s+\gamma/2}}|||g|||_{\Phi_\gamma}
+ \|g\|_{L^2_{s+\gamma/2}}|||f|||_{\Phi_\gamma}\\
 &+ \min\lbrace  ||f||_{L^2} || g||_{L^2_{s+\gamma /2}} ,||f||_{L^2_{s+\gamma /2}} || g||_{L^2} \rbrace \Big\}|||h|||_{\Phi_\gamma} \, .
\end{align*}
\end{prop}
\begin{proof}
Since $\mu (v,v_*)$ is a finite sum of $(\mu^{a} - \mu^{a}_*)^k \mu_*^{2b}\mu^{c}$ with $a,b >0$ and
$c \geq 0$, by setting $H = \mu^c h$, we can write
\begin{align*}
(\Gamma_\mu (f , g) , h)
& = \iiint b\,
\Phi_\gamma (\mu^{a} - \mu^{a}_*)^k \mu_*^b \big(f'_* g'- f_* g\big) \mu_*^b H d \sigma dv dv_*  \\
&=  {1\over 2}\iiint b\,\Phi_\gamma (\mu^{a} - \mu^{a}_*)^k \mu_*^b \big( f'_* g'- f_* g\big) \big( \mu_*^b H-
{\mu'}_*^b H' \big) d \sigma dv dv_*\\
& + {1\over 2}\iiint b\, \Phi_\gamma \big( f'_* g'- f_* g \big) \Big\{  ( {\mu'}^{a} - {\mu'}^{a}_* )^k {\mu'}_*^b -
(\mu^a- \mu^a_*)^k \mu_*^b \Big\}
{\mu'}_*^b H'  d \sigma dv dv_*\,.
\end{align*}
By setting $\Phi_\gamma = \tilde \Phi_\gamma^{1/2} \big( \Phi_\gamma\tilde \Phi_\gamma^{-1/2}\big)$,   the Cauchy-Schwarz inequality gives
\[
\left|(\Gamma_\mu (f , g) , h) \right|
\leq \frac 1 2 A^{1/2} \Big( D^{1/2} + E^{1/2}\Big)\,,
\]
where
\begin{align*}
A& =
\iiint  b  \tilde \Phi_\gamma \mu_*^{2b} \Big(f'_*
g' - f_* g \Big) ^2  d \sigma dv dv_* \,,\\
D&= \iiint  b \big( \Phi_\gamma\tilde \Phi_\gamma^{-1/2}  (\mu^{a} - \mu^{a}_*)^k   \big)^2
\big( \mu_*^b H-
{\mu'}_*^b H' \big)^2 d \sigma dv dv_* \,, \\
E&= \iiint  b \Big( \Phi_\gamma\tilde \Phi_\gamma^{-1/2} \Big\{  ( {\mu'}^{a} - {\mu'}^{a}_* )^k {\mu'}_*^b -
(\mu^a- \mu^a_*)^k \mu_*^b \Big\}   \Big)^2 H^2
 d \sigma dv dv_* \,.
\end{align*}
It is easy to see  $D + E \lesssim |||H|||_{\Phi_\gamma} ^2$ because
$\big( \Phi_\gamma\tilde \Phi_\gamma^{-1/2}  (\mu^{a} - \mu^{a}_*)^k   \big)^2 \lesssim
 \Phi_\gamma$, and
\[
\Big( \Phi_\gamma\tilde \Phi_\gamma^{-1/2} \Big\{  ( {\mu'}^{a} - {\mu'}^{a}_* )^k {\mu'}_*^b -
(\mu^a- \mu^a_*)^k \mu_*^b \Big\}   \Big)^2
\lesssim \Phi_\gamma\big( \mu_*^b + {\mu'_*}^b \big) \min \{ |v-v_*|^2 \theta^2\,,\, 1 \}  \,.
\]
The estimation on $A$ is just the same as the one in the proof of Proposition \ref{part2-prop5.2}. And this completes the proof of the proposition.
\end{proof}



To estimate the last term in \eqref{part2-weight-radja-3}, we have the following proposition.

\begin{prop}\label{part2-weight-radja-proposition-2.2}
For all $0<s<1$ and  $\gamma>\max\{-3,- 3/2 -2s\} $,  one has
$$
| (\Gamma_{mod,i} (f , g) , h) | \lesssim  \Big\{ \|f\mu^{c_i}\|_{L^2_{s+\gamma/2}}|||g\mu^{c_i}|||_{\Phi_\gamma}
+ \|g\mu^{c_i}\|_{L^2_{s+\gamma/2}}|||f\mu^{c_i}|||_{\Phi_\gamma} \Big\}|||h\mu^{d_i}|||_{\Phi_\gamma} \, .
$$
\end{prop}

For this, we consider
$$
(\Gamma_{mod, i} (f , g) , h) = (Q ( \mu^{c_i} f, \mu^{c_i} g) , \mu^{d_i} h ) .
$$
We therefore  consider  $(Q(F, G), H)$ with
\begin{equation}\label{part2-weight-radja-4}
F =f\mu^{c_1},\  G= g\mu^{c_2},\ H = h\mu^{c_3},
\end{equation}
for some positive constants $c_1$, $c_2$ and $c_3$.

Let $0\leq \phi (v)\leq 1$ be a smooth  radial function with
value $1$ for $v$ close to $0$, and $0$ for large values of $v$. Set
$$
\Phi_\gamma (v) = \Phi_\gamma (v) \phi (v) + \Phi_\gamma (v) (1-\phi (v)) = \Phi_c (v) + \Phi_{\bar c} (v).
$$
And then correspondingly we can write
$$
Q (F, G) = Q_c (F, G) + Q_{\bar c} (F, G),
$$
where the kinetic factor in the collision
operator is defined according to the decomposition respectively.
To prove  Proposition \ref{part2-weight-radja-proposition-2.2}, it suffices to prove the following two lemmas, by taking $m=-s$ in the statements. The general form below for any real $m$ will be needed in Part II.

\begin{lemm} \label{part2-weight-radja-proposition-4.1}
For all $0< s < 1$ and  $\gamma >-3$, one has
$$
\Big|\Big( Q_{\bar c} (F,G),H\Big)\Big|
\leq C\Big\{ \|F\|_{L^2_{s+\gamma/2}}|||G|||_{\Phi_\gamma}
+ \|G\|_{L^2_{s+\gamma/2}}|||F|||_{\Phi_\gamma} \Big\}|||H|||_{\Phi_\gamma} \, .
$$
\end{lemm}

\begin{proof}
One has for some positive constant $\beta$
\begin{align*}
& | (Q_{\bar c} (F , G) ; H) = | \iiint b\, \Phi_{\bar c} (v-v_*) \mu_*^\beta \mu^\beta [F'_* G'- F_* G]  H dvdv_*d\sigma |\\
&= \frac 1 2 | \iiint b\, \Phi_{\bar c} (v-v_*)  [F'_* G'- F_* G]  \mu_*^\beta \mu^\beta[H' - H] dvdv_*d\sigma |\\
&\lesssim A^{1/2}\,\,B^{1/2},
\end{align*}
where
$$ A = \iiint b\ \Phi_{\bar c} (v-v_*)  [F'_* G'- F_* G]^2  \mu_*^\beta \mu^\beta dvdv_*d\sigma,
$$
and
$$
B= \iiint b  \Phi_{\bar c} (v-v_*)\mu_*^\beta \mu^\beta[H' - H]^2 dvdv_*d\sigma .
$$
$B$ is clearly estimated from above by the dissipative norm of $H$, while for $A$, we note that  $\Phi_{\bar c} \lesssim \tilde \Phi_\gamma$.
The proof of Proposition \ref{part2-prop5.2} in Subsection \ref{part2-section2.2} can
then be applied to $A$. And this gives the desired estimate and then completes
the proof of the lemma.
\end{proof}

Next, let us note that, from the Appendix, $| \hat \Phi_c (\xi )| \lesssim < \xi>^{-3-\gamma}$ .
We shall prove
\begin{lemm} \label{part2-weight-radja-proposition-4.2}

Let $m\in \RR$.
For all $0< s<1$ and   $\gamma>\max\{-3, -\frac 32 -2s\} $, one has
$$
| (Q_c (F,  G); H )|  \lesssim \Big( || F|_{L^2} || G||_{H^{ (m+2s)^+ }}+
|| F||_{H^{{ s+ (m+s)^+}}} || G||_{L^2}
+ || F||_{H^{{ s}}} || G||_{H^{ (m+s)^+}}\Big)|| H||_{H^{ -m}}\,.
$$
\end{lemm}

It is important to note that even though
the statement of Lemma \ref{part2-weight-radja-proposition-4.2} is
not as sharp as the one  of Lemma \ref{part2-weight-radja-proposition-4.1},
 by recalling \eqref{part2-weight-radja-4}, we have
 all the needed weights because we are dealing with functions of the form $F$, $G$ and $H$ that contain Gaussians. Hence,
 these two lemmas together imply Proposition \ref{part2-weight-radja-proposition-2.2}.

 For the proof of Lemma \ref{part2-weight-radja-proposition-4.2}, first of all, by using the formula from the Appendix of \cite{al-1}, we have
\begin{align*}
\cF(Q_c (F, G)) (\xi ) & = \iint_{\RR^3_{\xi_*}\times \SS^2}b \Big({\xi\over{ | \xi |}} \cdot \sigma \Big) \hat \Phi_c (\xi_* ) \hat F(\xi^-+ \xi_* ) \hat G (\xi^+ -\xi_*) \\
&- \iint_{\RR^3_{\xi_*}\times \SS^2} b \Big({\xi\over{ | \xi |}} \cdot \sigma \Big) \hat \Phi_c (\xi_* ) \hat F (\xi_* ) \hat G(\xi - \xi_* ).
\end{align*}
We change variables in $\xi_*$ in the first integral to obtain
\begin{align*}
( Q_c(F, G); H ) =& \iiint b \Big({\xi\over{ | \xi |}} \cdot \sigma \Big) [ \hat\Phi_c (\xi_* - \xi^- ) - \hat \Phi_c (\xi_* ) ] \hat F (\xi_* ) \hat G(\xi - \xi_* ) \overline{{\hat H} (\xi )} d\xi d\xi_*d\sigma .\\
= & \iiint_{ | \xi^- | \leq {1\over 2} \la \xi_*\ra }  \cdots d\xi d\xi_*d\sigma
+ \iiint_{ | \xi^- | \geq {1\over 2} \la \xi_*\ra } \cdots d\xi d\xi_*d\sigma \,\\
=& A_1(F,G,H)  +  A_2(F,G,H) \,\,.
\end{align*}
 $A_2(F,G,H)$ can be naturally decomposed into
\begin{align*}
A_2 &=  \iiint b \Big({\xi\over{ | \xi |}} \cdot \sigma \Big) {\bf 1}_{ | \xi^- | \ge {1\over 2}\la \xi_*\ra }
\hat\Phi_c (\xi_* - \xi^- ) \hat F (\xi_* ) \hat G(\xi - \xi_* ) \overline{{\hat H} (\xi )} d\xi d\xi_*d\sigma .\\
&- \iiint b \Big({\xi\over{ | \xi |}} \cdot
 \sigma \Big){\bf 1}_{ | \xi^- | \ge {1\over 2}\la \xi_*\ra } \hat \Phi_c (\xi_* ) \hat F (\xi_* ) \hat G(\xi - \xi_* ) \overline{{\hat H} (\xi )} d\xi d\xi_*d\sigma \\
&= A_{2,1}(F,G,H) - A_{2,2}(F,G,H)\,.
\end{align*}

For $A_1$, we use the Taylor expansion of $\hat \Phi_c$ at order $2$ to have
$$
A_1 = A_{1,1} (F,G,H) +A_{1,2} (F,G,H)
$$
where
$$
A_{1,1} = \iiint b\,\, \xi^-\cdot (\nabla\hat\Phi_c)( \xi_*)
{\bf 1}_{ | \xi^- | \leq {1\over 2} \la \xi_*\ra }  \hat F (\xi_* ) \hat G(\xi - \xi_* ) \bar{\hat H} (\xi ) d\xi d\xi_*d\sigma,
$$
and $A_{1,2} (F,G,H)$ is the remaining term corresponding to the second order term in the Taylor expansion of $\hat\Phi_c$. The $A_{i,j}$ with
$i,j=1,2$ are estimated by the following lemmas.

\begin{lemm}
 Let $m \in \RR$.
For all $\gamma>\max (-3, -  \frac{3}{2} - 2s )$, one has
$$
| A_{1,j} | \lesssim \Big( || F|_{L^2} || G||_{H^{ (m+2s)^+ }}+ || F||_{H^{{ s+ (m+s)^+}}} || G||_{L^2}
+ || F||_{H^{{ s}}} || G||_{H^{ (m+s)^+}}\Big)|| H||_{H^{ -m}} \,\,, j=1,2 \,.
$$
\end{lemm}

\begin{proof}
We first consider $A_{1,1}$. By writing
\[
\xi^- = \frac{|\xi|}{2}\left(\Big(\frac{\xi}{|\xi|}\cdot \sigma\Big)\frac{\xi }{|\xi|}-\sigma\right)
+ \left(1- \Big(\frac{\xi}{|\xi|}\cdot \sigma\Big)\right)\frac{\xi}{2},
\]
we see that the integral corresponding to the first term on the right hand side vanishes because of the symmetry
on $\SS^2$.
Hence, we have
\[
A_{1,1}= \iint_{\RR^6} K(\xi, \xi_*)
\hat F (\xi_* ) \hat G(\xi - \xi_* ) \bar{\hat H} (\xi ) d\xi d\xi_* \,,
\]
where
\[
K(\xi,\xi_*) = \int_{\SS^2}
 b \Big({\xi\over{ | \xi |}} \cdot \sigma \Big)
\left(1- \Big(\frac{\xi}{|\xi|}\cdot \sigma\Big)\right)\frac{\xi}{2}\cdot
(\nabla\hat\Phi_c)( \xi_*)
{\bf 1}_{ | \xi^- | \leq {1\over 2} \la \xi_*\ra } d \sigma \,.
\]
Note that $| \nabla \hat \Phi_c (\xi_*) | \lesssim {1\over{\la \xi_*\ra^{3+\gamma +1}}}$, from the Appendix.
If $\sqrt 2 |\xi| \leq \la \xi_* \ra$, then $|\xi^-| \leq \la \xi_* \ra/2$
and this  implies that for
$0 \leq \theta \leq \pi/2$,
\begin{align*}
|K(\xi,\xi_*)| &\lesssim \int_0^{\pi/2} \theta^{1-2s} d \theta\frac{ \la \xi\ra}{\la \xi_*\ra^{3+\gamma +1}}
\lesssim \frac{\la \xi\ra^s  \la \xi_*\ra^s   }{\la \xi_*\ra^{3+\gamma +2s}}\left(
\frac{\la \xi \ra}{\la \xi_*\ra}\right)^{1-s} \\
&\lesssim \frac{\la \xi\ra^{ m+s}  \la \xi_*\ra^s   }{\la \xi_*\ra^{3+\gamma +2s}}\la \xi\ra^{-m}
\lesssim \la \xi_*\ra^s \frac{  \la \xi_* \ra^{( m+s)^+} + \la \xi -\xi_* \ra^{( m+s)^+}      }{\la \xi_*\ra^{3+\gamma +2s}}\la \xi\ra^{ -m}
\end{align*}
On the other hand, if $\sqrt 2 |\xi| \geq \la \xi_* \ra$, then
\begin{align*}|K(\xi,\xi_*)| &\lesssim \int_0^{\pi\la \xi_*\ra /(2|\xi|)} \theta^{1-2s} d \theta\frac{ \la \xi\ra}{\la \xi_*\ra^{3+\gamma +1}}
\lesssim \frac{\la \xi\ra^{2s}   }{\la \xi_*\ra^{3+\gamma +2s}}
\frac{\la \xi_* \ra}{\la \xi \ra}\\
&\lesssim \frac{ \Big( \la \xi - \xi_* \ra^{(m+2s)^+} + \la \xi_* \ra^{ (m+2s)^+} \Big)  }{\la \xi_*\ra^{3+\gamma +2s}} \la \xi\ra^{ -m}\,.
\end{align*}
Since $\la \xi_*\ra^{-(3+\gamma +2s)} \in L^2$ when
 $\gamma >-3/2 -2s$, we obtain the desired estimate for $A_{1,1}$.

Now we turn to $A_{1,2} (F, G, H)$, which comes from the second order term of the Taylor expansion. Note that
$$
A_{1,2} = \iiint  b \Big({\xi\over{ | \xi |}} \cdot \sigma \Big)\int^1_0 d\tau (\nabla^2\hat \Phi_c) (\xi_* -\tau\xi^- ) \cdot\xi^- \cdot\xi^- \hat F (\xi_* ) \hat G(\xi - \xi_* ) \bar{\hat H} (\xi ) d\sigma d\xi d\xi_*\, .
$$
>From the Appendix, we have
$$
| (\nabla^2\hat \Phi_c) (\xi_* -\tau\xi^- ) | \lesssim {1\over{\la  \xi_* -\tau \xi^-\ra^{3+\gamma +2}}}
\lesssim
 {1\over{\la \xi_*\ra^{3+\gamma +2}}},
$$
because $|\xi^-| \leq \la \xi_*\ra/2$.
Similar to $A_{1,1}$, we can obtain
\[
|A_{1,2}| \lesssim
 \iint_{\RR^6} \tilde K(\xi, \xi_*)
\hat F (\xi_* ) \hat G(\xi - \xi_* ) \bar{\hat H} (\xi ) d\xi d\xi_* \,,
\]
where $\tilde K(\xi,\xi_*)$ has the following upper bound
\[
\tilde K(\xi,\xi_*) \lesssim \int_0^{\min(\pi/2, \,\, \pi\la \xi_*\ra /(2|\xi|))} \theta^{1-2s} d \theta
\frac{ \la \xi\ra^2}{\la \xi_*\ra^{3+\gamma +2}}
\lesssim \frac{\la \xi\ra^{2s}   }{\la \xi_*\ra^{3+\gamma +2s}}\,,
\]
which yields the desired estimate for $A_{1,2}$. And this completes
the proof of the lemma.
\end{proof}

\begin{lemm}
Let $m \in \RR$.
For all $0<s<1$ and $\gamma>\max(-3, - {3\over 2}- 2s)$, one has
$$
 | A_{2,1} |+| A_{2,2}|  \lesssim \big( || F|_{L^2} || G||_{H^{ (m+2s)^+ }}+ || F||_{H^{{ (m+2s)^+}}} || G||_{L^2} \big)|| H||_{H^{-m}} .
$$
\end{lemm}

\begin{proof}
In view of the definition of $A_{2,2}$, Since
we assume that $\theta \ge {1/2 } | \xi_*| | \xi |^{-1}$, we also
have ${1/2 } \la \xi_*\ra | \xi |^{-1} \le {\pi\over 2}$, that is, $\la  \xi_* \ra \lesssim | \xi|$.
We can then
directly compute the spherical integral appearing inside $A_{2,2}$ together with $\Phi$ by using the inequality
\begin{align*}
  {1\over{\la \xi_* \ra^{3+\gamma }}} \frac{\la  \xi\ra^{2s} }{\la \xi_*\ra^{2s}}
\lesssim {1\over{\la \xi_* \ra^{3+\gamma +{2 s}}}}  \la \xi \ra^{ -m}
\Big( \la\xi_*\ra ^{(m+2s)^+ } + \la \xi - \xi_* \ra^{  (m+2s)^+} \Big) \,,
\end{align*}
to obtain the estimate for $A_{2, 2}$.

We now turn to
$$
A_{2,1}=  \iiint b\,\, {\bf 1}_{ | \xi^- | \ge {1\over 2} \la  \xi_*\ra }\hat \Phi_c (\xi_* - \xi^-) \hat F (\xi_* ) \hat G(\xi - \xi_* ) \bar{\hat H} (\xi ) d\sigma d\xi d\xi_* .
$$
Firstly, note that we can  work on the set $| \xi_* \,\cdot\,\xi^-| \ge {1\over 2} | \xi^-|^2$. In fact, on the complementary of this
set, we have
 $| \xi_* \,\cdot\,\xi^-| \leq {1\over 2} | \xi^-|^2$ so that
 $|\xi_* -\xi^-| \gtrsim | \xi_*|$, and in this case,
we can proceed in the same way as for $A_{2,2}$. Therefore, it suffices to estimate
\begin{align*}
A_{2,1,p}&=  \iiint b\,\, {\bf 1}_{ | \xi^- | \ge {1\over 2} \la \xi_*\ra }{\bf 1}_{| \xi_* \,\cdot\,\xi^-| \ge {1\over 2} | \xi^-|^2}\hat \Phi_c (\xi_* - \xi^-) \hat F (\xi_* ) \hat G(\xi - \xi_* ) \overline{{\hat H} (\xi )} d\sigma d\xi d\xi_* .\\
&= \iiint b\,\, \frac{{\bf 1}_{ | \xi^- | \ge {1\over 2} \la \xi_*\ra }{\bf 1}_{| \xi_* \,\cdot\,\xi^-| \ge {1\over 2} | \xi^-|^2}}
{\Big( \la\xi_*\ra ^{ (m+2s)^+ } + \la \xi - \xi_* \ra^{ (m+2s)^+}  \Big)}
\,\hat \Phi_c (\xi_* - \xi^-)
\\&\qquad  \qquad \times
\overline{{\hat H} (\xi ) }  \sum_{j=1}^2\hat F_j (\xi_* ) \hat G_j(\xi - \xi_* ) d\sigma d\xi d\xi_* \,,
\end{align*}
where $\hat F_1= \la \xi \ra^{ (m+2s)^+} \hat F$ , $\hat G_1 =\hat G$ and $\hat F_2 = \hat F$,
$\hat G_2 = \la \xi \ra^{ (m+2s)^+} \hat G$.  On the set for
the above integral, we have $\la \xi_* -\xi^- \ra^{2s} \lesssim \,
\la \xi_* \ra^{2s}$, because $| \xi^- | \lesssim | \xi_*|$
that follows from  $| \xi^-|^2 \le 2 | \xi_* \cdot\xi | \lesssim |\xi^-|\, | \xi_*|$.
By the Cauchy-Shwarz inequality, we have
\[
|A_{2,1,p}| \leq \sum_{j=1}^2 D^{1/2} D_j^{1/2},
\]
where
$$
D= \int_{\RR^3}\Big(\iint_{\RR^3\times \SS^2} b \Big({\xi\over{ | \xi |}} \cdot \sigma \Big) {\bf 1}_{ | \xi^- | \ge {1\over 2} \la \xi_*\ra} {{| \hat \Phi_c (\xi_*  - \xi^-)|^2}\over{\la\xi\ra^{ 2s +2m}  \la \xi_* - \xi^-\ra^{2s}}} d\sigma d\xi_* \Big) |\hat H(\xi)|^2 d\xi
$$
and
$$
D_j = \iiint_{\RR^6\times \SS^2} b \Big({\xi\over{ | \xi |}} \cdot \sigma \Big)
\frac{ {\bf 1}_{ | \xi^- | \ge {1\over 2} \la \xi_*\ra } \la\xi\ra^{ 2s +2m}  \la \xi_* \ra^{2s}  }
{\Big( \la\xi_*\ra ^{ (m+2s)^+ } + \la \xi - \xi_* \ra^{ (m+2s)^+}  \Big)^2}
\, |\hat F_j (\xi_* )|^2 |\hat G_j(\xi-\xi_*)|^2   d\sigma d\xi_* d\xi.
$$
Since $\int_{\SS^2}  b \Big({\xi\over{ | \xi |}} \cdot \sigma \Big)
 {\bf 1}_{ | \xi^- | \ge {1\over 2} \la \xi_* \ra} d \sigma \lesssim |\xi|^{2s}\la \xi_*\ra^{-2s}$, we obtain
\[
D_j \lesssim \|F_j\|^2_{L^2} \|G_j\|^2_{L^2}\,\,.
\]

For $D$, we use the change of variables in $\xi_*$, $u= \xi_* -\xi^-$ to get
$$
D=\int_{\RR^3}\Big( \iint_{\RR^3\times \SS^2} b \Big({\xi\over{ | \xi |}} \cdot \sigma \Big) {\bf 1}_{ | \xi^- | \ge {1\over 2}
\la u+\xi^-\ra } {{| \hat \Phi_c (u)|^2}\over{\la \xi\ra ^{2s+2m }   \la u \ra^{2s}}} d\sigma d u\Big)   |\hat H(\xi)|^2 d\xi\,.
$$
By noting that $| \xi ^-| \ge {1\over 2} \la u +\xi^-\ra $ implies  $|\xi^-| \geq \la u\ra/\sqrt {10}$, we have
$$
 D\lesssim \int_{\RR^6} \left(\frac{|\xi|}{\la u\ra } \right)^{2s}
\frac {| \hat \Phi_c (u)|^2}{\la\xi \ra^{ 2s+2m }  \la u \ra^{2s}}   |\hat H(\xi)|^2 d\xi  d u
\lesssim \|H\|^2_{H^{  -m }}\,,
$$
because $2(\gamma+3) +4s > 3$. And this completes the proof of the lemma.
\end{proof}

\subsection{Estimation of commutators }
 By using the arguments similar to those used
in previous subsections, we now prove the following estimation on commutators.

\begin{prop}\label{part2-weight-radja-lemma-6.1} Assume that  $0< s<1$ and $\gamma >-3$. Then, for any $l\ge 1$, one has
\begin{align}\label{part2-comm-1}
& \Big| \Big( W_l \Gamma (f,g) - \Gamma (f, W_lg),\,\,  h \Big)_{L^2}\Big| \notag\\
& \lesssim \Big( || f||_{L^2_{s+\gamma /2}} || W_{l-s} g ||_{L^2_{s+\gamma /2}} + || g||_{L^2_{s+\gamma /2}} || W_{l-s} f ||_{L^2_{s+\gamma /2}}\notag\\
& + \min \lbrace || f||_{L^2} || W_{l-s} g ||_{L^2_{s+\gamma /2}} , || f||_{L^2_{s+\gamma /2}} || W_{l-s} g ||_{L^2} \rbrace  \\
&+ \min \lbrace || g||_{L^2} || W_{l-s} f ||_{L^2_{s+\gamma /2}} , || g||_{L^2_{s+\gamma /2}} || W_{l-s} f ||_{L^2} \rbrace \Big) ||| h|||_{\Phi_\gamma}\, ;\notag
\end{align}
and for any $0<l <1$, one has
\begin{align}\label{part2-comm-2}
&\Big| \Big( W_l \Gamma (f,g) - \Gamma (f, W_lg),\,\,  h \Big)_{L^2}\Big| \notag\\
&  \lesssim \Big( || f||_{L^2_{s+\gamma /2}} || W_{l-sl} g ||_{L^2_{s+\gamma /2}} + || g||_{L^2_{s+\gamma /2}} || W_{l-sl} f ||_{L^2_{s+\gamma /2}}\notag\\
& +\min \lbrace || f||_{L^2} || W_{l-sl} g ||_{L^2_{s+\gamma /2}} , || f||_{L^2_{s+\gamma /2}} || W_{l-sl} g ||_{L^2} \rbrace  \\
&+ \min \lbrace || g||_{L^2} || W_{l-sl} f ||_{L^2_{s+\gamma /2}} , || g||_{L^2_{s+\gamma /2}} || W_{l-sl} f ||_{L^2} \rbrace \Big) ||| h|||_{\Phi_\gamma}\, .\notag
\end{align}
\end{prop}

\begin{rema}\label{part2-weight-radja-lemma-6.1+1} Assume that  $0< s<1$ and $\gamma >-3$. Then, for any $l\ge 0$, one has
\begin{align*}
 &\Big| \Big( \tilde{W}_l \Gamma (f,g) - \Gamma (f, \tilde{W}_lg),\,\,  h \Big)_{L^2}\Big| \\
 & \lesssim \Big( || f||_{L^2_{s+\gamma /2}} || \tilde{W}_{l} g ||_{L^2_{s+\gamma /2}} + || g||_{L^2_{s+\gamma /2}} || \tilde{W}_{l} f ||_{L^2_{s+\gamma /2}}\\
& + \min \lbrace || f||_{L^2} || \tilde{W}_{l} g ||_{L^2_{s+\gamma /2}} , || f||_{L^2_{s+\gamma /2}} || \tilde{W}_{l} g ||_{L^2} \rbrace  \\
&+ \min \lbrace || g||_{L^2} || \tilde{W}_{l} f ||_{L^2_{s+\gamma /2}} , || g||_{L^2_{s+\gamma /2}} || \tilde{W}_{l} f ||_{L^2} \rbrace \Big) ||| h|||_{\Phi_\gamma}\, ,
\end{align*}
where we use \eqref{part2-comm-1} if $l\, |\gamma/2+s|\geq 1$ and \eqref{part2-comm-2} if $l\, |\gamma/2+s|<1$.
\end{rema}

\begin{proof}
In view of the decomposition given for $\Gamma$, it is enough to consider \eqref{part2-weight-radja-2}
with
$$
\mu (v,v_*) =    (\mu^c - \mu^{c}_*)^2  \mu_*^{a},
$$
for some constants $a,c>0$.
Indeed, all the other terms have compensation by some Gaussian function so that
any algebraic weight is not a problem.
For this term, the commutator is then given by
$$
\Big| \Big( W_l \Gamma _\mu(f,g) - \Gamma_\mu (f, W_l g),\,\,  h \Big)_{L^2}\Big|  = \iiint b\, \Phi_\gamma (v-v_*) \mu (v,v_*) f'_* g'(W_l -W_l')  h dvdv_* d\sigma,
$$
which can be written as
\begin{align*}
&\iiint b\, \Phi_\gamma (v-v_*) \mu (v,v_*) f'_* g'(W_l -W_l')  h dvdv_* d\sigma\\
&= \iiint b\, \Phi_\gamma (v-v_*) \mu (v,v_*) f'_* g' (W_l -W_l') ( h- h') dvdv_* d\sigma\\
&+ \iiint b\, \Phi_\gamma (v-v_*) [\mu (v',v_*') - \mu (v,v_*)] f_* g(W_l' -W_l)  h dvdv_* d\sigma\\
&+ \iiint b\, \Phi_\gamma  (v-v_*) \mu (v,v_*) f_* g(W_l' -W_l)  h dvdv_* d\sigma\\
&=A + B+ C,
\end{align*}
by the usual change of variables. For $A$, we use the Cauchy-Schwarz inequality to get
\begin{align*}
A& = \iiint b\, \Phi_\gamma  (v-v_*) \mu (v,v_*) f'_* g' (W_l -W_l') [ h- h'] dvdv_* d\sigma \\
&\lesssim \Big( \iiint b\, \Phi_\gamma (v-v_*)  (\mu^c - \mu^{c}_*)^4  \mu_*^{a} |f'_*|^2 |g'|^2 |W_l -W_l'|^2  dvdv_* d\sigma \Big)^{1/2} ||| h|||_{\Phi_\gamma} \\
&\lesssim U^{1/2}||| h|||_{\Phi_\gamma},
\end{align*}
where
$$
U = \iiint b\,  <v-v_*>^\gamma \mu^{a '}_*|f_*|^2 |g|^2 |W_l' -W_l|^2  dvdv_* d\sigma .
$$
If $l\ge 1$, by using the Taylor's formula, we have
$$
| W_l'- W_l |^2 \lesssim \min \lbrace \theta^2 | v-v_*|^2, \, ( <v>+ <v_*>)^2 \rbrace\, ( <v_*>^{l-1} + <v>^{l-1})^2\, ,
$$
and then
$$
\int_{\SS^2} b \, | W_l'- W_l |^2 d\sigma \lesssim | v-v_*|^{2s} ( <v>^{2l- 2s}+ <v_*>^{2l-2s}).
$$
Then we note immediately that $U$ is similar to the term $A_3$ in the proof of Proposition \ref{part2-prop5.2}, because we have a Gaussian
 inside the definition of $U$. Taking into account the weights here gives
\begin{align*}
 U &\lesssim  || f||_{L^2_{s+\gamma /2}} || W_{l-s} g ||_{L^2_{s+\gamma /2}} + \min \lbrace || f||_{L^2} || W_{l-s} g ||_{L^2_{s+\gamma /2}} ,\,
 || f||_{L^2_{s+\gamma /2}} || W_{l-s} g ||_{L^2} \rbrace\\
&+ || g||_{L^2_{s+\gamma /2}} || W_{l-s} f ||_{L^2_{s+\gamma /2}} + \min \lbrace || g||_{L^2} || W_{l-s} f ||_{L^2_{s+\gamma /2}} ,\, || g||_{L^2_{s+\gamma /2}} || W_{l-s} f ||_{L^2} \rbrace .
\end{align*}
If $0 < l\le 1$, then we note that
$$
| W'_l  - W_l |^2\lesssim\, < v>^{2l} + <v_*>^{2l} \mbox{ and } | W'_l -W_l |^2 \lesssim \theta^2 | v-v_*|^2 ,
$$
so that
$$
| W'_l -W_l |^2 \lesssim \min \lbrace \theta^2 | v-v_*|^2 ,\, < v>^{2l} + <v_*>^{2l} \rbrace .
$$
Then we obtain
\begin{equation}\label{part2-aa2}
\int_{\SS^2} b\, | W_l'- W_l |^2 d\sigma \lesssim | v-v_*|^{2s} ( <v>^{2(1-s)l} + <v_*>^{2 (1-s)l}),
\end{equation}
and therefore the same argument gives
\begin{align*}
U &\lesssim  || f||_{L^2_{s+\gamma /2}} || W_{l-sl} g ||_{L^2_{s+\gamma /2}} + \min \lbrace || f||_{L^2} || W_{l-sl} g ||_{L^2_{s+\gamma /2}} ,\,
 || f||_{L^2_{s+\gamma /2}} || W_{l-sl} g ||_{L^2} \rbrace\\
&+ || g||_{L^2_{s+\gamma /2}} || W_{l-sl} f ||_{L^2_{s+\gamma /2}} + \min \lbrace || g||_{L^2} || W_{l-sl} f ||_{L^2_{s+\gamma /2}} ,\, || g||_{L^2_{s+\gamma /2}} || W_{l-sl} f ||_{L^2} \rbrace,
\end{align*}
which gives the final conclusion, for $\gamma >-3$. Terms $B$ and $C$ can be dealt with similarly so that we omit the
details for brevity. And this completes the proof of the proposition.
\end{proof}


\section{Functional estimates in full space }\label{part2-section3}
\setcounter{equation}{0}

In this section, we prove the estimations on the collision operators in some weighted function space of
variables $(x, v)\in\RR^6$. Together with the essential coercivity estimates proved in Section \ref{part1-section2}, we give
coercivity results for the linear operator in some weighted spaces. These tools are crucial for the proofs of the
existence results, both in the local and global cases. Recall the assumption $\gamma+2s\leq 0$.

Let $N\in\NN, l\in\RR$, we define the weighted function spaces
\begin{align*}
\cB^N_\ell(\RR^6)&=\Big\{g\in\cS'(\RR^6_{x, v})\, ;\,\,
\|g\|^2_{{\cB}^N_\ell(\RR^6)}=\sum_{|\alpha|+|\beta|\leq N}\int_{\RR^3_x}|||W_{\ell-|\beta|}\,
\partial^\alpha_{\beta} g(x,\, \cdot\,)|||^2_{\Phi_\gamma}dx<+\infty\,\Big\}\, ,\\
{\widetilde\cB\,}^N_\ell(\RR^6)&=\Big\{ g\in\cS'(\RR^6);\,\,
||g||^2_{{\widetilde\cB\,}^N_\ell(\RR^6)}
=\sum_{|\alpha|+|\beta|\leq N}\int_{\RR^3_x}|||\tilde W_{\ell-|\beta|}\,
\partial^\alpha_{\beta} g(x,\, \cdot\,)|||^2_{\Phi_\gamma}dx <+\infty\Big\}\, ,
\end{align*}
and also
$$
\cX^{N}(\RR^6)=\Big\{ g\in\cS'(\RR^6);\,\,\
||g||_{\cX^N(\RR^6)}^2=\sum_{|\alpha|\le N}
\int_{\RR^3_x}|||\p_x^\alpha g|||^2_{\Phi_\gamma}dx<+\infty\Big \}.
$$

\subsection{Estimations without weight}
First of all, one has
\begin{lemm}\label{part2-lemm2.6} For all $0<s<1$, $\gamma >-3$, and for any $\alpha, \beta\in  \NN^3$,
\begin{equation}\label{part2-4-2-2.8}
||\partial^{\alpha}_\beta\,\pP
g||_{\cX^0(\RR^6)}+||\pP (
\partial^{\alpha}_\beta\,\, g)||_{\cX^0(\RR^6)}
\leq C_{\beta}||\partial^{\alpha}_x
g||^2_{L^2(\RR^6)},
\end{equation}
\begin{equation}\label{part2-4-2-2.9}
\frac{\eta_0}{2}||g||^2_{\cX^0(\RR^6)}-C||g||^2_{L^2(\RR^6)} \leq \Big(\cL
g\, ,g\Big)_{L^2(\RR^6)}\leq  2\Big(\cL_1 g,\, g\Big)_{L^2(\RR^6)}
 \lesssim ||g||^2_{\cX^0(\RR^6)},
\end{equation}
and
\begin{equation}\label{part2-4-2-2.10}
||g||^2_{L^2_{s+\frac{\gamma}{2}}(\RR^6)} +||g||^2_{L^2(\RR^3_x; H^s_{\frac{\gamma}{2}}(\RR^3_v))} \lesssim
||g||^2_{\cX^0(\RR^6)}\lesssim ||g||^2_{L^2(\RR^3_x;
H^{s}_{s+\frac{\gamma}{2}}(\RR^3_v))}.
\end{equation}
\end{lemm}

\begin{proof}

{}From \cite{guo-1}, one has
$$
\pP g = \Big( a_g (t,x)  + v.b_g (t,x) + | v|^2 c_g (t,x) \Big) \mu^{1/2},
$$
where
$$
a_g (t,x) = \int_{\RR^3} \big(2 - {{| v|^2}\over 2} \big) g((t,x,v) \mu^{1/2} (v) dv,
$$
$$
b_g (t,x) =  \int_{\RR^3} g(t,x,v) v\mu^{1/2} (v) dv,
$$
and
$$
c_g (t,x) = \int_{\RR^3} \big( {{| v|^2}\over 6} - {1\over2} \big) g(t,x,v) \mu^{1/2} (v) dv .
$$
Thus, \eqref{part2-4-2-2.8} can be obtained by using integration by parts.
To get \eqref{part2-4-2-2.9}, we use the results from Section \ref{part1-section2} to obtain
\begin{align*}
||g||^2_{\cX^0(\RR^6)}&\gtrsim \Big(\cL
g,\, g\Big)_{L^2(\RR^6)}
\geq \eta_0||(\iI-\pP) g||^2_{\cX^0(\RR^6)}\\
&\geq \frac{\eta_0}{2}||g||^2_{\cX^0(\RR^6)}-
C||\pP g||^2_{\cX^0(\RR^6)} \geq \frac{\eta_0}{2}||g||^2_{\cX^0(\RR^6)}
-C||g||^2_{L^2(\RR^6)}.
\end{align*}
Finally, \eqref{part2-4-2-2.10} follows directly {}from Section \ref{part1-section2}.
\end{proof}

The following Lemma is an application of the Sobolev embedding for functions with values in a Hilbert space.

\begin{lemm}\label{part2-lemm2.7}
\begin{equation*}
\sup_{x\in\RR^3}\||f(x,\,\cdot\,)\||_{\Phi_\gamma}\lesssim
\|f\|_{\cX^{\,2}(\RR^6)}.
\end{equation*}
\end{lemm}
\begin{proof}
It follows from the definition that
\begin{align*}
\Big(\sup_{x\in\RR^3}\||f(x,\,\cdot\,)\||_{\Phi_\gamma}\Big)^2 & \leq
 \iiint B\, \mu_*\, \Big(\sup_{x\in \RR^3}
\big(f(x,v')-f(x,v)\,\big)^2 \Big)dv dv_* d\sigma \,\\
 &+\iiint B\,
\Big(\sup_{x\in\RR^3} f(x,v_*)^2 \Big)\big(\sqrt{\mu'}\,\, - \sqrt{\mu}\,\,
\big)^2dv dv_* d\sigma  \\
&\lesssim
\iiint B\, \mu_*\, \Big(\sum_{|\alpha| \leq 2} \int
\big(\pa_x^\alpha f(x,v')-\pa_x^\alpha f(x,v)\,\big)^2 dx \Big)dv dv_* d\sigma \,\\
 &+\iiint B\,
\Big(\sum_{|\alpha| \leq 2} \int  f(x,v_*)^2 dx \Big)\big(\sqrt{\mu'}\,\, - \sqrt{\mu}\,\,
\big)^2dv dv_* d\sigma  \\
&\lesssim  \sum_{|\alpha| \leq 2} \int|||\pa_x^\alpha f(x,\cdot)|||^2_{\Phi_\gamma}dx\,.
\end{align*}
\end{proof}

\begin{prop}\label{part2-prop2.8}
Under the assumption
of Theorem \ref{part2-theo1.1}, for any $ N\geq 3$, we have, for all $\alpha\in\NN^3, |\alpha|\leq N$,
\begin{align*}
\left|\Big(\partial^\alpha_{x} \Gamma(f,\, g\,),\,\,
h\Big)_{L^2(\RR^6)}\right|&\lesssim \Big\{ ||f||_{H^N(\RR^3_x; L^2(\RR^3_v))}\,\,
|| g||_{\cX^{N}(\RR^6)} \\
&+||f||_{\cX^N(\RR^6)}\,\,
|| g||_{H^N(\RR^3_x; L^2(\RR^3_v))}\Big\}\,\, || h||_{\cX^0(\RR^6)}.
\end{align*}
\end{prop}

\begin{proof} Firstly, if $ |\alpha_1|\leq N-2$, we get from Theorem \ref{theorem-1.1-b},
Lemma \ref{part2-lemm2.7}, and usual Sobolev embedding, replacing the "min" term by the corresponding terms without the weights that
\begin{align*}
& \left|\Big(\Gamma(\partial^{\alpha_1} f,\, \partial^{\alpha_2} g),\,\, h\Big)_{L^2(\RR^6)}\right|
\\
&\lesssim \left(\int_{\RR^3_x}\Big( \|\partial^{\alpha_1}
f\|^2_{L^2}||| \partial^{\alpha_2} g|||^2_{\Phi_\gamma}
+\||\partial^{\alpha_1} f\||^2_{\Phi_\gamma}\|\partial^{\alpha_2} g||^2_{L^2}
{ + ||\partial^{\alpha_1} f||_{L^2_{}}^2 || \partial^{\alpha_2} g ||^2_{L^2} }\Big)dx
\right)^{1/2}||h||_{\cX^{\,0}(\RR^6)} \\
&\lesssim \Big(\|f\|_{H^{|\alpha_1|+3/2+\epsilon}(\RR^3_x; L^2(\RR^3_v))}
||g||_{\cX^{|\alpha_2|}(\RR^6)}
+||f||_{\cX^{|\alpha_1|+2}(\RR^6)}||g||_{H^{|\alpha_2|}(\RR^3_x; L^2(\RR^3_v))}\Big)
||h||_{\cX^0(\RR^6)}\, .
\end{align*}
If $|\alpha_1|=N-1, N$, then $|\alpha_2|+2\leq N$, we get in a similar way, again from Theorem \ref{theorem-1.1-b} that
\begin{align*}
& \left|\Big(\Gamma(\partial^{\alpha_1} f,\, \partial^{\alpha_2} g),\,\, h\Big)_{L^2(\RR^6)}\right|\\
&\lesssim \left(\int_{\RR^3_x}\Big( \|\partial^{\alpha_1}
f\|^2_{L^2}||| \partial^{\alpha_2} g|||^2_{\Phi_\gamma}
+\||\partial^{\alpha_1} f\||^2_{\Phi_\gamma}\|\partial^{\alpha_2} g||^2_{L^2}
{ + ||\partial^{\alpha_1} f||_{L^2}^2 || \partial^{\alpha_2} g ||^2_{L^2} }\Big)dx
\right)^{1/2}||h||_{\cX^{\,0}(\RR^6)} \\
&\lesssim \Big(\|f\|_{H^{|\alpha_1|}(\RR^3_x; L^2(\RR^3_v))}
||g||_{\cX^{|\alpha_2|+2}(\RR^6)}
+||f||_{\cX^{|\alpha_1|}(\RR^6)}||g||_{H^{|\alpha_2|+3/2+\epsilon}(\RR^3_x; L^2(\RR^3_v))}\Big)
||h||_{\cX^0(\RR^6)}\, .
\end{align*}
The proof of the proposition is then completed by recalling the Leibniz formula
$$
\partial_x^\alpha \Gamma (f,g) = \sum_{\alpha_1 +\alpha_2 =\alpha} C_{\alpha_1 , \alpha_2} \Gamma(\partial_x^{\alpha_1} f,\,\,
\partial_x^{\alpha_2} g)\,.
$$
\end{proof}
\begin{rema}
The above proof shows that, for $|\alpha_1|<N$,
\begin{equation*}
\left|\Big(\Gamma(\partial^{\alpha_1} f,\, \partial^{\alpha_2}
g\,)
,\,\, h\Big)_{L^2(\RR^6)}\right|
\lesssim\| f\|_{H^{N}(\RR^3_x; L^2(\RR^3_v))} \,||h||_{\cX^0(\RR^6)} \Big(||g||_{\cX^N(\RR^6)}
+||g||_{H^{N}(\RR^3_x; L^2(\RR^3_v))}\Big)\, .
\end{equation*}
\end{rema}

Finally, the estimate on
the linear operator $\cL_2$ can be  given as follows, which in fact can be  deduced from Section \ref{part1-section2}.
\begin{prop}\label{part2-prop2.10} For all $0< s<1$, $\gamma >-3$ and any $\alpha\in\NN^3$, we have
\begin{equation*}
\left|\Big(\partial^\alpha_{x} \cL_2(f),\,\,
h\Big)_{L^2(\RR^6)}\right|\leq C_{|\alpha|} ||f||_{H^{|\alpha|}(\RR^3_x; L^2(\RR^3_v))}\,\, ||\mu^{1/10^3}
h||_{L^2(\RR^6)}\, .
\end{equation*}
\end{prop}

\subsection{Estimation with weight}
We now prove the following upper bound with weights.
\begin{prop}\label{part2-prop5.5}
For all $0<s <1$,  $\gamma >-3$, and for any $N\geq 6, \ell\geq N$, $|\alpha|+|\beta|\leq N$,
we have
\begin{align*}
&\big|\big(W_{\ell-|\beta|}\partial^\alpha_\beta\,\Gamma(f,\, g),\, h\big)_{L^2(\RR^6)}\big|\lesssim ||h||_{\cB^0_0(\RR^6)}\Big(\|f\|_{ \cH^N_\ell(\RR^6)}\|g\|_{ \cH^N_\ell(\RR^6)}
\\
&\qquad\qquad\qquad+\|f\|_{\cH^N_\ell(\RR^6)}\|g\|_{{\cB\,}^N_\ell(\RR^6)}
+\|g\|_{{ \cH\,}^N_\ell(\RR^6)}\|f\|_{{ \cB\,}^N_\ell(\RR^6)}
\Big).\nonumber
\end{align*}
\end{prop}

\begin{proof} By using Leibniz formula, we have
\begin{align*}
&(W_{\ell-|\beta|}\partial^\alpha_\beta\,\Gamma(f,\, g),\, h)=
\sum C^{\alpha_1, \beta_1}_{\alpha_2, \beta_2,\beta_3}
(W_{\ell-|\beta|}\,\cT(\partial^{\alpha_1}_{\beta_1} f,\, \partial^{\alpha_2}_{\beta_2} g, \mu_{\beta_3}),\, h)\\
&=\sum C^{\alpha_1, \beta_1}_{\alpha_2, \beta_2,\beta_3}
(W_{\ell-|\beta|}\,\cT(\partial^{\alpha_1}_{\beta_1} f,\, \partial^{\alpha_2}_{\beta_2} g, \mu_{\beta_3})-\cT(\partial^{\alpha_1}_{\beta_1} f,\, W_{\ell-|\beta|}\,\partial^{\alpha_2}_{\beta_2} g, \mu_{\beta_3}),\, h)\\
&\qquad\qquad\qquad\qquad+\sum C^{\alpha_1, \beta_1}_{\alpha_2, \beta_2,\beta_3}
(\cT(\partial^{\alpha_1}_{\beta_1} f,\, W_{\ell-|\beta|}\,\partial^{\alpha_2}_{\beta_2} g, \mu_{\beta_3}),\, h)\,.
\end{align*}
Note that $\cT$ shares the same upper bound properties as $\Gamma$ given in
the previous propositions. In fact, by using Proposition \ref{part2-weight-radja-lemma-6.1}
\begin{align*}
A=&\big|(W_{\ell-|\beta|}\,\cT(\partial^{\alpha_1}_{\beta_1} f,\,
\partial^{\alpha_2}_{\beta_2} g, \mu_{\beta_3})-\cT(\partial^{\alpha_1}_{\beta_1} f,\,
W_{\ell-|\beta|}\,\partial^{\alpha_2}_{\beta_2} g, \mu_{\beta_3}),\, h)_{L^2(\RR^6)}\big|\\
&\lesssim ||h||_{\cB^0_0(\RR^6)}\Big\{\int_{\RR^3}
\Big( \|W_{\ell-|\beta|+\gamma/2}\partial^{\alpha_1}_{\beta_1} f(x, \cdot\,)\|^2_{L^2(\RR^3_v)} \|\partial^{\alpha_2}_{\beta_2}g(x, \cdot\,)\|^2_{L^2_{\gamma/2+s}(\RR^3_v)} \\
&+ \|\partial^{\alpha_1}_{\beta_1}f(x, \cdot\,)\|^2_{L^2_{\gamma/2+s}(\RR^3_v)} \|W_{\ell-|\beta|+\gamma/2} \partial^{\alpha_2}_{\beta_2}g(x, \cdot\,)\|^2_{L^2(\RR^3_v)} \\
& + \min\Big\{\|\partial^{\alpha_1}_{\beta_1} f(x, \cdot\,)\|^2_{L^2(\RR^3_v)} \|W_{l-|\beta|+\gamma/2}\partial^{\alpha_2}_{\beta_2}g(x, \cdot\,)\|^2_{L^{2}(\RR^3_v)}\, ,\\
&\quad\|\partial^{\alpha_1}_{\beta_1} f(x, \cdot\,)\|^2_{L^2_{\gamma/2+s}(\RR^3_v)} \|W_{l-s}\partial^{\alpha_2}_{\beta_2}g(x, \cdot\,)\|^2_{L^2(\RR^3_v)}\Big\}\\
& + \min\Big\{\|\partial^{\alpha_2}_{\beta_2}g(x, \cdot\,)\|^2_{L^2(\RR^3_v)}\|W_{l-|\beta|+\gamma/2}\partial^{\alpha_1}_{\beta_1} f(x, \cdot\,)\|^2_{L^2(\RR^3_v)},\,\\
&\quad\|\partial^{\alpha_2}_{\beta_2}g(x, \cdot\,)\|^2_{L^2_{\gamma/2+s}(\RR^3_v)}
\|W_{l-|\beta|-s}\partial^{\alpha_1}_{\beta_1} f(x, \cdot\,)\|^2_{L^2(\RR^3_v)} \Big\}
\Big)dx\Big\}^{1/2}.
\end{align*}
For this, we divide the discussion into two cases.

{\bf Case 1: $|\alpha_1|+|\beta_1|\leq N-2$}. We have, by using $H^2(\RR^3_x)\subset L^\infty(\RR^3_x)$,
 and $\gamma+2s\leq 0$ that
\begin{align*}
A\lesssim &||h||_{\cB^0_0(\RR^6)}\Big( \|W_{\ell-|\beta|+\gamma/2}\nabla^{2}_{x} \partial^{\alpha_1}_{\beta_1} f\|_{L^2(\RR^6)} \|\partial^{\alpha_2}_{\beta_2}g\|_{L^2(\RR^6)} \\
&+ \|\nabla^2_x \partial^{\alpha_1}_{\beta_1}f\|_{L^2(\RR^6)} \|W_{\ell-|\beta|+\gamma/2} \partial^{\alpha_2}_{\beta_2}g\|_{L^2(\RR^6)}
\Big)\\
&\lesssim ||h||_{\cB^0_0(\RR^6)}\|f\|_{ \cH^N_\ell(\RR^6)}\|g\|_{\cH^N_\ell(\RR^6)}.
\end{align*}

{\bf Case 2: $|\alpha_1|+|\beta_1|> N-2$}. Then $|\alpha_2|+|\beta_2|\leq 1$ and $|\alpha_2|+|\beta_2|+5\leq N$, and we have
\begin{align*}
A\lesssim &||h||_{\cB^0_0(\RR^6)}\Big( \|W_{\ell-|\beta|+\gamma/2}
\partial^{\alpha_1}_{\beta_1} f\|_{L^2(\RR^6)} \|\nabla^{2}_{x}
\partial^{\alpha_2}_{\beta_2}g\|_{L^2(\RR^6)} \\
&+ \|\partial^{\alpha_1}_{\beta_1}f\|_{L^2(\RR^6)}
\|W_{\ell-|\beta|+\gamma/2}\nabla^{2}_{x}
\partial^{\alpha_2}_{\beta_2}g\|_{L^2(\RR^6)}
\Big)\\
&\lesssim ||h||_{\cB^0_0(\RR^6)}\|f\|_{H^\cH_\ell(\RR^6)}\|g\|_{ \cH^N_\ell(\RR^6)}.
\end{align*}
Therefore, we get
\begin{align}\label{part2-4-2-5.4+1}
&\sum C^{\alpha_1, \beta_1}_{\alpha_2, \beta_2,\beta_3}
\big|\big(W_{\ell-|\beta|}\,\cT(\partial^{\alpha_1}_{\beta_1} f,\,
\partial^{\alpha_2}_{\beta_2} g, \mu_{\beta_3})-
\cT(\partial^{\alpha_1}_{\beta_1} f,\, W_{\ell-|\beta|}\,
\partial^{\alpha_2}_{\beta_2} g, \mu_{\beta_3}),\, h\big)\big|\\
&\qquad\qquad\leq C||h||_{\cB^0_0(\RR^6)}\|f\|_{\cH^N_\ell(\RR^6)}
\|g\|_{\cH^N_\ell(\RR^6)}\, .\nonumber
\end{align}
Next, we deal with the following term and the discussion is
also divided into several steps:
$$\sum C^{\alpha_1, \beta_1}_{\alpha_2, \beta_2,\beta_3}
(\cT(\partial^{\alpha_1}_{\beta_1} f,\, W_{\ell-|\beta|}\,\partial^{\alpha_2}_{\beta_2} g, \mu_{\beta_3}),\, h)\,.
$$

{\bf Case I: $|\alpha_1|+|\beta_1|\leq 1$}. From Proposition \ref{part2-prop5.1}, we get
\begin{align*}
B=&\big | \big(\cT(\partial^{\alpha_1}_{\beta_1} f,\, W_{\ell-|\beta|}\,
\partial^{\alpha_2}_{\beta_2} g, \mu_{\beta_3}),\, h \big)_{L^2(\RR^6)} \big|\\
&\lesssim ||h||_{\cB^0_0}
\Big\{ \|f\|_{H^5(\RR^6)} ||W_{\ell-|\beta|}\,\partial^{\alpha_2}_{\beta_2} g||_{\cB^0_0(\RR^6)} +
\| f \|_{H^6(\RR^6)}
\|W_{\ell-|\beta|}\,\partial^{\alpha_2}_{\beta_2} g\|_{L^2_{ s+\gamma/2}(\RR^6)} \Big \}\\
&\lesssim ||h||_{\cB^0_0}
\|f\|_{{\cH\, }^N_\ell(\RR^6)} ||g||_{{\cB\,}^N_\ell(\RR^6)}\,.
\end{align*}

{\bf Case II: $2\leq |\alpha_1|+|\beta_1|\leq 3$}. Again from Proposition \ref{part2-prop5.1}, one has
\begin{align*}
&B\lesssim ||h||_{\cB^0_0}
\Big\{ \|f\|_{H^5(\RR^6)} ||W_{\ell-|\beta|}\nabla^2_x\,\partial^{\alpha_2}_{\beta_2} g||_{\cB^0_0(\RR^6)} +
\| f \|_{H^6(\RR^6)}
\|W_{\ell-|\beta|}\,\nabla^2_x\partial^{\alpha_2}_{\beta_2} g\|_{L^2_{ s+\gamma/2}(\RR^6)} \Big \}\\
&\lesssim ||h||_{\cB^0_0}
\|f\|_{{\cH\, }^N_\ell(\RR^6)} ||g||_{{\cB\,}^N_\ell(\RR^6)}\,.
\end{align*}

{\bf Case III: $4\leq |\alpha_1|+|\beta_1|\leq N-1$}. Then
$|\alpha_2|+|\beta_2|\leq N-4$, and from Proposition \ref{part2-prop5.2}, we get
\begin{align*}
&B\lesssim ||h||_{\cB^0_0}
\Big\{ \|\partial^{\alpha_1}_{\beta_1} f\|_{L^2_{\gamma/2+s}(\RR^6)}
||W_{\ell-|\beta|}\nabla^2_x\,\partial^{\alpha_2}_{\beta_2} g||_{\cB^0_0(\RR^6)} +
\|\partial^{\alpha_1}_{\beta_1}f \|_{\cB^0_0(\RR^6)}
\|W_{\ell-|\beta|}\,\nabla^2_x\partial^{\alpha_2}_{\beta_2} g\|_{L^2_{s+\gamma/2}(\RR^6)} \Big \}\\
&+ \|\partial^{\alpha_1}_{\beta_1} f\|_{L^2(\RR^6)}
||W_{\ell-|\beta|}\nabla^2_x\,\partial^{\alpha_2}_{\beta_2} g||_{L^2_{\gamma/2+s}(\RR^6)}
+ \|\mu^{1/60}\partial^{\alpha_1}_{\beta_1} f\|_{L^2(\RR^6)}
||\mu^{1/60}W_{\ell-|\beta|}\,\partial^{\alpha_2}_{\beta_2} g||_{H^{3/2+\epsilon}(\RR^3_x; H^{3/2-\epsilon}(\RR^3_v))}\\
&+ \|\mu^{1/60}\partial^{\alpha_1}_{\beta_1} f\|_{H^s(\RR^6)}
||\mu^{1/60}W_{\ell-|\beta|}\,\partial^{\alpha_2}_{\beta_2} g||_{H^{3+\epsilon}(\RR^6)}\lesssim ||h||_{\cB^0_0}
\|f\|_{{\cH\, }^N_\ell(\RR^6)} ||g||_{{H\, }^N_\ell(\RR^6)}\,.
\end{align*}

{\bf Case IV: $|\alpha_1|+|\beta_1|=|\alpha|+|\beta|\leq N$}. Again from Proposition \ref{part2-prop5.2}, we have
\begin{align*}
&B=\big | \big(\Gamma(\partial^\alpha_\beta f,\, W_{\ell-|\beta|} g),\, h \big )_{L^2(\RR^6)}  \big|
\lesssim ||h||_{\cB^0_0(\RR^6)}
\Big\{ \|\partial^\alpha_\beta f\|_{L^2_{s+\gamma/2}(\RR^6)}||W_{\ell-|\beta}\nabla^2_x g||_{\cB^0_0(\RR^6)}\\
&\qquad+ \|W_{\ell-|\beta|}\nabla^2_x g\|_{L^2_{s+\gamma/2}(\RR^6)}||\partial^\alpha_\beta f||_{\cB^0_0(\RR^6)}+ \|\partial^\alpha_\beta f\|_{L^2_{s+\gamma/2}(\RR^6)}\|W_{\ell-|\beta}\nabla^2_x g\|_{L^2(\RR^6)} \\
&\qquad+
\|\mu^{1/40}\partial^\alpha_\beta f\|_{L^2_{s+\gamma/2}(\RR^6)}
\|\mu^{1/40} \,W_{\ell-|\beta|}g\|_{H^{3/2+\epsilon}(\RR^3_x; H^{\max (-\gamma/2,\, 1)}(\RR^3_v))}\\
 &\qquad\qquad\qquad\qquad+   ||\mu^{1/40}\partial^\alpha_\beta f||_{\cB^0_0(\RR^6)}
 \|\mu^{1/40} \,W_{\ell-|\beta|} g\|^2_{L^\infty(\RR^6)}
\Big \}\\
&\qquad\lesssim ||h||_{\cB^0_0(\RR^6)}
\|g\|_{{\cH\,}^N_\ell(\RR^6)}\|f\|_{{\cB\,}^N_\ell(\RR^6)}\,.
\end{align*}
Finally, the proof of the proposition is completed.
\end{proof}

\subsection{Estimation with modified weight}
In the sequel, we shall often use the inequalities
\begin{equation}\label{part2-sgamma}
\|f\|_{L^2_{s+\gamma/2}}\le \|f\|_{L^2}, \quad
 \|f\|_{L^2_{s+\gamma/2}}\le |||f|||.
\end{equation}
The first inequality is valid since  we are assuming
that $\gamma +2s \le 0$, while the second one comes
from Section \ref{part1-section2}.

To obtain the global existence, the upper bound given in Proposition \ref{part2-prop5.5} for the general case is not enough because
this bound can not be controlled by  the coercivity
estimate of the linearized operator that contains loss of weight in the case of soft potential.
To overcome this,
we now prove  the following upper bound  under the conditions of Theorem \ref{part2-theo1.2}.

\begin{prop}\label{part2-upperGG}
Under the assumptions  of Theorem \ref{part2-theo1.2} on the parameters $\gamma $ and $s$, for any $N\geq 6, \ell\geq N$, $|\alpha|+|\beta|\leq N$, one has
 \begin{align}\label{part2-WpGamma}
|(\tilde{W}_{l-|\beta|}\p^\alpha_\beta\Gamma(f,g), \tilde{W}_{l-|\beta|}\p^\alpha_\beta h)_{L^2( \RR^6)}|
\le &
\Big(\|f\|_{\tilde{\cH}^N_l(\RR^6)}\ ||g||_{\tilde{\cB}^N_l(\RR^6)}
\\&\hspace{1cm}+\|g\|_{\tilde{\cH}^N_l(\RR^6)}\ ||f||_{\tilde{\cB}^N_l(\RR^6)}\Big)
||h||_{\tilde{\cB}^N_l(\RR^6)}.\notag
\end{align}
\end{prop}
\begin{proof}
First, notice that from Remark \ref{part2-weight-radja-lemma-6.1+1} and \eqref{part2-sgamma},
we have for $\gamma>-3$,
\begin{align}\label{part2-comm1}
|(\tilde{W}_l \Gamma(f,g)-\Gamma(f,\tilde{W}_l g), h)_{L^2(\RR^3_v)}|
\le \Big(\|f\|_{L^2(\RR^3_v)}|||\tilde{W}_l  g|||_{\Phi_\gamma}+|||\tilde{W}_l f|||_{\Phi_\gamma}
\ \|g\|_{L^2(\RR^3_v)}||\Big)|||h|||_{\Phi_\gamma}.
\end{align}

Recall the definition of $\cT$ to deduce that
\begin{align*}
&|(\tilde{W}_l\p_\beta\Gamma(f,g), \tilde{W}_l\p_\beta h)_{L^2(\RR^3_v)} | \le
 \sum_{\beta_1+\beta_2+\beta_3=\beta}|(\tilde{W}_l\cT(\p_{\beta_1}f,
 \p_{\beta_2}g, \p_{\beta_3}\mu^{1/2}), \tilde{W}_l\p_\beta h)|_{L^2(\RR^3_v)}|\\
 &\quad=\sum \Big\{
|(\cT(\p_{\beta_1}f,\tilde{W}_l\p_{\beta_2}g, \p_{\beta_3}\mu^{1/2}), \tilde{W}_l\p_\beta h)_{L^2(\RR^3_v)}|\\
&\qquad +|(\tilde{W}_l\cT(\p_{\beta_1}f,\p_{\beta_2}g, \p_{\beta_3}\mu^{1/2})
-(\cT(\p_{\beta_1}f,\tilde{W}_l\p_{\beta_2}g, \p_{\beta_3}\mu^{1/2}), \tilde{W}_l\p_\beta h)_{L^2(\RR^3_v)}|\Big\}.
\end{align*}
By using Theorem \ref{theorem-1.1-b}, we obtain
\begin{align*}
 &|(\Gamma(\p_{\beta_1}f,\tilde{W}_l\p_{\beta_2}g), \tilde{W}_l\p_\beta h)|_{L^2(\RR^3_v)}
\\&
\le\Big(\|\p_{\beta_1}f\|_{L^2(\RR^3_v)}\ |||\tilde{W}_l\p_{\beta_2}g|||_{\Phi_\gamma}
+
||| \p_{\beta_1}f|||_{\Phi_\gamma}\ \|\tilde{W}_l\p_{\beta_2}g\|_{L^2(\RR^3_v)}\Big)|||\tilde{W}_l\p_{\beta}h|||_{\Phi_\gamma}.
\end{align*}
Moreover, \eqref{part2-comm1} implies that
\begin{align*}
&|(\tilde{W}_l \Gamma(\p_{\beta_1}f,\p_{\beta_2}g)-\Gamma(\p_{\beta_1}f,\tilde{W}_l\p_{\beta_2}g),
\tilde{W}_l\p_\beta h)_{L^2(\RR^3_v)}|
\\
&\le \Big(\|\p_{\beta_1}f\|_{L^2(\RR^3_v)}|||\tilde{W}_l \p_{\beta_2}g|||_{\Phi_\gamma}+|||\tilde{W}_l\p_{\beta_1}f|||_{\Phi_\gamma}
\ \|\p_{\beta_2}g\|_{L^2(\RR^3_v)}\Big)|||\tilde{W}_l\p_{\beta}h|||_{\Phi_\gamma}.
\end{align*}
As a consequence,
\begin{align}
|(\tilde{W}_l\p_\beta&\Gamma(f,g), \tilde{W}_l\p_\beta h)_{L^2(\RR^3_v)} |
 \le\sum \Big(\|\p_{\beta_1}f\|_{L^2(\RR^3_v)}\ |||\tilde{W}_l\p_{\beta_2}g|||_{\Phi_\gamma}
\\&+\notag
||| \p_{\beta_1}f|||_{\Phi_\gamma}\ \|\tilde{W}_l\p_{\beta_2}g\|_{L^2(\RR^3_v)}
+|||\tilde{W}_l\p_\beta f|||_{\Phi_\gamma}\ \|\p_{\beta_2}g\|_{L^2(\RR^3_v)}\Big)|||\tilde{W}_l\p_{\beta}h|||_{\Phi_\gamma}.
\end{align}
By the Leibniz rule in $x$ variable, one has
\begin{align*}
&|(\tilde{W}_{l-|\beta|}\p^\alpha_\beta\Gamma(f,g), \tilde{W}_{l-|\beta|}\p^\alpha_\beta h)_{L^2(\RR^6)}|
\\&
\le
\sum_{\alpha_1+\alpha_2=\alpha}|(\tilde{W}_{l-|\beta|}\p_\beta\Gamma(\p^{\alpha_1}f,
\p^{\alpha_2}g), \tilde{W}_{l-|\beta|}\p^\alpha_\beta h)_{L^2(\RR^6)}|
\\&\le\sum\int_{\RR^3}
\Big(\|\p^{\alpha_1}_{\beta_1}f\|_{L^2(\RR^3_v)}
\ |||\tilde{W}_{l-|\beta|}\p^{\alpha_2}_{\beta_2}g|||_{\Phi_\gamma}
+||| \p^{\alpha_1}_{\beta_1}f|||_{\Phi_\gamma}\ \|\tilde{W}_{l-|\beta|}\p^{\alpha_2}_{\beta_2}g\|_{L^2(\RR^3_v)}
\\&\hspace{1cm}
+|||\tilde{W}_{l-|\beta|}\p^{\alpha_1}_{\beta_1}f|||_{\Phi_\gamma}
\ \|\p^{\alpha_2}_{\beta_2}g\|_{L^2(\RR^3_v)}\Big)
|||\tilde{W}_{l-|\beta|}\p^\alpha_{\beta}h|||_{\Phi_\gamma}dx \notag
\\&
\le\sum\int_{\RR^3}
\Big(\|\tilde{W}_{l-\beta_1}\p^{\alpha_1}_{\beta_1}f\|_{L^2(\RR^3_v)}\
|||\tilde{W}_{l-\beta_2}\p^{\alpha_2}_{\beta_2}g|||_{\Phi_\gamma}
+||| \tilde{W}_{l-\beta_1}\p^{\alpha_1}_{\beta_1}f|||_{\Phi_\gamma}\
\|\tilde{W}_{l-\beta_2}\p^{\alpha_2}_{\beta_2}g\|_{L^2(\RR^3_v)}
\\&\hspace{1cm}
+|||\tilde{W}_{l-\beta_1}\p^{\alpha_1}_{\beta_1}f|||_{\Phi_\gamma}
\ \|\tilde{W}_{l-\beta_2}\p^{\alpha_2}_{\beta_2}g\|_{L^2(\RR^3_v)}\Big)|||\tilde{W}_{l-|\beta|}
\p^\alpha_{\beta}h|||_{\Phi_\gamma} dx \notag\\
&\le\sum\int_{\RR^3}
\Big(\|\tilde{W}_{l-\beta_1}\p^{\alpha_1}_{\beta_1}f\|_{L^2(\RR^3_v)}
\ |||\tilde{W}_{l-\beta_2}\p^{\alpha_2}_{\beta_2}g|||_{\Phi_\gamma}
\\&\hspace{1cm}+||| \tilde{W}_{l-\beta_1}\p^{\alpha_1}_{\beta_1}f|||_{\Phi_\gamma}
\ \|\tilde{W}_{l-\beta_2}\p^{\alpha_2}_{\beta_2}g\|_{L^2(\RR^3_v)}
\Big)|||\tilde{W}_{l-|\beta|}\p^\alpha_{\beta}h|||_{\Phi_\gamma} dx\notag
\\&=\sum\Big(G^1_{\alpha_1,\beta_1,\beta_2}+
G^2_{\alpha_1,\beta_1,\beta_2}\Big).
\end{align*}
Now these terms are discussed in the following two cases.

\noindent$\bullet$ When $|\alpha_1|+|\beta_1|\le N/2$, we have
\begin{align*}
G^1_{\alpha_1,\beta_1,\beta_2}\le&
\|\tilde{W}_{l-\beta_1}\p^{\alpha_1}_{\beta_1}f\|_{L^\infty(\RR^3_x;L^2(\RR^3_v))}\
\int_{\RR^3}|||\tilde{W}_{l-\beta_2}\p^{\alpha_2}_{\beta_2}g|||_{\Phi_\gamma} \
|||\tilde{W}^{l-|\beta|}\p^\alpha_{\beta}h|||_{\Phi_\gamma} dx
\\&\le ||f||_{\tilde{\cH}^N_l(\RR^6)}   ||g||_{\tilde{\cB}^N_l(\RR^6)}
||h||_{\tilde{\cB}^N_l(\RR^6)} .
\\
G^2_{\alpha_1,\beta_1,\beta_2}\le&
\| \ |||\tilde{W}_{l-\beta_1}\p^{\alpha_1}_{\beta_1}f|||_{\Phi_\gamma}\ \|_{L^\infty(\RR^3_x)}
\int_{\RR^3}\|\tilde{W}_{l-\beta_2}\p^{\alpha_2}_{\beta_2}g\|_{L^2_v}
|||\tilde{W}^{l-|\beta|}\p^\alpha_{\beta}h|||_{\Phi_\gamma} dx
\\&\le
||f||_{\tilde{\cB}^N_l(\RR^6)}
||g||_{\tilde{\cH}^N_l(\RR^6)}   ||h\|_{\tilde{\cB}^N_l(\RR^6)}  .
\end{align*}
$\bullet$ When $|\alpha_1|+|\beta_1|\ge N/2$, we have
\begin{align*}
G^1_{\alpha_1,\beta_1,\beta_2}\le&
\| \ |||\tilde{W}_{l-\beta_2}\p^{\alpha_2}_{\beta_2}g|||_{\Phi_\gamma}\ \|_{L^\infty(\RR^3_x)}
\int_{\RR^3} \|\tilde{W}_{l-\beta_1}\p^{\alpha_1}_{\beta_1}f\|_{L^2_v}
\ \ |||\tilde{W}^{l-|\beta|}\p^\alpha_{\beta}h|||_{\Phi_\gamma} dx
\\&\le||f||_{\tilde{\cH}^N_l(\RR^6)}\| g\|_{\tilde{\cB}^N_l(\RR^6)}   ||h||_{\tilde{\cB}^N_l(\RR^6)} .
\\
G^2_{\alpha_1,\beta_1,\beta_2}\le&
\|\tilde{W}_{l-\beta_2}\p^{\alpha_2}_{\beta_2}g\|_{L^\infty_x(L^2_v)}
\int_{\RR^3} |||\tilde{W}_{l-\beta_1}\p^{\alpha_1}_{\beta_1}f|||_{\Phi_\gamma}\
|||\tilde{W}^{l-|\beta|}\p^\alpha_{\beta}h|||_{\Phi_\gamma} dx
\\&\le
||f||_{\tilde{\cB}^N_l(\RR^6)}  ||g||_{\tilde{\cH}^N_l(\RR^6)}
  ||h\|_{\tilde{\cB}^N_l(\RR^6)}   .
\end{align*}
Here, we have used Lemma \ref{part2-lemm2.7} to get
\begin{align*}
\| \ |||\tilde{W}_{l-\beta_1}\p^{\alpha_1}_{\beta_1}f|||_{\Phi_\gamma}\ \|_{L^\infty(\RR^3_x)}\
\le \|f\|_{\tilde{\cB}^N(\RR^6)},
\end{align*}
for $|\alpha_1|+|\beta_1|\le N/2$. Therefore, we complete the proof of the
proposition.
\end{proof}

\subsection{Weighted coercivity of the linearized operator}
We turn to the weighted lower estimates, more precisely,
 the lower bound for $(W_l \cL g , W_l g)_{L^2(\RR^3_v)}$.
Let us recall that  in Section \ref{part1-section2} it was shown that, if $\gamma >-3$, then there exists a constant $C>0$ such that
\begin{equation}\label{part2-aa1}
|||g |||^2_{\Phi_\gamma}\geq \Big(\cL_1 g,\, g\Big)_{L^2(\RR^3_v)} \geq  \eta_0 |||g |||_{\Phi_\gamma}^2 - C
\|g\|_{L^2_{\gamma/2}(\RR^3_v)}^2\,.
\end{equation}

In the estimation on the weighted linearized collisional operator $\cL_1$, we need to consider the commutator estimate
of the weight and the linearized operator. However, we can not apply the Proposition \ref{part2-weight-radja-lemma-6.1} directly because the error
term then will have  the weight of the order of $s +\gamma/2$. The purpose of the following proposition is to show that
the error term coming from this commutator has the weight
of  order  $\gamma/2$ only.

\begin{prop}\label{part2-weight-radja-proposition-8.1}
For all $0<s<1$, $\gamma >-3$, and for any $l\ge 0$, there exists a positive constant
$C$ such that
\begin{equation}\label{part2-aa3}
(W_l \cL_1 g, W_l g)_{L^2(\RR^3_v)}
\ge \frac{\eta_0}{2} ||| W_l g|||^2_{\Phi_\gamma} - C|| W_l g||^2_{L^2_{\gamma /2}(\RR^3_v)} .
\end{equation}
Moreover, for any $\beta\in\NN^3\setminus \{0\}$, one has
\begin{align*}
(W_l\partial_\beta \cL_1 (g), W_l \partial_\beta g)_{L^2(\RR^3_v)} &\ge \frac{\eta_0}{2}
||| W_l \partial_\beta g|||^2_{\Phi_\gamma} - C|| W_l \partial_\beta g||^2_{L^2_{\gamma /2}(\RR^3_v)} \\
 &-C\Big( \sum_{\beta_1 <\beta} ||| W_l \partial_{\beta_1}
 g|||_{\Phi_\gamma} \Big) ||| W_l \partial_\beta g |||_{\Phi_\gamma} .
\end{align*}
\end{prop}

\begin{proof}
According to \eqref{part2-aa1}, it is enough to show that
$$
|([W_l, \cL_1] g, W_l g)_{L^2} |\le C_\delta || W_l g||^2_{L^2_{\gamma /2}} +\delta || W_l g||^2_{L^2_{s+\gamma /2}},
$$
where $\delta>0$ can be arbitrarily small.
By using the above expression, one has
\begin{align*}
(W_l \cL_1 g ,& W_l g)_{L^2} = - \iiint b\,\Phi_\gamma \Big((\mu'_*)^{1/2} g'-
(\mu_\ast)^{1/2} g \Big) \mu^{1/2}_* W_l W_l g dv_*d\sigma dv\\
& = - \iiint b\,\Phi_\gamma \Big((\mu'_*)^{1/2} W_l g'-
(\mu_\ast)^{1/2} W_lg \Big) \mu^{1/2}_* W_l  g dv_*d\sigma dv\\
& = - \iiint b\,\Phi_\gamma \Big( [ (\mu'_*)^{1/2} W'_l g'-
(\mu_\ast)^{1/2} W_lg ] + (\mu'_*)^{1/2}g' ( W_l - W'_l )\Big) \mu^{1/2}_* W_l  g \\
& = {1/2} \iiint b\,\Phi_\gamma \Big( (\mu'_*)^{1/2} W'_l g'-
(\mu_\ast)^{1/2} W_lg \Big)^2 dv_*d\sigma dv\\
& - \iiint b\,\Phi_\gamma \Big( (\mu'_*)^{1/2}g' ( W_l - W'_l )\Big) \mu^{1/2}_* W_l  g dv_*d\sigma dv\\
& = (\cL_1 (W_l g), (W_lg))_{L^2} + I\, ,\\
\end{align*}
where
\begin{align*}
 I=([W_l, \cL_1] g, W_l g)_{L^2}= - \iiint b\,\Phi_\gamma  (\mu'_*)^{1/2}g' \mu^{1/2}_* W_l  g ( W_l - W'_l )  dv_*d\sigma dv .\\
\end{align*}
Changing variables yields
$$
I = - \iiint b\,\Phi_\gamma (\mu_*)^{1/2}g \mu^{1/2 '}_* W'_l  g' ( W'_l - W_l ) dv_*d\sigma dv .
$$
Adding the above two equations gives
$$
2I = - \iiint b\,\Phi_\gamma  (\mu_*)^{1/2}g \mu^{1/2 '}_*   g' ( W'_l - W_l )^2  dv_*d\sigma dv .
$$
Then, by using the Cauchy-Schwarz inequality with respect to full variables, we find
$$
|I| \lesssim \iiint b\,\Phi_\gamma  (\mu_*)^{1/2}g^2 ( W'_l - W_l )^2 dv_*d\sigma dv.
$$
When $l\ge 1$, since
\begin{align*}
&\int_{\SS^2} b\, ( W'_l - W_l )^2 d\sigma  \lesssim  | v-v_*|^{2s} [ <v>+ <v_*>]^{2l-2s}\\
&\lesssim < v>^{2s} <v_*>^{2s} ( <v>^{2l-2s}+ <v_*>^{2l-2s})\\
&\lesssim \,\, <v_*>^{2s} <v>^{2l}+ <v>^{2s}<v_*>^{2l} \lesssim \,\,<v_*>^{2l} <v>^{2l}\, ,
\end{align*}
by using Section \ref{part1-section2}, we have
$$|I|\lesssim ||W_l g||^2_{L_{\gamma/2}^2}.
$$
When  $ 0< l \le 1$, by \eqref{part2-aa2}, we have
\begin{align*}
\int_{\SS^2} b \, | W_l'- W_l |^2 d\sigma
&\lesssim\, < v>^{2l(1-s)} <v_*>^{2l (1-s)} | v-v_*|^{2s}.
\end{align*}

Consider
$$
I = \iiint_{| v| \le R} +\iiint_{| v| \ge R}=I_1+I_2\, .
$$
It is obvious that for any fixed $R$,
$$
|I_1| \lesssim || W_l g||^2_{L^2_{\gamma /2}} .
$$
For $I_2$, we split
$$
I_2 = \iint_{| v| \ge R}\int_{\theta | v-v_*| \le {1\over 2} | v|} +\iint_{| v| \ge R}
\int_{\theta | v-v_*| \ge {1\over 2} | v|}=I_{2,1}+I_{2,2},
$$
which are the singular part  and the non-singular part respectively.
For the singular part $I_{2,1}$, note that
\begin{align*}
\int_{\theta | v-v_*| \le {1\over 2} | v|} b(\cos\theta) |W'_l-W_l|^2d\sigma &
\lesssim  | v-v_*|^2 <v>^{2(l-1)}{{| v|^{2-2s}}\over{| v-v_*|^{2-2s}}} \\
&\lesssim | v-v_*|^{2s} <v>^{2l -2s} \lesssim \,<v>^{2l} <v_*>^{2s};
\end{align*}
while for the non singular part $I_{2,2}$, one has
$$\int_{\theta | v-v_*| \ge {1\over 2} | v|} b(\cos\theta) |W'_l-W_l|^2d\sigma \lesssim | v-v_*|^{2s}  <v>^{2l}  <v_*>^{2l} | v|^{-2s} .
$$
Therefore, we have
$$
I_{2, 1} \lesssim || W_l g||^2_{L^2_{\gamma /2}},
$$
and
$$
I_{2,2} \lesssim R^{-2s} ||W_l g||^2_{L^2_{s+\gamma /2}}.
$$
By taking $R$ large enough,  we complete the proof of \eqref{part2-aa3}.

Now we turn to the derivatives in $v$ variable. For $\beta\in\NN^3\setminus\{0\}$, we have
$$
\partial_\beta \cL_1 (g)  =\cL_1 (\partial_\beta g ) + \sum_{\beta_1 < \beta } C_{\beta_1 , \beta_2}  \cT(\partial_{\beta_2} \mu ,\partial_{\beta_1} g, \partial_{\beta_3}\mu).
 $$
By \eqref{part2-aa3}, we have
\begin{align*}
(W_l\partial_\beta \cL_1 (g), W_l \partial_\beta g)  &= (W_l \cL_1 (\partial_\beta g ), W_l \partial_\beta g) + \sum_{\beta_1 <\beta } C_{\beta_1 , \beta_2} (W_l  \cT (\partial_{\beta_2} \mu ,\partial_{\beta_1} g, \partial_{\beta_3}\mu) , W_l \partial_\beta g)\\
&\ge \frac{\eta_0}{2} ||| W_l \partial_\beta g|||^2_{\Phi_\gamma} - C|| W_l \partial_\beta g||_{L^2_{\gamma /2}}  + II ,
\end{align*}
where
$$
II=  \sum_{\beta_1 <\beta} C_{\beta_1 , \beta_2} (W_l \cT(\partial_{\beta_2} \mu ,\partial_{\beta_1} g, \partial_{\beta_3}\mu) ; W_l \partial_\beta g) .
$$

Recall that the operator $\cT$ shares the same commutator properties  as $\Gamma$. As in the proofs given in Section \ref{part1-section2}, the linearized operator
$ \cT(\partial_{\beta_2} \mu ,\partial_{\beta_1} g, \partial_{\beta_3}\mu)$ satisfies
$$
| (W_l  \cT(\partial_{\beta_2} \mu ,\partial_{\beta_1} g, \partial_{\beta_3}\mu) ; W_l \partial_\beta g ) | \lesssim ||| W_l \partial_{\beta_1} g|||_{\Phi_\gamma} ||| W_l \partial_\beta g |||_{\Phi_\gamma} .
$$
Hence
$$
| II | \lesssim \Big( \sum_{\beta_1 <\beta} ||| W_l \partial_{\beta_1} g|||_{\Phi_\gamma} \Big)
 ||| W_l \partial_\beta g |||_{\Phi_\gamma} .
 $$
 This completes the proof of the proposition.
\end{proof}

\section{Local existence}
\smallskip
\setcounter{equation}{0}
In the following two subsections, we prove Theorem \ref{part2-theorem-anygamma} and the local existence of solutions in the function space considered in Theorem \ref{part2-theo1.1}.
\subsection{Classical solutions}
We now proceed to the proof of Theorem \ref{part2-theorem-anygamma}. The restriction of soft potential $\gamma +2s \le 0$ will play an important role.

Consider the following Cauchy problem for a linear Boltzmann equation with a given
function $f$,
\begin{equation}\label{part2-4-2-2.16}
\partial_t g + v\,\cdot\,\nabla_x g + \cL_1 g = \Gamma (f,\,g) -\cL_2 f
,\qquad g|_{t=0} = g_0\,,
\end{equation}
which is equivalent to the problem:
\begin{equation*}
\partial_t G + v\,\cdot\,\nabla_x G=
Q(F,\,G) ,\qquad G|_{t=0} = G_0,
\end{equation*}
with $F=\mu+\sqrt\mu\,f$ and $G=\mu+\sqrt\mu\,g$. The proof is based on energy estimates in the functional space
${\cH}^N_\ell(\RR^6)$.

For $N\geq 6, \ell\geq N$  and $\alpha, \beta\in\NN^3,
|\alpha|+|\beta|\leq N$, taking
$$
\varphi(t, x, v)= (-1)^{|\alpha|+|\beta|}\partial^\alpha_\beta
W^2_{\ell-|\beta|}\,\partial^\alpha_\beta\, g(t, x, v),
$$
as a test function for equation \eqref{part2-4-2-2.16}, we get
\begin{align*}
&\frac 12 \frac{d}{d t}\|W_{\ell-|\beta|}\partial^\alpha_\beta\,g\|^2_{L^2(\RR^6)}+
\Big(W_{\ell-|\beta|}\partial^\alpha_\beta \cL_1(g),\,
W_{\ell-|\beta|}\partial^\alpha_\beta g\Big)_{L^2(\RR^6)}
\\
&\qquad\qquad\qquad+\left( W_{\ell-|\beta|}\,[\partial^\alpha_\beta\,,\,\,v]\,
\cdot\,\nabla_x g,\,W_{\ell-|\beta|}\,\partial^\alpha_\beta g\right)_{L^2(\RR^6)}
\\
&=\Big(W_{\ell-|\beta|}\,\partial^\alpha_\beta \Gamma(f, \, g),\,
W_{\ell-|\beta|}\,\partial^\alpha_\beta g\Big)_{L^2(\RR^6)}-
\Big(W_{\ell-|\beta|}\,\partial^\alpha_\beta \cL_2(f),\,
W_{\ell-|\beta|}\,\partial^\alpha_\beta g\Big)_{L^2(\RR^6)},
\end{align*}
where we have used the fact that
$$
\left( v\,\cdot\,\nabla_x \big(W_{\ell-|\beta|}\,\partial^\alpha_\beta\,
g\big),\,W_{\ell-|\beta|}\,\partial^\alpha_\beta g\right)_{L^2(\RR^6)}=0\, .
$$
We have immediately
$$
\left|\Big( W_{\ell-|\beta|}\,[\partial^\alpha_\beta\,,\,\,v]\,\cdot\,
\nabla_x g,\,W_{\ell-|\beta|}\,\partial^\alpha_\beta g\Big)_{L^2(\RR^6)}
\right|\lesssim \|f\|_{\cH^{N}_\ell(\RR^6)}
\|g\|_{\cH^{N}_\ell(\RR^6)}.
$$
By Proposition \ref{part2-prop2.10}, one has
$$
\left|\Big(W_{\ell-|\beta|}\,\partial^\alpha_\beta \cL_2(f),\,
W_{\ell-|\beta|}\,\partial^\alpha_\beta g\Big)_{L^2(\RR^6)}
\right|\lesssim \|g\|^2_{\cH^{|\alpha|+|\beta|}_\ell(\RR^6)}.
$$
Now by using  \eqref{part2-4-2-5.4+1} with $f=\mu$, we have
$$
\Big|\big(W_{\ell-|\beta|}\partial^\alpha_\beta \cL_1(g)-
\cL_1\Big(W_{\ell-|\beta|}\,\partial^\alpha_\beta g\Big)
\, ,\,\,
W_{\ell-|\beta|}\partial^\alpha_\beta g\big)_{L^2(\RR^6)}\Big|\leq
C_\delta\|g\|^2_{\cH^N_\ell(\RR^6)}
+\delta\, \|g\|^2_{{\cB\,}^N_\ell(\RR^6)}\, .
$$
Applying  Proposition \ref{part2-prop5.5}, we get for any
$|\alpha|+|\beta|\leq N$,
\begin{align*}
&\frac 12 \frac{d}{d t}\|W_{\ell-|\beta|}\,\partial^\alpha_\beta\,g\|^2_{L^2(\RR^6)}+
\Big(\cL_1\Big(W_{\ell-|\beta|}\,\partial^\alpha_\beta g\Big),\,
W_{\ell-|\beta|}\,\partial^\alpha_\beta g\Big)_{L^2(\RR^6)}\\
&\lesssim  \|f\|_{\cH^{N}_\ell(\RR^6)} \,
||g||^2_{\cB^N_\ell(\RR^6)}+
\|f\|_{\cB^N_\ell(\RR^6)} \,\|g\|_{\cH^{N}_\ell(\RR^6)}
||g||_{\cB^N_\ell(\RR^6)}\\
&\qquad+ \|g\|^2_{\cH^{N}_\ell(\RR^6)}+
\|f\|_{\cH^{N}_\ell(\RR^6)} \,
||g||_{\cB^N_\ell(\RR^6)}+ \delta
||g||^2_{\cB^N_\ell(\RR^6)} .
\end{align*}

By using now the coercivity estimate from Section \ref{part1-section2} and Lemma \ref{part2-lemm2.6},
and by
taking summation over $|\beta|\leq N$, the Cauchy-Schwarz inequality and soft potential assumption imply that
\begin{align}\label{part2-4-2-5.6}
&\frac{d}{d t}\|g\|^2_{{\cH\,}^{N}_\ell(\RR^6)}+\frac{\eta_0}{2}
||g||^2_{{\cB\,}^N_\ell(\RR^6)}\leq C \Big\{
\|f\|_{{\cH\,}^{N}_\ell(\RR^6)}\,
||g||^2_{{\cB\,}^N_\ell(\RR^6)}\\
&\qquad+
\|g\|^2_{{\cH\,}^{N}_\ell(\RR^6)}\,
||f||^2_{{\cB\,}^N_\ell(\RR^6)}+ \|g\|^2_{{\cH\,}^{N}_\ell(\RR^6)}+
\|f\|^2_{{\cH\,}^{N}_\ell(\RR^6)} \Big\}\, .\nonumber
\end{align}

In conclusion, we are ready to prove the following proposition.

\begin{prop}\label{part2-theo5.6} Assume that $0 <s<1$, $\gamma>-3$ and let  $N\geq 6, \ell\geq N$.
 Suppose that  $g_0\in {\cH\,}^{N}_\ell(\RR^6)$ and
$$
f\in L^\infty([0, T];\,{\cH\,}^{N}_\ell(\RR^6)))\bigcap L^2([0, T];\, {\cB\,}^N_\ell(\RR^6)).$$
If $g\in L^\infty([0,
T];\,{\cH\,}^{N}_\ell(\RR^6))\bigcap L^2([0, T];\, {\cB\,}^N_\ell(\RR^6))$ is a
solution of the Cauchy problem \eqref{part2-4-2-2.16}, then there exists
$\epsilon_0>0$ such that if
\begin{equation}\label{part2-4-2-5.7}
\|f\|^2_{L^\infty([0, T];\,{\cH\,}^{N}_\ell(\RR^6))}+\|f\|^2_{L^2([0, T];\,{\cB\,}^N_\ell(\RR^6))}\leq \epsilon^2_0,
\end{equation}
we have
\begin{equation}\label{part2-4-2-5.8}
\|g\|^2_{L^\infty([0, T];\, {\cH\,}^{N}_\ell(\RR^6))}+ ||g||^2_{L^2([0,
T];\,{\cB\,}^N_\ell(\RR^6))}\leq Ce^{C\,
T}\Big(\|g_0\|^2_{{\cH\,}^{N}_\ell(\RR^6)}+\epsilon_0^2 T\Big),
\end{equation}
for a constant $C>0$ depending only on $N$ and $\ell$.
\end{prop}

\begin{proof}
{}From \eqref{part2-4-2-5.6}, we have, for $t\in ]0, T[$,
\begin{align*}
&\|g(t)\|^2_{{\cH\,}^{N}_\ell(\RR^6)}+\frac{\eta_0}{2} e^{C\, t}\int^t_0e^{-C\, s}
||g(s)||^2_{{\cB\,}^N_\ell(\RR^6)}ds\leq e^{C T}
\|g_0\|^2_{{\cH\,}^{N}_\ell(\RR^6)}\\
&\qquad + C e^{C\, T}\Big\{\int^T_0e^{-C\, s}
\|f(s)\|_{{\cH\,}^{N}_\ell(\RR^6)} \,
||g(s)||^2_{{\cB\,}^N_\ell(\RR^6)}ds\\
&\qquad +
\int^T_0e^{-C\, s}\|g(s)\|^2_{{\cH\,}^{N}_\ell(\RR^6)} \,
||f(s)||^2_{{\cB\,}^N_\ell(\RR^6)} ds+
\int^T_0e^{-C\, s}\|f(s)\|^2_{{\cH\,}^{N}_\ell(\RR^6)}ds \Big\}\, .
\end{align*}
Then
\begin{align*}
&\|g\|^2_{L^\infty([0, T]; {\cH\,}^{N}_\ell(\RR^6))}+\frac{\eta_0}{2}
||g||^2_{L^2([0, T]; {\cB\,}^N_\ell(\RR^6))}\leq e^{C\, T}\|g_0\|^2_{{\cH\,}^{N}_\ell(\RR^6)}\\
&\qquad + C e^{C\, T}\Big\{
\|f\|_{L^\infty([0, T]; {\cH\,}^{N}_\ell(\RR^6))} \,
||g||^2_{L^2([0, T]; {\cB\,}^N_\ell(\RR^6))}\\
&\qquad +\|g\|^2_{L^\infty([0, T]; {\cH\,}^{N}_\ell(\RR^6))} \,
||f||^2_{L^2([0, T]; {\cB\,}^N_\ell(\RR^6))} +
T\|f\|^2_{L^\infty([0, T]; {\cH\,}^{N}_\ell(\RR^6))} \Big\}\, .
\end{align*}
Hence, if we choose
$$
C e^{C\,  T}\epsilon_0\leq \frac{\eta_0}{4}, \,\,\,
C e^{C\, T}\epsilon^2_0\leq \frac{1}{2},
$$
then \eqref{part2-4-2-5.7} implies that
\begin{align*}
&\frac 1 2\|g\|^2_{L^\infty([0, T]; {\cH\,}^{N}_\ell(\RR^6))}+\frac{\eta_0}{4}
||g||^2_{L^2([0, T]; {\cB\,}^N_\ell(\RR^6))}\\
&\leq e^{C\, T}\|g_0\|^2_{{\cH\,}^{N}_\ell(\RR^6)}
+ C e^{C\, T}
T\epsilon^2_0 \, .
\end{align*}
And this completes the proof of the proposition.
\end{proof}

{}From the energy estimate \eqref{part2-4-2-5.8}, one can deduce the local existence as in \cite{amuxy3},
and we have proved the following precise version of Theorem \ref{part2-theorem-anygamma}.

\begin{theo}\label{part2-local-S}
Under the assumptions of Theorem \ref{part2-theorem-anygamma}, let $N\geq 6, \ell\geq N$. There exist $\epsilon_1,\, T>0$
 such that if $g_0\in \cH^{N}_\ell(\RR^6)$ and
$$
\|g_0\|_{\cH^{N}_\ell(\RR^6)}\leq \epsilon_1\, ,
$$
then the Cauchy problem \eqref{part2-4-2-1.5} admits a solution
$$
g \in L^\infty([0, T];\, \cH^{N}_\ell(\RR^6))\cap L^2([0, T];\,
\cB^N_\ell(\RR^6)).
$$
\end{theo}

\begin{rema}
By using the Proposition \ref{part2-upperGG}, we can get the same results as Theorem \ref{part2-local-S} if we replace
${\cH\,}^{N}_\ell(\RR^6), {\cB\,}^{N}_\ell(\RR^6)$ by ${\tilde{\cH}\,}^{N}_\ell(\RR^6), {\tilde{\cB}\,}^{N}_\ell(\RR^6)$ respectively.
In other words, Theorem \ref{part2-theorem-anygamma} holds also in the function space ${\tilde{\cH}\,}^{N}_\ell(\RR^6)$.
\end{rema}

\subsection{$L^2$-solutions}
Under some more restrictive conditions on the parameters $\gamma$ and $s$, we can prove local existence of
solutions with only differentiation in the $x$ variable. That is, we will
 deduce the energy estimate for the  equation \eqref{part2-4-2-2.16} in the function space
$H^N(\RR^3_x; L^2(\RR^3_v))$.
For $N\geq 3$  and $\beta\in\NN^3,
|\beta|\leq N$, by taking
$$
\varphi(t, x, v)= (-1)^{|\beta|}\partial^\beta_{x}
\partial^\beta_{x} g(t, x, v),
$$
as a test function on $\RR^3_x\times\RR^3_v$, we get
\begin{align*}
&\frac 12 \frac{d}{d t}\|\partial^\beta_x\,g\|^2_{L^2(\RR^6)}+
\Big(\partial^\beta_x \cL_1(g),\,
\partial^\beta_x g\Big)_{L^2(\RR^6)}\\
&=\Big(\partial^\beta_x \Gamma(f, \, g),\,
\partial^\beta_x g\Big)_{L^2(\RR^6)}-
\Big(\partial^\beta_x \cL_2(f),\,
\partial^\beta_x g\Big)_{L^2(\RR^6)},
\end{align*}
where we have used the fact that $
\left( v\,\cdot\,\nabla_x \Big(\partial^\beta_x\,
g\Big),\,\partial^\beta_x g\right)_{L^2(\RR^6)}=0\,$.

Applying now Proposition \ref{part2-prop2.8}, Proposition \ref{part2-prop2.10}, we get for any  $N\geq 3$ and
$|\beta|\leq N$,
\begin{align*}
&\frac 12 \frac{d}{d t}\|\partial^\beta\,g\|^2_{L^2(\RR^6)}+
\Big(\cL_1\Big(\partial^\beta_{x} g\Big),\,
\partial^\beta_{x}g\Big)_{L^2(\RR^6)}\\
&\lesssim \|f\|_{H^{N}(\RR^3_x; L^2(\RR^3_v))} \,
||g||^2_{\cX^N(\RR^6)}+
\|f\|_{\cX^N(\RR^6)} \,\|g\|_{H^{N}(\RR^3_x; L^2(\RR^3_v))}
||g||_{\cX^{N}(\RR^6)}\\
&\qquad+ \|g\|^2_{H^{N}(\RR^3_x; L^2(\RR^3_v))}+
\|f\|_{H^{N}(\RR^3_x; L^2(\RR^3_v))} \,
||g||_{\cX^N(\RR^6)}+ \|g\|_{H^{N}(\RR^3_x; L^2(\RR^3_v))}
||g||_{\cX^N(\RR^6)} .
\end{align*}
By using now the coercivity estimate \eqref{part2-4-2-2.9},
and by taking summation over $|\beta|\leq N$, the Cauchy-Schwarz
inequality leads to
\begin{align}\label{part2-4-2-2.17}
&\frac{d}{d t}\|g\|^2_{H^{N}(\RR^3_x; L^2(\RR^3_v))}+\frac{\eta_0}{2}
||g||^2_{\cX^N(\RR^6)}\leq C\Big\{
\|f\|_{H^{N}(\RR^3_x; L^2(\RR^3_v))} \,
||g||^2_{\cX^N(\RR^6)}\\
&\qquad+
\|g\|^2_{H^{N}(\RR^3_x; L^2(\RR^3_v))} \,
||f||^2_{\cX^N(\RR^6)}+ \|g\|^2_{H^{N}(\RR^3_x; L^2(\RR^3_v))}+
\|f\|^2_{H^{N}(\RR^3_x; L^2(\RR^3_v))} \Big\}\, .\nonumber
\end{align}

We are now ready to  prove the following

\begin{prop}
Under the assumptions  of Theorem \ref{part2-theo1.1}, let  $N\geq 3$, $g_0\in H^{N}(\RR^3_x; L^2(\RR^3_v))$ and
$$
f\in L^\infty([0, T];\,H^{N}(\RR^3_x; L^2(\RR^3_v)))\bigcap L^2([0, T];\, \cX^N(\RR^6)).$$
If $g\in L^\infty([0,
T];\,H^{N}(\RR^3_x; L^2(\RR^3_v)))\bigcap L^2([0, T];\, \cX^N(\RR^6))$ is a
solution of the Cauchy problem \eqref{part2-4-2-2.16}, then there exists
$\epsilon_0>0$ such that if
\begin{equation}
\|f\|^2_{L^\infty([0, T];\,H^{N}(\RR^3_x; L^2(\RR^3_v)))}+\|f\|^2_{L^2([0, T];\,\cX^N(\RR^6))}\leq \epsilon^2_0,
\end{equation}
we have
\begin{equation}\label{part2-4-2-2.19}
\|g\|^2_{L^\infty([0, T];\,H^{N}(\RR^3_x; L^2(\RR^3_v)))}+ ||g||^2_{L^2([0,
T];\,\cX^N(\RR^6))}\leq Ce^{C\,
T}\Big(\|g_0\|^2_{H^{N}(\RR^3_x; L^2(\RR^3_v))}+\epsilon_0^2 T\Big),
\end{equation}
for a constant $C>0$ depending only on $N$ .
\end{prop}

\begin{proof}
{}From \eqref{part2-4-2-2.17}, we have, for $t\in ]0, T[$,
\begin{align*}
&\|g(t)\|^2_{H^{N}(\RR^3_x; L^2(\RR^3_v))}+\frac{\eta_0}{2} e^{C t}\int^t_0e^{-C s}
||g(s)||^2_{\cX^N(\RR^6)}ds\leq e^{C T}
\|g_0\|^2_{H^{N}(\RR^3_x; L^2(\RR^3_v))}\\
&\qquad + Ce^{C T}\Big\{\int^T_0e^{-C s}
\|f(s)\|_{H^{N}(\RR^3_x; L^2(\RR^3_v))} \,
||g(s)||^2_{\cX^N(\RR^6)}ds\\
&\qquad +
\int^T_0e^{-C s}\|g(s)\|^2_{H^{N}(\RR^3_x; L^2(\RR^3_v))} \,
||f(s)||^2_{\cX^N(\RR^6)} ds+
\int^T_0e^{-C s}\|f(s)\|^2_{H^{N}(\RR^3_x; L^2(\RR^3_v))}ds \Big\}\, .
\end{align*}
Then
\begin{align*}
&\|g\|^2_{L^\infty([0, T]; H^{N}(\RR^3_x; L^2(\RR^3_v)))}+\frac{\eta_0}{2}
||g||^2_{L^2([0, T]; \cX^N(\RR^6))}\leq e^{C T}\|g_0\|^2_{H^{N}(\RR^3_x; L^2(\RR^3_v))}\\
&\qquad + Ce^{C T}\Big\{
\|f\|_{L^\infty([0, T]; H^{N}(\RR^3_x; L^2(\RR^3_v)))} \,
||g||^2_{L^2([0, T]; \cX^N(\RR^6))}\\
&\qquad +\|g\|^2_{L^\infty([0, T]; H^{N}(\RR^3_x; L^2(\RR^3_v)))} \,
||f||^2_{L^2([0, T]; \cX^N(\RR^6))} +
T\|f\|^2_{L^\infty([0, T]; H^{N}(\RR^3_x; L^2(\RR^3_v)))} \Big\}\, .
\end{align*}
By choosing $T$ small as in Proposition \ref{part2-theo5.6}, we complete the proof.
\end{proof}

As in \cite{amuxy3}, the energy estimate \eqref{part2-4-2-2.19} yields

\begin{theo}
Under the assumptions of Theorem \ref{part2-theo1.1}, for $N\geq 3$, there exist $\epsilon_1,\, T>0$
 such that if $g_0\in H^{N}(\RR^3_x; L^2(\RR^3_v))$ and
$$
\|g_0\|_{H^{N}(\RR^3_x; L^2(\RR^3_v))}\leq \epsilon_1\, ,
$$
then the Cauchy problem \eqref{part2-4-2-1.5} admits a solution
$$
g \in L^\infty([0, T];\, H^{N}(\RR^3_x; L^2(\RR^3_v)))\cap L^2([0, T];\,
\cX^N(\RR^6)).
$$
\end{theo}

\section{Global solutions}
\smallskip
\setcounter{equation}{0}

We are now ready to prove the global existence of weak and classical
solutions in the following two subsections.

\subsection{$L^2$-solutions }
We now conclude for the global existence issue in Theorem \ref{part2-theo1.1}.
We already gave  the
macro-micro decomposition of solutions introduced in \cite{guo-1}:
\begin{align*}
 &g=\pP g+(\iI-\pP)g=g_1+g_2,
 \\& g_1=(a+v\cdot b+ |v|^2 c)\sqrt {\mu}, \qquad \cA =(a,b,c)\, .
\end{align*}
Notice that
\begin{align*}
&||g||_{H^N(\RR^3_x; L^2(\RR^3_v))}^2\sim
\|\cA\|_{H^N(\RR^3)}^2+\|g_2\|_{H^N(\RR^3_x; L^2(\RR^3_v))}^2,
\\&
||g||^2_{\cX^{N}(\RR^6)}\sim \|\cA\|_{H^N(\RR^3)}^2+\|g_2\|_{\cX^N(\RR^6)}^2.
\end{align*}
The temporal energy functional and dissipation integral of solutions are
defined by
\begin{align*}
&\cE_N=\|g\|_{H^N(\RR^3_x; L^2(\RR^3_v))}^2=\|g_1\|_{H^N(\RR^3_x; L^2(\RR^3_v))}^2+
\|g_2\|_{H^N(\RR^3_x; L^2(\RR^3_v))}^2\\
&\qquad\sim
\|\cA\|_{H^N(\RR^3)}^2+\|g_2\|_{H^N(\RR^3_x; L^2(\RR^3_v))}^2,
\\&
\cD_N=
\|\nabla_x g_1\|^2_{H^{N-1}(\RR^3_x; L^2(\RR^3_v))}+\|g_2\|^2_{\cX^{N}(\RR^6)}
\sim \|\nabla_x\cA\|^2_{H^{N-1}(\RR^3)}+\|g_2\|^2_{\cX^{N}(\RR^6)},
\end{align*}
respectively.
Let  $g=g(t,x,v)$ be a solution to
\begin{equation}\label{part2-1.1}
 g_t+v\cdot\nabla_x g+\mathcal{L}g=\Gamma(g, g)\,,\,\,\,\,\,\,g|_{t=0}=g_0.
\end{equation}

We start with the  macroscopic energy estimate.
It is well-known that the macroscopic component
$
g_1=\pP g\sim \cA=(a,b,c),
$
satisfies the following set of equations
\begin{equation}\label{part2-macroeq}
\left\{
\begin{array}{rlrl}
v_i | v|^2 \mu^{1/2} :&\quad&\nabla_x c &= -\partial_t r_c+l_c + h_c,
\\
v^2_i \mu^{1/2}:&\quad&\partial_t c +\partial_ib_i &= -\partial_t r_i+l_i + h_i ,
\\
 v_iv_j \mu^{1/2}:& &\partial_ib_j + \partial_j b_i &
= -\partial_t r_{ij}+l_{ij} + h_{ij} , \quad i\neq j,
\\
 v_i \mu^{1/2} :&&\partial_t b_i + \partial_i a &
= -\partial_t r_{bi}+l_{bi} + h_{bi},
\\
 \mu^{1/2} :&& \partial_t a &= -\partial_t r_a+l_a + h_a,
\end{array}
\right.
\end{equation}
where
\begin{align*}
&r=(g_2,e)_{L^2(\RR^3_v)},\qquad l=-(v\cdot\nabla_x g_2+\mathcal{L} g_2,e)_{L^2(\RR^3_v)},
\qquad h=(\Gamma(g,g),e)_{L^2(\RR^3_v)},
\end{align*}
stand for $r_c, \cdots, h_a$, while
\begin{align*}
 e\in \text{span}
\lbrace v_i | v|^2 \mu^{1/2} , v^2_i \mu^{1/2} ,
v_iv_j \mu^{1/2}, v_i \mu^{1/2} , \mu^{1/2} \rbrace.
\end{align*}
\begin{lemm}\label{part2-abc2}  Assume $N\ge 3$ and let $\p^\alpha=\p^\alpha_x$, $\alpha\in\NN^3$, $|\alpha|\le N$. Then,
\begin{equation*}
\|\p^\alpha \cA^2\|_{L^2(\RR^3_x)}\le \|\nabla_x \cA
\|_{H^{N-1}(\RR^3_x)}\|\cA\|_{H^{N-1}(\RR^3_x)}.
\end{equation*}
\end{lemm}
\begin{proof}
Firstly, one has, for $|\alpha|=0$
\begin{align*}
&\|\cA^2\|_{L^2(\RR^3_x)}\le
\|\cA\|_{L^6(\RR^3_x)}\|\cA\|_{L^3(\RR^3_x)}
\le \|\nabla_x \cA\|_{L^2(\RR^3_x)}\|\cA\|_{H^1(\RR^3_x)}.
\end{align*}
Also for $|\alpha|=1$, we have
\begin{align*}
&\|\p \cA^2\|_{L^2(\RR^3_x)}\le \|\cA\p \cA\|_{L^2(\RR^3_x)}
\le \|\cA\|_{L^\infty(\RR^3_x)}\|\nabla_x \cA\|_{L^2(\RR^3_x)}
\\&\hspace*{3cm}
\le\|\cA\|_{H^{N-1}(\RR^3_x)}\|\nabla_x \cA\|_{L^2(\RR^3_x)},
\end{align*}
and for $2\le |\alpha|\le N$,
\begin{align*}
&\|\p^\alpha \cA^2\|_{L^2(\RR^3_x)}\le \sum_{k'\le \frac N2}
\|\p_x^{k'}\cA\p_x^{k-k'}\cA\|_{L^2(\RR^3_x)}
\\&\le \sum\|\p_x^{k'}\cA\|_{L^\infty(\RR^3_x)}
\|\p_x^{k-k'} \cA\|_{L^2(\RR^3_x)}
\le \|\cA\|_{H^{N-1}(\RR^3_x)}
\|\nabla_x \cA\|_{H^{N-1}(\RR^3_x)}.
\end{align*}
This completes the proof of the lemma.
\end{proof}

\begin{lemm} \label{part2-rlh} Assume $\gamma >-3$, $N\ge 3$. Let $\p^\alpha=\p^\alpha_x$, $\p_i=\p_{x_i}$, $|\alpha|\le N-1$. Then, one has
\begin{align}\label{part2-Dr}
&\|\p_i\p^\alpha r\,\|_{L^2(\RR^3_x)}+\|\p^\alpha   l\,\|_{L^2(\RR^3_x)}
\le\|g_2\|_{\cX^N(\RR^6)}\equiv A_1,
\\
 &\label{part2-Dh}
\|\p^\alpha   h\|_{L^2(\RR^3_x)}\le \|\nabla_x \cA\|_{H^{N-2}(\RR^3_x)}\|\cA\|_{H^{N-1}(\RR^3_x)}
\\&\quad+\|  \cA\|_{H^{N-1}(\RR^3_x)}\|g_2\|_{\cX^N(\RR^6)}+\|g_2\|_{H^N(\RR^3_x; L^2(\RR^3_v))}\|g_2\|_{\cX^N(\RR^6)}\notag\equiv A_2.
\end{align}
\end{lemm}
\noindent
\begin{proof} \eqref{part2-Dr} follows from
\begin{align*}
\|\p_i\p^\alpha r\,\|_{L^2(\RR^3_x)}
&=
\| (\p_i\p^\alpha g_2, e)_{L^2(\RR^3_v)}\|_{L^2(\RR^3_x)}
=\| (\tilde{W}^{-1}\p_i\p^\alpha g_2, \tilde{W} e)_{L^2(\RR^3_v)}\|_{L^2(\RR^3_x)}
\\&\le
\|\ \|\p_i\p^\alpha  g_2\|_{L^2_{s+\gamma/2}}
\|_{L^2(\RR^3_{x})}\le \|\nabla g_2\|_{\cX^{N-1}(\RR^6)},
\end{align*}
and
\begin{align*}
\|\p^\alpha   l\,\|_{L^2_{x}}&\le
 \| (\tilde{W}^{-1}\nabla_x\p^\alpha g_2, \tilde{W}\,v\,e)_{L^2(\RR^3_v)}\|_{L^2(\RR^3_x)}+
 \|(\tilde{W}^{-1}\p^\alpha g_2, \tilde{W}\mathcal{L}^*e)_{L^2(\RR^3_v)}\|_{L^2(\RR^3_x)}
 \\&
 \le \| \nabla_x\p^\alpha g_2 \|_{L^2_{s+\gamma/2}(\RR^6)}+
 \|\p^\alpha g_2\|_{L^2_{s+\gamma/2}(\RR^6)}
\le \|g_2\|_{\cX^N(\RR^6)}.
\end{align*}
Here, we have used \eqref{part2-sgamma}.

We prove \eqref{part2-Dh} as follows.
Firstly, set
\begin{align*}
H(f,g)&=(\Gamma(f,g),e)_{L^2(\RR^3_v)}
\\&
=\iiint b(\cos\theta)|v-v_*|^\gamma f_* g\Big((\mu^{1/2})_*'e'-\mu_*^{1/2}e\Big)dvdv_*d\sigma\\
&
=\iiint b(\cos\theta)|v-v_*|^\gamma(\mu^{1/2}f)_*(\mu^{1/2} g)\Big(q(v')-q(v)\Big)dvdv_*d\sigma,
\end{align*}
where $q(v)$ is a polynomial. Of course, $h=H(g,g)$.
We now write the
Taylor expansion of the second order of the function $q(v)$ as
\[
q(v)-q(v')=(\nabla q)(v)\cdot (v'-v)+\frac 12\int_0^1\nabla^2 q(v+\tau(v'-v')d\tau (v'-v)^2.
\]
Setting $ {\bf k} = \frac{v-v_*}{|v-v_*|}$, we
recall
\begin{align*}
v' -v = \frac{1}{2}|v-v_*|\Big( \sigma - (\sigma \cdot {\bf k})  {\bf k} \Big)
+ \frac{1}{2} ((\sigma \cdot {\bf k})-1)( v-v_*),
\end{align*}
and  notice that by virtue of  the symmetry
\[
\int_{\SS^2} b(\sigma
\cdot {\bf k})\Big( \sigma - (\sigma \cdot {\bf k})  {\bf k} \Big) d
\sigma =0.
\]
Therefore, we have
\begin{align*}
&H(f,g) = \frac{1}{2}\iint (\mu^{1/2}f)(\mu^{1/2} g)_* |v-v_*|^\gamma
\Big\{\int_{\SS}b(\cos\theta)
(\cos \theta -1)d\sigma \Big\}(\nabla q)(v)\cdot( v-v_*)
dv dv_*
\\
&\quad +
\frac{1}{2}\int_0^1\Big(\iiint (\mu^{1/2}f)(\mu^{1/2} g)_* |v-v_*|^\gamma
  b(\cos\theta) \nabla^2 q(v+\tau(v'-v')d\tau (v'-v)^2\
d\sigma dv dv_* \Big)
d \tau\\
&\quad = H_1(f,g)+H_2(f,g).
\end{align*}
For $H_1$, clearly, the integral in $\sigma$ is bounded for $0<s<1$.
By the Cauchy-Schwarz inequality, since $\gamma+1>-3$,
\begin{align*}
&|H_1(f,g)|\le \iint (|\nabla q(v)|\mu^{1/2}|f|)(\mu^{1/2}| g|)_* |v-v_*|^{\gamma+1}dvdv_*
\\&
\le\iint  \mu_*^{1/8}(\mu^{1/8}|f|) (\mu^{1/8} |g|)_* \mu^{1/8}|v-v_*|^{\gamma+1}dv_*dv
\\&
\le
\Big( \iint
 (\mu^{1/4} |g|^2)_* \mu^{1/4}|v-v_*|^{\gamma+1}dv_*dv \Big)^{1/2}
 \Big( \iint
 (\mu^{1/4} |f|^2) \mu_*^{1/4}|v-v_*|^{\gamma+1}dv_*dv \Big)^{1/2}
 \\&
\le \|\langle v\rangle^{(\gamma+1)/2}\mu^{1/8}f\|_{L^2_v}
\|\langle v\rangle^{(\gamma+1)/2}\mu^{1/8}g\|_{L^2_v}
\le \|f\|_{L^2_{l}(\RR^3_v)}\|g\|_{L^2_{m}(\RR^3_v)},
\end{align*}
for any $l, m\in\RR$.
Similarly, $\gamma+2>-3$ implies,
\begin{align*}
|H_2(f,g)|&\le
\iiint b(\cos\theta) \theta^2(\mu^{1/2}|f|)(\mu^{1/2} |g|)_* |v-v_*|^{\gamma +2}
  (|\nabla^2 q(v)|+|\nabla^2 q(v')|)d\sigma
dv dv_*
\\&
\le \|\langle v\rangle^{(\gamma+2)/2}\mu^{1/8}f\|_{L^1_v}\|\langle v\rangle^{(\gamma+2)/2}\mu^{1/8}g\|_{L^2_v}
\le \|f\|_{L^2_{l}(\RR^3_v)}\|g\|_{L^2_{m}(\RR^3_v)}.
\end{align*}
Combining these two estimates yields
\begin{align}\label{part2-Phiij}
|H(f,g)|\le \|f\|_{L^2_{l}(\RR^3_v)}\|g\|_{L^2_{m}(\RR^3_v)}.
\end{align}

Now $h$ is computed as follows.
\[
h=H(g,g)=\sum_{i,j=1,2}H(g_i,g_j)=\sum_{i,j=1,2}H^{(ij)}.
\]
Firstly, we have
\[
H^{(11)}\sim \cA^2H(\varphi_k,\varphi_m),
\]
where $\varphi_k\in \mathcal{N}$. Applying \eqref{part2-Phiij} for $l=m=0$, we get
by virtue of Lemma \ref{part2-abc2} and for $|\alpha|\le N-1$ that
$$
\|\p^\alpha  H^{(11)}\|_{L^2(\RR^3_x)}\le \|\p^\alpha  \cA^2\|_{L^2(\RR^3_x)}
 \le \|\nabla_x \cA\|_{H^{N-2}(\RR^3_x)}\|\cA\|_{H^{N-1}(\RR^3_x)};
$$
while taking  $l, m$ to be $0$ or $s+\gamma/2$ in \eqref{part2-Phiij}
and by the Leibniz rule and by \eqref{part2-sgamma},
\begin{align*}
  \|\p^\alpha   H^{(12)}\|_{L^2(\RR^3_x)}
&\le \sum_{\alpha_1+\alpha_2=\alpha}\|\p^{\alpha_1} \cA\|_{L^2(\RR^3_x)}
  \|\p^{\alpha_2}  g_2\|_{L^2_{s+\gamma/2}(\RR^6_{x, v})}
  \le \|  \cA\|_{H^{N-1}(\RR^3_x)}\|g_2\|_{\cX^N(\RR^6)},\\
\|\p^\alpha  H^{(21)}\|_{L^2(\RR^3_x)}&\le
\sum_{\alpha_1+\alpha_2=\alpha}\|\p^{\alpha_1}  g_2\|_{L^2_{s+\gamma/2}(\RR^6_{x, v})}
 \|\p^{\alpha_2}  \cA\|_{L^2(\RR^3_x)}
\le \|g_2\|_{\cX^N(\RR^6)}\|  \cA\|_{H^{N-1}(\RR^3_x)},\\
\|\p^\alpha   H^{(22)}\|_{L^2(\RR^3_x)}&\le
\sum_{\alpha_1+\alpha_2=\alpha}\|\p^{\alpha_1}  g_2\|_{L^2(\RR^6_{x, v})}^2
\|\p^{\alpha_2} g_2\|_{L^2_{s+\gamma/2}(\RR^6_{x, v})}\le \|g_2\|_{H^N(\RR^3_x; L^2(\RR^3_v))}\|g_2\|_{\cX^N(\RR^6)}.
 \end{align*}
Now the proof of Lemma \ref{part2-rlh} is completed.
\end{proof}

Next, we shall prove
\begin{lemm}\label{part2-pcA}  Assume $\gamma >-3$.
Let $|\alpha|\le N-1$. Then,
\begin{align}\label{part2-pabc}
\|\nabla_x \p^\alpha \cA\|_{L^2(\RR^3_x)}^2
\le
-\frac{d}{dt}\Big\{(\p^\alpha r,&\nabla_x \p^\alpha (a, - b, c))_{L^2(\RR^3_x)}
+(\p^\alpha b, \nabla_x \p^\alpha a)_{L^2(\RR^3_x)}\Big\}
\\&\quad +\notag\|g_2\|_{\cX^N(\RR^6)}^2+\|g_2\|_{H^N(\RR^3_x; L^2(\RR^3_v))}\cD_N.
\end{align}
\end{lemm}
\begin{proof}
(a) Estimate of $\nabla_x\p^\alpha  a$. Let $A_1,A_2$ be as in Lemma \ref{part2-rlh}.
{}From \eqref{part2-macroeq},
\begin{align*}
 \|\nabla_x \p^\alpha a\|_{L^2(\RR^3_x)}^2
&=(\nabla_x \p^\alpha a,\nabla_x \p^\alpha a)_{L^2(\RR^3_x)}
\\&= (\p^\alpha (-\p_t b-\p_tr+l+h),\nabla_x \p^\alpha a)_{L^2(\RR^3_x)}
\\&\le R_1+C_\eta (A_1^2+A_2^2)+\eta \|\nabla_x \p^\alpha a\|_{L^2(\RR^3_x)}^2.
\end{align*}
Here,
\begin{align*}
R_1&=-(\p^\alpha \p_tb+\p^\alpha \p_tr,\nabla_x \p^\alpha a)_{L^2(\RR^3_x)}
\\&=-\frac{d}{dt}(\p^\alpha (b+r),\nabla_x \p^\alpha a)_{L^2(\RR^3_x)}+
(\nabla_x\p^\alpha (b+r), \p_t \p^\alpha a)_{L^2(\RR^3_x)}
\\&
\le
-\frac{d}{dt}(\p^\alpha (b+r),\nabla_x \p^\alpha a)_{L^2(\RR^3_x)}+
C_\eta ( \|\nabla_x\p^\alpha  b\|_{L^2(\RR^3_x)}^2+A_1^2)
+\eta
\|\p_t\p^\alpha  a\|^2_{L^2(\RR^3_x)}.
\end{align*}
(b) Estimate of $\nabla_x\p^\alpha  b$. {}From \eqref{part2-macroeq} ,
\begin{align*}
\Delta_x&\p^\alpha b_i+\p^2_{i}\p^\alpha   b_i=
{\sum_{j \ne i} \p_j \p^{\alpha}( \p_j b_i + \p_i b_j)
+ \p_i \p^\alpha( 2 \p_i b_i - \sum_{j\ne i} \p_i b_j)
}
\\
&\hskip2cm =
\p_i\p^\alpha  (-\p_t r +l+h),
\\
&\|\nabla_x\p^\alpha  b\|_{L^2(\RR^3_x)}^2+\|\p_i\p^\alpha  b\|_{L^2(\RR^3_x)}^2=
-(\Delta_x\p^\alpha b_i+\p^2_{i}\p^\alpha   b_i, \p^\alpha  b)_{L^2(\RR^3_x)}
= R_2+R_3+R_4,
\end{align*}
where
\begin{align*}
R_2&=(\p_t\p^\alpha   r,\p_i\p^\alpha  b)_{L^2(\RR^3_x)}
=\frac{d}{dt}(\p^\alpha   r,\p_i\p^\alpha  b)_{L^2(\RR^3_x)}+
(\p_i\p^\alpha   r,\p_t\p^\alpha  b)_{L^2(\RR^3_x)}
\\&\hspace*{1cm}\le
\frac{d}{dt}(\p^\alpha   r,\p_i\p^\alpha  b)_{L^2(\RR^3_x)}+
C_\eta A_1^2
+\eta
\|\p_t\p^\alpha  b\|^2_{L^2(\RR^3_x)},
\\
R_3&=-(\p^\alpha   l,\p_i\p^\alpha  b)_{L^2(\RR^3_x)}\le
C_\eta A_1^2
+
\eta
\|\p_i\p^\alpha  b\|^2_{L^2(\RR^3_x)},
\\R_4&=-(\p^\alpha   h,\p_i\p^\alpha  b)_{L^2(\RR^3_x)}\le
C_\eta A_2^{{2}}
+
\eta
\|\p_i\p^\alpha  b\|^2_{L^2(\RR^3_x)}.
\end{align*}
(c) Estimate of $\nabla_x\p^\alpha  c$. {}From \eqref{part2-macroeq} ,
 \begin{align*}
 \|\nabla_x \p^\alpha &c\|_{L^2(\RR^3_x)}^2
=(\nabla_x \p^\alpha c,\nabla_x \p^\alpha c)_{L^2(\RR^3_x)}
= (\p^\alpha (-\p_tr+l+h),\nabla_x \p^\alpha c)_{L^2(\RR^3_x)}
\\&
\le R_5
+C_\eta(A_1^2+A_2^2)+\eta \|\nabla_x \p^\alpha c\|_{L^2(\RR^3_x)}^2,
\end{align*}
where
\begin{align*}
R_5&=-(\p^\alpha \p_tr,\nabla_x \p^\alpha c)_{L^2(\RR^3_x)}=
-\frac{d}{dt}(\p^\alpha r,\nabla_x \p^\alpha c)_{L^2(\RR^3_x)}+
(\nabla_x\p^\alpha r, \p_t \p^\alpha c)_{L^2(\RR^3_x)}
\\&\quad
\le
-\frac{d}{dt}(\p^\alpha r,\nabla_x \p^\alpha c)_{L^2(\RR^3_x)}+
C_\eta A_1^2
+\eta
\|\p_t\p^\alpha  c\|^2_{L^2(\RR^3_x)}.
\end{align*}
(d) Estimate of $\p_t\p^\alpha  (a,b,c)$.
\begin{align*}
\|\p_t\p^\alpha & \cA\|_{L^2(\RR^3_x)}=\|\p^\alpha \p_t \pP g\|_{L^2(\RR^6_{x, v})}
 \notag
\\&=\|\p^\alpha  \pP\Big(-v\cdot\nabla_xg-\mathcal{L}g+\Gamma(g,g)\Big)\|_{L^2(\RR^6_{x, v})}\notag
=\|\p^\alpha  \pP(v\cdot\nabla_xg)\|_{L^2(\RR^6_{x, v})}
\\&\le \|\nabla_x\p^\alpha  \cA\|_{L^2(\RR^3_x)}+
\|\nabla_x\p^\alpha \pP(vw^{-1}w g_2)\|_{L^2(\RR^6_{x, v})}
\\&
\le  \|\nabla_x\p^\alpha  (a,b,c)\|_{L^2(\RR^3_x)}+
\|\nabla_x\p^\alpha g_2\|_{L^2_{s+\gamma/2}(\RR^6_{x, v})}.
\notag
\end{align*}
By combining the above estimates and
taking $\eta>0$ sufficiently small, we deduce
\begin{align*}
\|\nabla_x \p^\alpha \cA\|_{L^2(\RR^3_x)}^2
&\le
-\frac{d}{dt}\Big\{(\p^\alpha r,\nabla_x \p^\alpha (a, -b,c))_{L^2(\RR^3_x)}
+(\p^\alpha b, \nabla_x \p^\alpha a)_{L^2(\RR^3_x)}\Big\}
\\&\quad +\notag
A_1^2+A_2^2+
 \eta\|\nabla_x\p^\alpha g_2\|_{L^2_{s+\gamma/2}(\RR^6_{x, v})}.
\end{align*}
Finally, by  choosing $|\alpha|\le N-1$, and using Lemma \ref{part2-rlh}, we obtain
\begin{align*}
A_1^2&+A_2^2+
 \eta\|\nabla_x\p^\alpha g_2\|_{L^2_{s+\gamma/2}(\RR^6_{x, v})}^2
\le \mathcal D_N\| g\|_{H^N(\RR^3_x; L^2(\RR^3_v))} +\eta \| g\|_{\cX^N(\RR^3)}^2,
\end{align*}
which completes the proof of Lemma \ref{part2-pcA}.
\end{proof}

We now turn to the estimation on the microscopic component $g_2$
 in the function space ${H^N(\RR^3_x; L^2(\RR^3_v))}$. Actually, we shall establish
\begin{lemm}\label{part2-microenergy}  Under the assumptions of Theorem \ref{part2-theo1.1}, for  $N\geq3$,
\begin{equation}\label{part2-micro}
\frac{d}{dt}\cE_N+\|g_2\|_{\cX^N(\RR^6)}^2\lesssim\cE_N^{1/2}\cD_N .
\end{equation}
\end{lemm}

\begin{proof}
We apply $\p_x^\alpha $ to \eqref{part2-1.1} and take the $L^2(\RR^6)$ inner product with
$\p_x^\alpha g$. Since the inner product including $v\cdot\nabla_xg$ vanishes by
integration by parts, we get
\begin{align}\label{part2-microinner}
\frac 12\frac{d}{dt}\|\p_x^\alpha g\|^2_N+(\mathcal{L}\p_x^\alpha g,
 \p_x^\alpha g)_{L^2(\RR^6)}=
(\p_x^\alpha \Gamma(g, g), \p_x^{\alpha}g)_{L^2(\RR^6)}.
\end{align}
In view of Section \ref{part1-section2} ,  we have,  for all $\gamma >-3$,
\[
\sum_{|\alpha|\le N}
(\mathcal{L}\p_x^\alpha g,
 \p_x^\alpha g)_{L^2(\RR^6)}\ge \eta_0
\sum_{|\alpha|\le N}\int_{\RR^3_x}||| \p_x^\alpha(\iI-\pP) g|||^2dx=\eta_0\| g_2\|_{\cX^N(\RR^6)}^2,
 \]
 while we shall show below that for $|\alpha|\le N,$ $N\geq 3$,
   \begin{align}\label{part2-nonlinear5}
 \Big|(\p_x^\alpha \Gamma(g, g), \p_x^{\alpha}g)_{L^2(\RR^6)}
\Big|
\lesssim  \cE_N^{1/2}\cD_N.
\end{align}
Lemma \ref{part2-microenergy} can then be
 concluded by plugging  these two estimates into \eqref{part2-microinner}.

\noindent
{\bf Proof of \eqref{part2-nonlinear5}:}
Write
 \begin{align*}
(\p_x^{\alpha}\Gamma(g, &g), \p_x^{\alpha}g)_{L^2(\RR^6)}
 =J^{11}+J^{12}+J^{21}+J^{22},
  \\&
J^{ij}=(\p_x^{\alpha}\Gamma(g_i, g_j), \p_x^{\alpha} g_2)_{L^2(\RR^6)}.
 \end{align*}


\noindent
{\bf Estimation of $J^{11}$.}
We shall estimate
\begin{align*}
&J^{11} \sim
\int_{\RR^3_x}(\p_x^{\alpha_1}\cA)(\p_x^{\alpha_2}\cA)
(\Gamma(\varphi_k,\varphi_m), \p_x^{\alpha} g_2)_{L^2(\RR^3_v)}dx.
\end{align*}
Firstly, for any function $\phi, \psi\in\cN$, set
\begin{align*}
\Psi_1&=(\Gamma(\phi,\psi), h)_{L^2(\RR^3)}.
\end{align*}
We shall prove that for
$\gamma>-3$,
\begin{align}\label{part2-Psi1}
|\Psi_1|\le \|h\|_{L^2_m(\RR^3)},
\end{align}
holds for any $m\in\RR$.

\noindent {\bf Proof of \eqref{part2-Psi1}.}  Notice that
\begin{align*}
\Psi_1&=\iiint b(\cos\theta)|v-v_*|^\gamma (\mu_*)^{1/2}\Big((\phi_*)'\psi'-
\phi_*\psi\Big)h\ dvdv_*d\sigma
\\&=
\iiint b(\cos\theta)|v-v_*|^\gamma (\mu_*)^{1/2}(\mu_*)^{1/2}(\mu)^{1/2}
\Big(q_*'r'-q_*r\Big)\ hdvdv_*d\sigma,
\end{align*}
where $q=q(v)$ and $r=r(v)$ are some polynomials.
First, write
\[
q_*'r'-q_*r=(q_*'-q_*)(r'-r)+(q_*'-q_*)r+q_*(r'-r)=S_1+S_2+S_3,
\]
and make a decomposition
\begin{align*}
\Psi_1&=\sum_{i=1}^3
\iiint b(\cos\theta)|v-v_*|^\gamma(\mu_*)^{1/2}(\mu_*)^{1/2}(\mu)^{1/2}\ S_i\ hdvdv_*d\sigma
\\&=\Psi_1^{1}+\Psi_1^{2}+\Psi_1^{3}.
\end{align*}
Since
$
|S_1|\le R_1(v,v_*)|v-v_*|^2\theta^2 $
where $R_1(v,v_*)$ is a polynomial of $v,v_*$, it holds that
\begin{align*}
|\Psi_1^{1}|&\lesssim \iint |v-v_*|^{\gamma+2}\Big[\int_{\SS^2} b(\cos\theta)\theta^2d\sigma\Big]
(\mu_*)^{1/2}(\mu_*)^{1/2}(\mu)^{1/2}|R_1(v,v_*)|\ |h|dvdv_*
\\&
\lesssim \int\la v\ra^{\gamma+2}\mu^{1/4}|h|dv\lesssim\|h\|_{L^2_m(\RR^3)},
\end{align*}
for any $m\in \RR$.

On the other hand,  the Taylor expansion of the second order gives
\[
q(v_*')-q(v_*)= (\nabla q)(v_*)\cdot (v_*'-v_*)+
\frac 12\int_0^1\nabla^2q(v_*+\tau(v_*'-v_*))d\tau (v_*'-v_*)^2\,.
\]
Since
\begin{align*}
\Big|\int_0^1\nabla^2q(v_*+\tau(v_*'-v_*))d\tau (v_*'-v_*)^2\Big|
\lesssim |R_2(v,v_*)||v-v_*|^2\theta^2,
\end{align*}
where $R_2$ is a polynomial of $v,v_*$, by symmetry, we have
\begin{align*}
|\Psi_1^2|&\lesssim \iint |v-v_*|^{\gamma+1}
(\mu_*)^{1/2}(\mu_*)^{1/2}(\mu)^{1/2}\Big(|\nabla q(v_*)|+
|R_2(v,v_*)||v-v_*|\Big)\ |r(v)h|dvdv_*
\\&
\lesssim \int(\la v\ra^{\gamma+1}+\la v\ra^{\gamma+2})\mu^{1/4}|r(v)h|dv\lesssim \|h\|_{L^2_m(\RR^3)},
\end{align*}
for any $m\in\RR$.

The estimation on  $\Psi_1^3$ can be carried out exactly in the same way to have
 \begin{align*}
 |\Psi_1^3|
\lesssim \int(\la v\ra^{\gamma+1}+\la v\ra^{\gamma+2})\mu^{1/4}
|h|dv\lesssim\|h\|_{L^2_m(\RR^3)},
\end{align*}
for any $m\in\RR$. This completes the proof of \eqref{part2-Psi1}.

Take $m=s+\gamma/2$ in \eqref{part2-Psi1} and use \eqref{part2-sgamma}
to obtain
$
|\Psi_1|\lesssim ||| h|||
$. Set $h=\p^\alpha g_2$.
Now by the Sobolev embedding theorem, for $\alpha_1+\alpha_2=\alpha, 1\leq|\alpha|\leq N$, we have
\begin{align*}
|J^{11}|&
\lesssim
\int_{\RR^3_x}|\p_x^{\alpha_1}\cA|\ |\p_x^{\alpha_2}\cA|\
\ |||\p_x^{\alpha} g_2|||_{\Phi_\gamma}  dx
\\&
\lesssim \left\{
\begin{array}{ll}
\|\p_x^{\alpha_1}\cA\|_{L^\infty_x}\|\p_x^{\alpha_2}\cA\|_{L^2_x}
||\p_x^{\alpha} g_2||_{\cX^0(\RR^6)}
\lesssim \|\cA\|_{H^N_x}\|\nabla_x\cA\|_{H^{N-1}_x}\|g_2\|_{\cX^N(\RR^6)},
&{ |\alpha_1|\le N-2,}
\\[0.3cm]
\|\p_x^{\alpha_1}\cA\|_{L^2_x}\|\p_x^{\alpha_2}\cA\|_{L^\infty_x}
||\p_x^{\alpha} g_2||_{\cX^0(\RR^6)}\lesssim \|\nabla_x\cA\|_{H^{N-1}_x}\|\cA\|_{H^N_x}
\|g_2\|_{\cX^N(\RR^6)}, &{ |\alpha_2|< 1,}
\end{array}
\right.
\\&
\lesssim \|\cA\|_{H^N_x}(\|\nabla_x\cA\|_{H^{N-1}_x}^2+\|g_2\|_{\cX^N(\RR^6)}
^2)=\cE_N^{1/2}\cD_N,
\end{align*}
and for $|\alpha_1+\alpha_2|=0$, we have
\begin{align*}
\int_{\RR^3_x}|\cA|^2
\ ||| g_2|||  dx
&\lesssim
\|\cA\|_{L^6_x}\|\cA\|_{L^3_x}\|g_2\|_{\cX^N(\RR^6)}
\\
&\lesssim \|\nabla_x \cA\|_{L^2_x}
\|\cA\|_{H^1_x}\|g_2\|_{\cX^N(\RR^6)}\lesssim\cE_N^{1/2}\cD_N.
\end{align*}
\noindent
{\bf Estimation of $J^{12}$:}
First, notice
\begin{align*}
 J^{12}&\sim \int_{\RR^3_x}(\p_x^{\alpha_1}\cA)
(\Gamma(\varphi_k, \p_x^{\alpha_2}g_2), \p_x^{\alpha} g_2)_{L^2(\RR^3_v)}dx.
\end{align*}
For some $\phi\in\cN$, set
\begin{align*}
\Psi_2&=(\Gamma(\phi, g), h)_{L^2(\RR^3)}.
\end{align*}

By virtue of Theorem \ref{theorem-1.1-b}, we get
\begin{align*}
|\Psi_2|&\lesssim
\Big(\|\phi\|_{L^2_{s+\gamma/2}(\RR^3)}\ |||g|||_{\Phi_\gamma}+||| \phi|||_{\Phi_\gamma}\
\|g\|_{L^2_{s+\gamma/2}(\RR^3)}
\\&
\hspace{1cm}+
\min\{\|\phi\|_{L^2(\RR^3)}\|g\|_{L^2_{s+\gamma/2}(\RR^3)},\
\|g\|_{L^2}\|\phi\|_{L^2_{s+\gamma/2}(\RR^3)}\}
\Big)|||h|||_{\Phi_\gamma}
\\&\lesssim
\Big(\|\phi\|_{L^2_{s+\gamma/2}(\RR^3)}\ |||g|||_{\Phi_\gamma}+||| \phi|||_{\Phi_\gamma}\
\|g\|_{L^2_{s+\gamma/2}(\RR^3)}+
\|\phi\|_{L^2(\RR^3)}\|g\|_{L^2_{s+\gamma/2}(\RR^3)}\Big)|||h|||_{\Phi_\gamma}
\\&\lesssim
\Big( |||g|||_{\Phi_\gamma}+\|g\|_{L^2_{s+\gamma/2}(\RR^3)}\Big)\ |||h|||_{\Phi_\gamma}\lesssim |||g|||_{\Phi_\gamma}\ |||h|||_{\Phi_\gamma},
\end{align*}
where we have chosen the first factor in the $\min$ term.

Now we have
  \begin{align*}
| J^{12}|&
 \lesssim \int_{\RR^3_x}\|\p_x^{\alpha_1}\cA|\
|||\p_x^{\alpha_2} g_2|||_{\Phi_\gamma}
 \ |||\p_x^{\alpha} g_2|||_{\Phi_\gamma}dx
\\&\lesssim
\left\{
\begin{array}{ll}
\|\p_x^{\alpha_1}\cA\|_{L^\infty_x}\int_{\RR^3}
|||\p_x^{\alpha_2} g_2|||_{\Phi_\gamma}\
|||\p_x^{\alpha} g_2|||_{\Phi_\gamma}dx\lesssim \|\cA\|_{H^N(\RR^3)}\|g_2\|_{\cX^N(\RR^6)}^2,  & {|\alpha_1|\le N-2,}
\\[0.3cm]
\| \ |||\p_x^{\alpha_2} g_2|||_{\Phi_\gamma}\ \|_{L^\infty_x}
\int_{\RR^3}|\p_x^{\alpha_1}\cA|
\
|||\p_x^{\alpha} g_2|||_{\Phi_\gamma}dx \lesssim \|\cA\|_{H^N(\RR^3)}\|g_2\|_{\cX^N(\RR^6)}^2,  &  { |\alpha_2|\le 1.}
\end{array}
\right.
\end{align*}

\noindent{\bf Estimation of $J^{21}$:}
A similar argument  applies to
\begin{align*}
 J^{21}&\sim\int_{\RR^3_x}(\p_x^{\alpha_2}\cA)
(\Gamma(\p_x^{\alpha_1}g_2, \varphi_k), \p_x^{\alpha} g_2)_{L^2(\RR^3_v)}dx.
\end{align*}
 In fact, for $\phi\in\cN$, set
\[
\Psi_3=(\Gamma(f,\phi),h) .
\]
Again by virtue of Theorem \ref{theorem-1.1-b} and taking the second factor in the $\min$ term,
  \eqref{part2-sgamma} gives
\begin{align*}
|\Psi_3|\lesssim
 |||f|||_{\Phi_\gamma}\ |||h|||_{\Phi_\gamma}.
\end{align*}

Thus, proceeding as for $J^{12}$ yields
\[
|J^{21}| \lesssim \|\cA\|_{H^N(\RR^3)}\|g_2\|_{\cX^N(\RR^6)}^2.
\]

\noindent{\bf Estimation on $J^{22}$:}
It follows from the Leibniz rule that
\begin{align*}
 |J^{22}|\lesssim
\int |(\Gamma(\p_x^{\alpha_1}g_2, \p_x^{\alpha_2}g_2), \p_x^{\alpha} g_2)_{L^2(\RR^3)}|dx.
\end{align*}
Different from the above, we now use Theorem \ref{theorem-1.1-b} and
\eqref{part2-sgamma} in the following way.
\begin{align*}
|(\Gamma(f,& g), h)_{L^2(\RR^3)}|
\lesssim
\Big(\|f\|_{L^2_{s+\gamma/2}(\RR^3)}\ |||g|||_{\Phi_\gamma}+||| f|||_{\Phi_\gamma}\
\|g\|_{L^2_{s+\gamma/2}(\RR^3)}
\\&\notag
\hspace{1cm}+
\min\Big\{\|f\|_{L^2(\RR^3)}\|g\|_{L^2_{s+\gamma/2}(\RR^3)},\
\|g\|_{L^2}\|f\|_{L^2_{s+\gamma/2}(\RR^3)}\Big\}
\Big)|||h|||_{\Phi_\gamma}
\\&\lesssim\notag
\Big( \|f\|_{L^2(\RR^3)}||| g|||_{\Phi_\gamma}+|||f|||_{\Phi_\gamma}\
\|g\|_{L^2(\RR^3)}\Big)|||h|||_{\Phi_\gamma}.
\end{align*}
This is valid for the assumptions imposed
 on $\gamma$ and $s$ from Theorem \ref{part2-theo1.1}. Then,
\begin{align*}
 |J^{22}|&\lesssim \int
\Big( \|\p^{\alpha_1}g_2\|_{L^2(\RR^3)}||| \p^{\alpha_2}g_2|||_{\Phi_\gamma}+|||\p^{\alpha_1}g_2|||_{\Phi_\gamma}\
\|\p^{\alpha_2}g_2\|_{L^2(\RR^3)}\Big)|||\p^{\alpha} g_2|||_{\Phi_\gamma}dx.
\end{align*}

Suppose $|\alpha_1|\le N-2$. Then, by the Sobolev embedding theorem,
\begin{align*}
|J^{22}|&
\lesssim
\|\p_x^{\alpha_1}g_2\|_{L^\infty(\RR^3_x; L^2(\RR^3_v))}\int_{\RR^3}
|||\p_x^{\alpha_2}g_2|||_{\Phi_\gamma}\
||| \p_x^{\alpha} g_2|||_{\Phi_\gamma}dx
\\&\qquad
+\|\ |||\p_x^{\alpha_1}g_2|||_{\Phi_\gamma} \ \|_{L^\infty(\RR^3_x)}
\int_{\RR^3}||\p_x^{\alpha_2}g_2||_{L^2(\RR^3_v)}\
||| \p_x^{\alpha} g_2|||_{\Phi_\gamma}dx
\\&\lesssim
\|g_2\|_{H^N(\RR^3_x; L^2(\RR^3_v))}\|g_2\|_{\cX^N(\RR^3)}^2\lesssim \cE_N^{1/2}\cD_N.
\end{align*}
Similarly, when $|\alpha_1|> N-2$ then $|\alpha_2|\le 1$, we get
\begin{align*}
|J^{22}|&
\lesssim
\|\p_x^{\alpha_2}g_2\|_{L^\infty_x(L^2_v)}\int_{\RR^3}|||\p_x^{\alpha_1}g_2|||_{\Phi_\gamma}\
||| \p_x^{\alpha} g_2|||_{\Phi_\gamma}dx
\\&\qquad +\|\ |||\p_x^{\alpha_2}g_2|||_{\Phi_\gamma} \ \|_{L^\infty_x}
\int_{\RR^3}||\p_x^{\alpha_1}g_2||_{L^2_v}\
||| \p_x^{\alpha} g_2|||_{\Phi_\gamma}dx
\\&\lesssim\cE_N^{1/2}\cD_N.
\end{align*}

Now, combining the above estimates yields the estimate \eqref{part2-nonlinear5} and this completes the proof of the
Lemma \ref{part2-microenergy}.
\end{proof}

Taking a suitable linear combination of  \eqref{part2-pabc} and \eqref{part2-micro} gives the
\begin{prop} {\rm \bf (Global Energy Estimate without Weight) \ }
Under the assumptions of Theorem \ref{part2-theo1.1}, for $N\geq3$, there exists a constant $C>0$ such that
\[
\frac{d}{dt}\cE_N+\cD_N\le C\cE^{1/2}_N\cD_N
\]
holds as far as $g$ exists.
\end{prop}
This proposition assures that a usual continuation argument of local solutions can be carried out
under the smallness assumption of
initial data.
Thus, we established the existence of global solutions in the space
$H^N(\RR^3_x; L^2(\RR^3_v))$.

\subsection{Classical solutions}
\smallskip

We now turn to the energy estimates involving also $v$ derivatives of solutions.
To close this type of energy estimate,
we then need to use the weighted norms in the $v$ variable, cf. also Guo \cite{guo},
with the weight function $ \tilde W$.
Recall that we  assume $s+\gamma/2\le 0$.
Set
\begin{align*}
\cE_{N,\ell}&= \cE_N+\|g\|_{\widetilde{\cH}^N_\ell(\RR^6)}^2
\sim \|g_1\|_{\cH^N(\RR^6)}^2+\|g\|_{\widetilde{\cH}^N_\ell(\RR^6)}^2
\sim\|\cA\|_{H^{N}(\RR^3)}^2
+\|g_2\|^2_{{\tilde \cH}^N_\ell(\RR^6)},
\\ \cD_{N,\ell}&= \cD_N+\|g_2\|_{\widetilde{\cB}^N_\ell(\RR^6)}^2\sim\|\nabla g_1\|_{\cH^{N-1}(\RR^3)}+\|g_2\|_{\widetilde{\cB}^N_\ell(\RR^6)}^2\sim
\|\nabla_x\cA\|_{H^{N-1}(\RR^3)}+\|g_2\|_{\widetilde{\cB}^N_\ell(\RR^6)}^2.
\end{align*}
Recall
\[
\p^\alpha_{\beta}=\p_x^\alpha\p^\beta_v,\quad |\alpha|+|\beta|\le N, \quad \beta \ne 0,  \quad N\ge {6 },
\]
and apply $\tilde W_{\ell-|\beta|}\p^\alpha_{\beta}$ to \eqref{part2-1.1} to get
\begin{align*}
 \p_t (\tilde W_{\ell-|\beta|}\p^\alpha_{\beta} &g_2)+ v\cdot\nabla_x
(\tilde W_{\ell-|\beta|}\p^\alpha_{\beta} g_2) + \cL _1(\tilde W_{\ell-|\beta|}\p^\alpha_{\beta} g_2)
\\=&\notag
\tilde W_{\ell-|\beta|}\p^\alpha_{\beta} \Gamma(g,\,g)+
[v\cdot\nabla_x, \tilde W_{\ell-|\beta|}\p^\alpha_{\beta}]g_2
- \tilde W_{\ell-|\beta|} \p^\alpha_{\beta} [\pP,\,\, v\cdot \nabla]g \notag\\
&+\notag
[\mathcal{L}_1,\, \tilde W_{\ell-|\beta|}\p^\alpha_{\beta}]g_2
-\tilde W_{\ell-|\beta|}\p^\alpha_{\beta}\cL_2(g_2).
\end{align*}
Then take the $L^2(\RR^6)$ inner product with $\tilde W_{\ell-|\beta|}\p^\alpha_{\beta} g_2$ to
 get
\begin{equation}\label{part2-energy1+}
\frac 12 \frac{d}{dt}\|\tilde W_{\ell-|\beta|}\p^\alpha_{\beta} g_2\|^2_{L^2(\RR^6)}+D\le K.
\end{equation}
Here, $D$ is a dissipation rate given by
\begin{align*}
 D&=(\cL _1(\tilde W_{\ell-|\beta|}\p^\alpha_{\beta} g_2),
\tilde W_{\ell-|\beta|}\p^\alpha_{\beta} g_2)_{L^2(\RR^6)}.
 \end{align*}
 Due to the coercivity inequality from Section \ref{part1-section2},
which holds true for $\gamma >-3$, we get
 \begin{align*}
D
&
\ge\eta_0\int_{\RR^3}||| (\iI-\pP)\tilde W_{\ell-|\beta|}\p^\alpha_{\beta} g_2 |||^2dx
\\&\ge \eta_0\| \
||| \tilde W_{\ell-|\beta|}\p^\alpha_{\beta} g_2 ||| \ \|_{L^2(\RR^3_x)}^2-C
\ || \p^\alpha g_2 ||^2_{ L^2_{s+\gamma/2}(\RR^6)},
\end{align*}
 where we have used, with $\psi\in \mathcal{N}$ and  $m\in\NN$,
\begin{align*}
|||\pP \tilde W_{\ell-|\beta|}\p^\alpha_{\beta} g_2|||^2 &=|||\p^\alpha_x
(\psi, \tilde W_{\ell-|\beta|}\p^\beta_v g_2)_{L^2(\RR^3_v)}\psi|||^2
\\&
=|(\tilde{\psi},\tilde W^{-m}\p^\alpha_x g_2)_{L^2_{v}}|\ |||\psi|||^2
\lesssim ||\p^\alpha_x g_2||_{L^2_{-m}(\RR^6_{x, v})}^2.
\end{align*}
Note that we will use the above estimate later by choosing $m=-|s+\gamma/2|$.
On the other hand, $K$ is given by
\begin{align*}
K=& (\tilde W_{\ell-|\beta|}\p^\alpha_{\beta} \Gamma(g,\,g), \tilde W_{\ell-|\beta|}\p^\alpha_{\beta} g_2)_{L^2(\RR^6)}
+
([v\cdot\nabla_x,\, \tilde W_{\ell-|\beta|}\p^\alpha_{\beta}]g_2, \tilde W_{\ell-|\beta|}\p^\alpha_{\beta} g_2)_{L^2(\RR^6)}
\\&
- (\tilde W_{\ell-|\beta|} \p^\alpha_{\beta} [\pP,\, v\cdot \nabla]g, \tilde W_{\ell-|\beta|}\p^\alpha_{\beta} g_2)_{L^2(\RR^6)}
+
([\mathcal{L}_1,\, \tilde W_{\ell-|\beta|}\p^\alpha_{\beta}]g_2, \tilde W_{\ell-|\beta|}\p^\alpha_{\beta} g_2)_{L^2{(\RR^6)}}
\\&-(\tilde W_{\ell-|\beta|}\p^\alpha_{\beta}\cL_2(g_2),\tilde W_{\ell-|\beta|}\p^\alpha_{\beta} g_2)_{L^2(\RR^6)}
\\=&
K_1+K_2+K_3+K_4+ K_5.
\end{align*}

\noindent{\bf (I) Estimation of $K_1$:}
First, we show that
\begin{equation}\label{part2-K1}
|K_1|\lesssim \cE_{N,\ell}^{1/2}\cD_{N,\ell}.
\end{equation}
For the proof, write
 \begin{align*}
 K_1&=\sum_{i,j=1}^2(\tilde W_{\ell-|\beta|}\p^\alpha_{\beta} \Gamma(g_i,\,g_j),
\tilde W_{\ell-|\beta|}\p^\alpha_{\beta} g_2)_{L^2(\RR^6)}
  \\&=K_{111}+K_{112}+K_{121}+K_{122}.
 \end{align*}

\noindent
{\bf (1) Estimation on $K_{111}$:}
Proceeding as in the computation for $\Psi_1$ in \eqref{part2-Psi1}, we get for $\gamma >-3$,
\[
|(\tilde W_{\ell-|\beta|}\p_{\beta}\Gamma(\varphi_k,\varphi_m), \tilde W_{\ell-|\beta|}\p_{\beta} g_2)_{L^2(\RR^3_v)}|\lesssim \|\tilde W_{\ell-|\beta|}\p_{\beta} g_2\|_{L^2(\RR^3_v)},
\]
which leads to
 \begin{align*}
K_{111} \sim&
\sum_{\alpha_1+\alpha_2=\alpha}\int_{\RR^3_x}|(\p_x^{\alpha_1}\cA)(\p_x^{\alpha_2}\cA)\|\tilde W_{\ell-|\beta|}\p^\alpha_{\beta} g_2\|_{L^2(\RR^3_v)}dx
\\&\lesssim\sum_{\alpha_1+\alpha_2=\alpha}
\|(\p_x^{\alpha_1}\cA)(\p_x^{\alpha_2}\cA)\|_{L^2(\RR^3_x)}\|g_2\|_{\tilde{\cB}_\ell^N(\RR^6)}
\\&\lesssim \|\cA\|_{H^N_x}(\|\nabla_x\cA\|_{H^{N-1}_x}^2+\|g_2\|_{\tilde{\cB}_\ell^N(\RR^6)}^2)
\lesssim
\cE_{N,\ell}^{1/2}\cD_{N,\ell}.
\end{align*}
Here, we used that for $\alpha_1=\alpha_2=0$,
\[
\|(\cA)^2\|_{L^2(\RR^3_x)}\lesssim\|\cA\|_{L^3(\RR^3_x)}\|\cA\|_{L^6(\RR^3_x)}
\lesssim \|\cA\|_{H^1(\RR^3_x)}\|\nabla \cA\|_{L^2(\RR^3_x)}.
\]
{\bf (2) Estimation on $K_{112}$:}
Since $g_1\sim\cA \varphi$, Proposition \ref{part2-upperGG} implies
\begin{align*}
 |K_{112}|\lesssim &
\sum_{\begin{subarray}{l}
\alpha_1+\alpha_2=\alpha
\\
\beta_1+\beta_2\le\beta
\end{subarray}}
\int_{\RR^3}
 |||\tilde W_{\ell-|\beta|} \p^{\alpha_1}_{\beta_1}g_1|||_{\Phi_\gamma}\
||| \tilde W_{\ell-|\beta|}  \p^{\alpha_2}_{\beta_2}g_2|||_{\Phi_\gamma} \
  ||| \tilde W_{\ell-|\beta|} \p^\alpha_{\beta}g_2|||_{\Phi_\gamma} dx
 \\
\lesssim& \sum_{
\begin{subarray}{l}
\alpha_1+\alpha_2=\alpha
\\
\beta_1+\beta_2\le\beta
\end{subarray}
}
\int_{\RR^3}
 |\p^{\alpha_1}\cA|\
||| \tilde W_{\ell-|\beta|}  \p^{\alpha_2}_{\beta_2}g_2|||_{\Phi_\gamma} \
 ||| \tilde W_{\ell-|\beta|} \p^{\alpha}_{\beta}g_2|||_{\Phi_\gamma} dx
 =\sum_{\begin{subarray}{l}
\alpha_1+\alpha_2=\alpha
\\
\beta_1+\beta_2\le\beta
\end{subarray}}
L^{\alpha_1,\alpha_2}_{\beta_1,\beta_2}.
\end{align*}
We have, for $|\alpha_1|<N/2$
\begin{align*}
L^{\alpha_1,\alpha_2}_{\beta_1,\beta_2}\lesssim&
 \|\p^{\alpha_1}\cA\|_{L^\infty(\RR^3)} \
\int_{\RR^3}||| \tilde W^{\ell-|\beta_2|}  \p^{\alpha_2}_{\beta_2}g_2|||_{\Phi_\gamma} \
 ||| \tilde W_{\ell-|\beta|} \p^\alpha_{\beta}g_2|||_{\Phi_\gamma} dx
 \\ &
 \lesssim
 \|\p^{\alpha_1}\cA\|_{H^{3/2+\epsilon}(\RR^3)}\ \|g_2\|_{\widetilde{\cB}^N_\ell(\RR^6)}^2
 \lesssim
\|\cA\|_{H^N(\RR^3)}\ \|g_2\|_{\widetilde{\cB}^N_\ell(\RR^6)}^2;
\end{align*}
while  for $|\alpha_2|\le N/2$
\begin{align*}
L^{\alpha_1,\alpha_2}_{\beta_1,\beta_2}\lesssim&
 \| \ ||| \tilde W^{\ell-|\beta_2|}  \p^{\alpha_2}_{\beta_2}g_2|||_{\Phi_\gamma} \ \|_{L^\infty(\RR^3)} \
\int_{\RR^3}
|\p^{\alpha_1}\cA| \ \
 ||| \tilde W_{\ell-|\beta|} \p^\alpha_{\beta}g_2|||_{\Phi_\gamma} dx
 \\
 \lesssim &
 \| \ ||| \tilde W^{\ell-|\beta_2|}  \p^{\alpha_2}_{\beta_2}g_2|||_{\Phi_\gamma} \ \|_{H^{3/2+\epsilon}(\RR^3)}
\|\p^{\alpha_1}\cA\|_{L^2(\RR^3)}\ \|g_2\|_{\widetilde{\cB}^N_\ell(\RR^6)}
\\
\lesssim &
\|\cA\|_{H^N(\RR^3)}\ \|g_2\|_{\widetilde{\cB}^N_\ell(\RR^6)}^2.
\end{align*}
Consequently,
\[
|K_{112}|\lesssim \cE_{N,l}^{1/2}\cD_{N.\ell}.
\]

\noindent{\bf (3) Estimation on $K_{121}$:}
As for $K_{112}$, we get
\begin{align*}
 |K_{121}|\le&
\sum_{\begin{subarray}{l}
\alpha_1+\alpha_2=\alpha
\\
\beta_1+\beta_2\le\beta
\end{subarray}}
\int_{\RR^3}
 |\p^{\alpha_2}\cA| \
||| \tilde W_{\ell-|\beta|}  \p^{\alpha_1}_{\beta_1}g_2|||_{\Phi_\gamma} \
 ||| \tilde W_{\ell-|\beta|} \p^\alpha_{\beta}g_2|||_{\Phi_\gamma} dx
 \\
\le &
 \cE_{N,\ell}^{1/2}\cD_{N,\ell}.
\end{align*}

\noindent{\bf (4) Estimation on $K_{122}$:}
We shall re-use \eqref{part2-WpGamma} in the form,
\begin{align*}
|K_{122}|=&
|(\tilde W_{\ell-|\beta|}\p^{\alpha}_{\beta} \Gamma(g_2,\,g_2),
\tilde W_{\ell-|\beta|}\p^\alpha_{\beta} g_2)_{L^2(\RR^6)}|
\\&
\lesssim
\sum_{\alpha_1+\alpha_2=\alpha}
|(\tilde W_{\ell-|\beta|}\p_{\beta} \Gamma(\p^{\alpha_1}g_2,\,\p^{\alpha_2}g_2),
\tilde W_{\ell-|\beta|}\p^\alpha_{\beta} g_2)_{L^2(\RR^6)}|
\\&
\lesssim
\sum_{\begin{subarray}{l}\alpha_1+\alpha_2=\alpha
\\
\beta_1+\beta_2\le \beta
\end{subarray}
}
\int_{\RR^3}\Big(
 \| \tilde W_{\ell-|\beta|} \p^{\alpha_1}_{\beta_1}g_2\|_{L^2_{s +\gamma /2}(\RR^3)}
||| \tilde W_{\ell-|\beta|} \p^{\alpha_2}_{\beta_2}g_2|||_{\Phi_\gamma}
\\&\hspace{2cm}
+||| \tilde W_{\ell-|\beta|} \p^{\alpha_1}_{\beta_1}g_2|||_{\Phi_\gamma}
\ \| \tilde W_{\ell-|\beta|} \p^{\alpha_2}_{\beta_2}g_2\|_{L^2_{s +\gamma /2}(\RR^3)}
 \Big) ||| \tilde W_{\ell-|\beta|} \p_{\beta}g_2|||_{\Phi_\gamma}dx
\\&
\lesssim
\sum_{\begin{subarray}{l}\alpha_1+\alpha_2=\alpha
\\
\beta_1+\beta_2\le \beta
\end{subarray}
}
\int_{\RR^3}\Big(
 \| \tilde W^{\ell-|\beta_1|} \p^{\alpha_1}_{\beta_1}g_2\|_{L^2(\RR^3)}
||| \tilde W^{\ell-|\beta_2|} \p^{\alpha_2}_{\beta_2}g_2|||_{\Phi_\gamma}
\\&\hspace{2cm}
+||| \tilde W^{\ell-|\beta_1|} \p^{\alpha_1}_{\beta_1}g_2|||_{\Phi_\gamma} \
\| \tilde W^{\ell-|\beta_2|} \p^{\alpha_2}_{\beta_2}g_2\|_{L^2(\RR^3)}
 \Big) ||| \tilde W_{\ell-|\beta|} \p_{\beta}g_2|||_{\Phi_\gamma}dx.
 \end{align*}
 Suppose $|\alpha_1|\le N-2$. Then,
\begin{align*}
\int_{\RR^3}&
 \| \tilde W^{\ell-|\beta_1|} \p^{\alpha_1}_{\beta_1}g_2\|_{L^2(\RR^3)}
||| \tilde W^{\ell-|\beta_2|} \p^{\alpha_2}_{\beta_2}g_2|||_{\Phi_\gamma}
\ ||| \tilde W_{\ell-|\beta|} \p_{\beta}g_2|||_{\Phi_\gamma}dx
\\&\lesssim
\| \tilde W^{\ell-|\beta_1|} \p^{\alpha_1}_{\beta_1}g_2\|_{L^\infty(\RR^3;
L^2(\RR^3))}\int_{\RR^3}||| \tilde W^{\ell-|\beta_2|} \p^{\alpha_2}_{\beta_2}g_2|||_{\Phi_\gamma}
\ ||| \tilde W_{\ell-|\beta|} \p_{\beta}g_2|||_{\Phi_\gamma}dx
\\&\lesssim
\|g_2\|_{\tilde{\cH}^N(\RR^6)}\|g_2\|_{\tilde{\cB}^N(\RR^6)}^2 .
\end{align*}
The case $|\alpha_2|\le 1$ can be computed similarly,
we finally conclude
\begin{align*}
|K_{122}|
 \lesssim \|g_2\|_{\widetilde{\cH}^N_\ell(\RR^6)}\|g_2\|_{\widetilde{\cB}^N_\ell(\RR^6)}^2\lesssim
 \cE_{N,\ell}^{1/2}\cD_{N,\ell},
\end{align*}
and therefore the estimate \eqref{part2-K1} holds.

{\bf (II) Estimation of $K_2$:}
On the other hand, we have, for
$|\alpha|+|\beta|\le N, \beta\not=0$,
 \begin{align*}
 |K_2|=&\big|([v\cdot\nabla_x, \tilde W_{\ell-|\beta|}\p^\alpha_{\beta}]g_2, \tilde W_{\ell-|\beta|}\p^\alpha_{\beta} g_2)_{L^2(\RR^6)}
 \big|
 \\&
  \lesssim \|\tilde W^{\ell-(|\beta|-1)-1/2}\p_x^{\alpha+1}\p_v^{\beta-1}g_2\|_{L^2(\RR^6)}
\|\tilde W^{\ell-|\beta|-1/2}\p^\alpha_{\beta} g_2\|_{L^2(\RR^6)}
\\
&\lesssim
C_\delta\|\tilde W^{\ell-(|\beta|-1)-1/2}\p_x^{\alpha+1}\p_v^{\beta-1}g_2\|_{L^2(\RR^6)}^2+\delta
\|\tilde W^{\ell-|\beta|-1/2}\p^\alpha_{\beta} g_2\|_{L^2(\RR^6)}^2.
\\
&=
C_\delta\|\tilde W^{\ell-(|\beta|-1)}\p_x^{\alpha+1}\p_v^{\beta-1}g_2\|_{L^2_{s+\gamma/2}(\RR^6)}^2
+\delta
\|\tilde W_{\ell-|\beta|}\p^\alpha_{\beta} g_2\|_{L^2_{s+\gamma/2}(\RR^6)}^2
 \\
  &\lesssim
  C_\delta\| \ |||\tilde W^{\ell-(|\beta|-1)}\p_x^{\alpha+1}\p_v^{\beta-1}g_2|||_{\Phi_\gamma} \ \|_{L^2(\RR^3_x)}^2
 +\delta
\|g_2\|_{\widetilde{\cB}^N_\ell(\RR^6)}^2.
 \end{align*}

{\bf (III) Estimation of $K_3$:}
Again we assume $\beta\not=0$ or $|\alpha|\le N-1$.
\begin{align*}
|K_3|=&|(\tilde W_{\ell-|\beta|} \p^\alpha_{\beta} [\pP,\, v\cdot \nabla]g,\,
\tilde W_{\ell-|\beta|}\p^\alpha_{\beta} g_2)_{L^2(\RR^6)}
\\&
\lesssim |(\p_\beta \Big\{\tilde W^{2(\ell-|\beta|)} \p^\alpha_{\beta} [\pP,\, v\cdot \nabla]g\Big\},
\p^\alpha g_2)_{L^2(\RR^6)}|
\\&
\lesssim
\|\nabla_x\p^\alpha g_1\|_{L^2(\RR^3_x)}^2+
 \delta
\|\tilde W^{-1}\nabla_x\p^\alpha g_2\|_{L^2(\RR^6_{x, v})}^2+\delta
\|\tilde W^{-1}\p^\alpha  g_2\|_{L^2(\RR^6_{x, v})}^2
\\&\lesssim
\|\nabla_x \p^\alpha\cA\|_{L^2(\RR^3_x)}^2+\delta ||\nabla_x\p^\alpha g_2||_{\cB^0_0}
+\delta
||\p^\alpha g_2||_{\cB^0_0}
\\&\lesssim
\|\nabla_x \cA\|_{H^{N-1}(\RR^3_x)}^2
+\delta\|g_2\|_{\cB^N(\RR^6)}^2,
\end{align*}
where $\cD_N$ is the dissipation integral with only $x$ derivatives.

{\bf (IV) Estimation of $K_4$: }
The main ingredients of the estimation are the commutator estimates  $I$ and $II$
established  in  the proof of Proposition \ref{part2-weight-radja-proposition-8.1}
that  are valid for $\gamma >-3$. We re-produce them here.
\begin{align*}
|I|&=|([\tilde W_l,  \cL_1] g , \tilde W_l g)_{L^2(\RR^3)}|
\lesssim
\|\tilde W_l g \|_{L^2_{\gamma/2}(\RR^3)}^2.
\\
|II|&=  |(\tilde W_l [\partial_\beta, \cL_1 ]g, \tilde W_l \partial_\beta g)_{L^2(\RR^3)}
\\&
\lesssim \sum_{\beta_1 +\beta_2 +\beta_3=\beta, \ \beta_2 \neq 0}
| (\tilde W_l\cT(\p_{\beta_1}\mu^{1/2}, \partial_{\beta_2} g, \p_{\beta_3}\mu^{1/2}) ; \tilde W_l \partial_\beta g)_{L^2(\RR^3)}|
\\&\lesssim
\Big( \sum_{\beta_1 +\beta_2 =\beta, \ \beta_2 \neq 0} ||| \tilde W_l \partial_{\beta_1} g|||_{\Phi_\gamma} \Big) ||| \tilde W_l \partial_\beta g |||_{\Phi_\gamma} .
\end{align*}
We also need the interpolation inequality
\begin{equation}\label{part2-interpolation}
\|\tilde W^l\p_\beta h\|_{L^2_{\gamma/2}}\lesssim
C_\delta\|\tilde W^l\p_{|\beta|-1} h\|_{L^2_{\gamma/2}}+\delta
\|\tilde W^l\p_\beta h\|_{H^s_{\gamma/2}}
\lesssim
C_\delta\|\tilde W^l\p_{\beta-1} h\|_{L^2}+\delta
|||\tilde W^l\p_\beta h|||_{\Phi_\gamma}.
\end{equation}

We shall prove
\begin{equation}
\label{part2-K4}
|K_4|\lesssim
C_\delta\|g_2\|_{\widetilde{\cB}^{N-1}_\ell(\RR^6)}^2+\delta
\|g_2\|_{\widetilde{\cB}^{N}_\ell(\RR^6)}^2.
\end{equation}
To this end, first, notice that
\begin{align*}
K_4&=([\tilde W_{\ell-|\beta|}\partial_\beta, \, \cL_1]\p^\alpha g_2, \tilde W_{\ell-|\beta|}\partial^\alpha_\beta g_2)_{L^2(\RR^6)}
\\&=([\tilde W_{\ell-|\beta|} , \, \cL_1]\partial^\alpha_\beta g_2, \tilde W_{\ell-|\beta|}\partial^\alpha_\beta g_2)_{L^2(\RR^6)}
\\&\hspace{1cm}+(\tilde W_{\ell-|\beta|}[\partial_\beta, \, \cL_1]\p^\alpha g_2, \tilde W_{\ell-|\beta|}
\partial^\alpha_\beta g_2)_{L^2(\RR^6)}=K_{41}+K_{42}.
\end{align*}
Then, by virtue of the estimate for $|I|$ and the interpolation inequality
\eqref{part2-interpolation},
\begin{align*}
|K_{41}|&\lesssim
\|\tilde W_{\ell-|\beta|}\p^\alpha_\beta g_2 \|_{L^2_{\gamma/2}(\RR^6)}^2
\\&\lesssim
C_\delta
\|\tilde W_{\ell-|\beta|}\p^\alpha_{\beta-1} g_2 \|_{L^2(\RR^3_x;L^2_{\gamma/2}(\RR^3_v))}^2
+\delta
\|\tilde W_{\ell-|\beta|}\p^\alpha_{\beta} g_2 \|_{L^2(\RR^3_x; H^s_{\gamma/2}(\RR^3_v))}^2
\\&
\lesssim
C_\delta
\|\tilde W^{\ell-(|\beta|-1)}\p^\alpha_{\beta-1} g_2 \|
_{L^2(\RR^3_x;L^2_{s+\gamma/2}(\RR^3))}^2
+\delta \|\
|||\tilde W_{\ell-|\beta|}\p^\alpha_{\beta} g_2|||_{\Phi_\gamma} \ \|_{L^2(\RR^3_x)}^2
\\&\lesssim
C_\delta\|g_2\|_{\widetilde{\cB}^{N-1}_\ell(\RR^6)}^2+\delta
\|g_2\|_{\widetilde{\cB}^{N}_\ell(\RR^6)}^2.
\end{align*}
On the other hand, the estimate for $|II|$ leads to
\begin{align*}
|K_{42}|\lesssim&
\sum_{\beta_1 +\beta_2 =\beta, \ \beta_2 \neq 0}\int_{\RR^3}
 ||| \tilde W^{\ell-|\beta_1|} \partial^\alpha_{\beta_1} g_2|||_{\Phi_\gamma}  \
||| \tilde W_{\ell-|\beta|} \partial^\alpha_\beta g_2 ||| dx
\\&
\lesssim \|g_2\|_{\widetilde{\cB}^{N-1}_\ell(\RR^6)}\|g_2\|_{\widetilde{\cB}^{N}_\ell(\RR^6)}
\\&\lesssim
C_\delta\|g_2\|_{\widetilde{\cB}^{N-1}_\ell(\RR^6)}^2+\delta
\|g_2\|_{\widetilde{\cB}^{N}_\ell(\RR^6)}^2.
\end{align*}
This completes the proof of \eqref{part2-K4}.

{\bf (V) Estimation of $K_5$:}
 Further,
by Proposition \ref{part2-prop2.10}  that holds  for $\gamma >-3$, we can
proceed as in the computation for $K_{41}$ to obtain
\begin{align*}
|K_{5}|&\lesssim \|\mu^{\min(a,b)}\tilde W_{\ell-|\beta|}\p^\alpha_{\beta} g_2\|_{L^2(\RR^6)}^2
\lesssim \|\tilde W_{\ell-|\beta|}\p^\alpha_\beta g_2 \|_{L^2_{\gamma/2}(\RR^6)}^2
\\&
\lesssim
C_\delta
\|\tilde W^{\ell-(|\beta|-1)}\p^\alpha_{\beta-1} g_2 \|_{L^2(\RR^6)}^2
+\delta \|\
|||\tilde W_{\ell-|\beta|}\p^\alpha_{\beta} g_2|||_{\Phi_\gamma} \ \|_{L^2(\RR^3_x)}^2
\\&\lesssim
C_\delta\|g_2\|_{\widetilde{\cB}^{N-1}_\ell(\RR^6)}^2+\delta
\|g_2\|_{\widetilde{\cB}^{N}_\ell(\RR^6)}^2.
\end{align*}

\noindent
{\bf Conclusion:} Plug the above estimates into \eqref{part2-energy1+} to deduce, for $|\alpha+\beta|\le N, |\beta|\ge 1$,
\begin{align}\label{part2-weighted}
\frac{d}{dt}\Big(
\|\p^\alpha_{\beta}&g_2\|_{L^2_{\ell-|\beta|}(\RR^6)}+
(\p^\alpha r,\nabla_x \p^\alpha (a, -b,c))_{L^2(\RR^3_x)}
+(\p^\alpha b, \nabla_x \p^\alpha a)_{L^2(\RR^3_x)}
\Big)
\\&\notag\hspace{1cm}
+\frac 12\| \
||| \tilde W_{\ell-|\beta|}\p^\alpha_{\beta} g_2 |||_{\Phi_\gamma} \ \|_{L^2(\RR^3_x)}^2
\\&\lesssim\notag
 || \p^\alpha g_2 ||^2_{ L^2_{s+\gamma/2}(\RR^6)}+
\cE_{N,\ell}^{1/2}\cD_{N,\ell}
\\&\notag+\| \ |||\tilde W^{\ell-(|\beta|-1)}\p_x^{\alpha+1}\p_v^{\beta-1}g_2|||_{\Phi_\gamma} \ \|_{L^2(\RR^3_x)}^2
+\delta
\|g_2\|_{\widetilde{\cB}^N_\ell(\RR^6)}^2
+\cD_N
\\&\notag
+\|g_2\|_{\widetilde{\cB}^{N-1}_\ell(\RR^6)}^2+\delta
\|g_2\|_{\widetilde{\cB}^{N}_\ell(\RR^6)}^2.
\end{align}

We can then make a suitable linear combination of \eqref{part2-pcA},\eqref{part2-micro}, and \eqref{part2-weighted} to
deduce the following energy estimate.
\begin{prop} {\rm \bf (Global Energy Estimate with Weight) \ }
Under the assumptions of Theorem \ref{part2-theo1.2}, for $N\geq 6, \ell\ge N$,
\begin{align*}
\frac{d}{dt}\cE_{N,\ell}+\cD_{N,\ell}\lesssim
\cE_{N,\ell}^{1/2}\cD_{N,\ell}
\end{align*}
holds as far as $g$ exists.
\end{prop}
We can now conclude in a standard way that the global classical
solutions exist for small initial data in the weighted space
$\widetilde{\cH}^N_\ell$, and this completes
 the proof of Theorem \ref{part2-theo1.2}.


\section{Appendix}

Let us recall that $\Phi_\gamma =|v|^\gamma$.
Let $\phi$ be a smooth, positive radial function that takes
value $1$ for small value and $0$ for large value of $| v|$.
Set $\Phi_c (v) = | v|^\gamma \phi (v)$.
We shall show the following

\begin{lemm} Assume $ \gamma >-3$. Then, for all integer $k$, one has
$$
 | D^k \hat \Phi_c (\xi ) | \lesssim <\xi>^{-3-\gamma -k}, \mbox{ for all } \xi \in \RR^3 .
 $$
\end{lemm}

\begin{proof}
Since $\Phi_c$ is bounded and compactly supported, clearly, for any integer $k$, $|D^k \hat \Phi_c(\xi ) | \lesssim 1$ so we can only consider
the case when  $| \xi | >>1$.

We first consider the case : $-3 < \gamma <0$.
We use the fact that the Fourier transform of $| v|^\gamma $ is (up to constant) $|\xi|^{-3-\gamma}$, see Page 243 of \cite{taylor}.

Let $\psi =\psi (\xi )$ a smooth positive function supported  on $| \xi| \le 1$, and is equal to $1$ for  on $| \xi| \le 1/2$.
Write
$$
\hat \Phi_c  (\xi ) = \int_\eta {1\over{| \xi -\eta |^{3+\gamma}}} \hat \phi (\eta ) d\eta = J_1 +J_2,
$$
where
$$
J_1=  \int_\eta {1\over{| \xi -\eta |^{3+\gamma}}} \psi (\xi -\eta )\hat \phi (\eta ) d\eta,\quad \mbox{ and } \quad J_2= \int_\eta {1\over{| \xi -\eta |^{3+\gamma}}}[1- \psi (\xi -\eta )] \hat \phi (\eta ) d\eta .
$$
For $J_1$, the support is on $| \xi -\eta| \le 1$. This means that $| \eta | \ge | \xi | -1 \ge c| \xi|$, for some constant $c$ and because we have assumed $| \xi| >>1$. Then, we can use the decay of $\hat \phi$ to get, for any $m$ positive
$$
J_1 \lesssim \int_{\eta , | \xi -\eta | \le 1} {1\over{| \xi -\eta |^{3+\gamma}}} <\eta>^{-m} d\eta \lesssim <\xi>^{-m} .
$$
For $J_2$, the integration is over $| \xi -\eta | \ge 1/2$. So we can replace $| \xi -\eta|$ by $<\xi-\eta>$ to get
$$
J_2= \int_{\eta , | \xi -\eta | \ge 1/2} {1\over{< \xi -\eta >^{3+\gamma}}} <\eta>^{-m} d\eta .
$$
Choose $m= M + 3 + \gamma$ with $M$ large enough. Then
$$
J_2= \int_{\eta , | \xi -\eta | \ge 1/2} {1\over{< \xi -\eta >^{3+\gamma}}} <\eta>^{-M} <\eta >^{-3-\gamma} d\eta
$$
$$
\lesssim  <\xi >^{-3-\gamma} \int_{\eta , | \xi -\eta | \ge 1/2}  <\eta>^{-M}  d\eta .
$$
Thus, we have shown that $| \hat \Phi_c (\xi ) | \lesssim < \xi>^{-3-\gamma}$, which proves the Lemma in the case when $k=0$.

 This proof works well for derivatives.
For example, consider the case when $k=1$. First note that
$$
\nabla \hat \Phi_c  (\xi ) = \int_\eta {1\over{| \xi -\eta |^{3+\gamma}}} \nabla\hat \phi (\eta ) d\eta = K_1 +K_2,
$$
where
$$
K_1=  \int_\eta {1\over{| \xi -\eta |^{3+\gamma}}} \psi (\xi -\eta )\nabla \hat \phi (\eta ) d\eta \mbox{ and } K_2= \int_\eta {1\over{| \xi -\eta |^{3+\gamma}}}[1- \psi (\xi -\eta )] \nabla \hat \phi (\eta ) d\eta .
$$
$K_1$ is estimated directly as for $J_1$, with all the decay.

For $K_2$, integration by parts gives
$$
K_2= -\int_\eta \nabla [{1\over{| \xi -\eta |^{3+\gamma}}}][1- \psi (\xi -\eta )]  \hat \phi (\eta ) d\eta
$$
$$
+ \int_\eta {1\over{| \xi -\eta |^{3+\gamma}}}\nabla \psi (\xi -\eta ) \nabla \hat \phi (\eta ) d\eta.
$$
Here,  the first term has the good decay in ${-3-\gamma -1}$,
 while the second one has all the decay.

We now consider the case $\gamma \geq 0$. Of course, for $\gamma =0$, the result is clear, because then $\hat \Phi_c$ is in $\cS$.

For $2>\gamma >0$ we have
\begin{align*}
|v|^\gamma \varphi(|v|) &=
\int \Big( - \Delta_{\xi} e^{iv\cdot \xi} \Big ) {\cF}_{v \rightarrow \xi}
\Big( |v|^{\gamma-2} \varphi(v)\Big)(\xi) d\xi/(2\pi^3)\\
&= -\int e^{iv\cdot \xi}         \Delta_{\xi}  {\cF}_{v \rightarrow \xi}
\Big( |v|^{\gamma-2} \varphi(v)\Big)(\xi) d\xi/(2\pi^3)\,,
\end{align*}
which gives
\[
\Big|\partial_\xi^\alpha \cF\Big(|v|^\gamma \varphi(|v|)\Big)(\xi)\Big|
= \Big|\partial_\xi^\alpha \Delta_\xi  \cF\Big(|v|^{\gamma-2} \varphi(|v|)\Big)(\xi)\Big|
\leq C_\alpha \la \xi \ra^{-3-\gamma-|\alpha|}\,,
\]
by using the previous negative case since $\gamma -2 <0$. The remaining cases are similar and this completes the proof of the lemma.
\end{proof}

\noindent
{\bf Acknowledgements :}
The research of the first author was supported in part by the Zhiyuan foundation and Shanghai Jiao Tong University. The research of the second author was
supported by  Grant-in-Aid for Scientific Research No.22540187,
Japan Society of the Promotion of Science.
The last author's research was supported by the General Research
Fund of Hong Kong,
 CityU No.103109, and the Lou Jia Shan Scholarship programme of
Wuhan University. The authors would like to
thank the financial supports from City University of Hong Kong, Kyoto
University, Rouen University and Wuhan University for their visits.

\end{document}